\documentclass[11pt]{article}
\usepackage{amsmath}
\usepackage{amsfonts}
\usepackage{amsthm}
\usepackage[thinlines]{easymat}
\usepackage{graphicx}
\usepackage{overpic}
\usepackage{a4wide}
\usepackage{amssymb}
\usepackage{rotating}
\usepackage{color}

\usepackage{comment}

\numberwithin{equation}{section}

\newtheorem{theorem}{Theorem}[section]
\newtheorem{proposition}[theorem]{Proposition}
\newtheorem{lemma}[theorem]{Lemma}
\newtheorem{corollary}[theorem]{Corollary}
\newtheorem{Definition}[theorem]{Definition}

\newtheorem{Remark}[theorem]{Remark}
\newenvironment{remark}{\begin{Remark}\rm}{\end{Remark}}

\newcommand{\C}{\mathbb{C}}

\newcommand{\N}{\mathbb{N}}
\newcommand{\R}{\mathbb{R}}

\newcommand{\RR}{\mathcal{R}}

\renewcommand{\Re}{{\rm Re} \,}
\renewcommand{\Im}{{\rm Im} \,}

\def\diag{\mathop{\mathrm{diag}}\nolimits}
\def\Ai{\mathop{\mathrm{Ai}}\nolimits}

\def\det{\mathop{\mathrm{det}}\nolimits}
\def\supp{\mathop{\mathrm{supp}}\nolimits}
\def\Tr{\mathop{\mathrm{Tr}}\nolimits}

\begin{document}
\title{Non-intersecting squared Bessel paths
and multiple orthogonal polynomials for modified Bessel weights}
\author{A.B.J.\ Kuijlaars, A.\ Mart\'{\i}nez-Finkelshtein, and F.\ Wielonsky}
\date{}

\maketitle

\begin{abstract}
We study a model of $n$ non-intersecting squared Bessel processes
in the confluent case: all paths start at time $t = 0$ at the
same positive value $x = a$, remain positive, and are conditioned to end at time $t
= T$ at $x = 0$. In the limit $n \to \infty$, after appropriate
rescaling, the paths fill out a region in the $tx$-plane that we describe
explicitly.
In particular, the paths initially stay away from the hard edge at $x
= 0$, but at a certain critical time $t^*$ the smallest paths hit the
hard edge and from then on are stuck to it. For $t \neq t^*$
we obtain the usual scaling limits from random matrix theory, namely
the sine, Airy, and Bessel kernels.
 A key fact is that the
positions of the paths at any time $t$ constitute a multiple orthogonal polynomial
ensemble, corresponding to a system of two modified
Bessel-type weights. As a consequence, there is a $3 \times 3$ matrix valued
Riemann-Hilbert problem characterizing this model, that we analyze in the large $n$ limit
using the Deift-Zhou steepest descent method. There are some novel
ingredients in the Riemann-Hilbert analysis that are of independent interest.
\end{abstract}

\section{Introduction}

Determinantal point processes are of considerable current interest in
probability theory and mathematical physics, since they arise
naturally in random matrix theory, non-intersecting paths, certain combinatorial
and stochastic growth models and representation theory of large groups,
see e.g.\ Deift \cite{De}, Johansson \cite{Jo1},
Katori and Tanemura \cite{KT1},
Borodin and Olshanski \cite{BO}, and many other papers cited therein.
See also the surveys of Soshnikov \cite{So}, K\"onig \cite{Ko}, Hough et al.\ \cite{HKPV},
and Johansson \cite{Jo2}.

A determinantal point process  is characterized by a
correlation kernel $K$  such that for every $m$ the
$m$-point correlation function (or joint intensities) takes
the determinantal form
\[  \det \left[ K(x_j,x_k)\right]_{j,k=1,\ldots,m} \]
We will only consider determinantal point processes on $\mathbb R$.

As pointed out by Borodin \cite{Bo} certain determinantal point processes
arise as biorthogonal ensembles, i.e., joint probability density functions
on $\mathbb R^n$ of the form
\begin{equation} \label{biorthogonal}
    \mathcal P(x_1, \ldots, x_n) = \frac{1}{Z_n} \det[ f_j(x_k)]_{j,k=1, \ldots,n} \det[g_j(x_k)]_{j,k=1, \ldots, n}
    \end{equation}
for certain given functions $f_1, \ldots, f_n$, and $g_1, \ldots, g_n$.
The correlation kernel is then given by
\begin{equation} \label{biorthogonalkernel}
    K(x,y) = \sum_{j=1}^n \phi_j(x) \psi_j(y)
    \end{equation}
where $\phi_j$, $\psi_j$, $j=1,\ldots, n$ are such that
\[ \textrm{span}\{\phi_1, \ldots, \phi_n\} = \textrm{span}\{f_1, \ldots, f_n\}, \qquad
   \textrm{span}\{\psi_1, \ldots, \psi_n\} = \textrm{span}\{g_1, \ldots, g_n\} \]
and they have the biorthogonality property
\[ \int_{\mathbb R} \phi_j(x) \psi_k(x) \, dx = \delta_{j,k}. \]

The joint probability distribution function for the eigenvalues of unitary
invariant ensembles of random Hermitian matrices
$ (1/\widetilde{Z}_n) e^{- \Tr V(M)} dM$
has the form \eqref{biorthogonal} where
\begin{equation} \label{fjgjfromRMT}
    f_j(x) = g_j(x) = x^{j-1} e^{-\frac{1}{2} V(x)}, \qquad j=1,2, \ldots, n.
    \end{equation}
Orthogonalizing the functions \eqref{fjgjfromRMT} leads to
\[ \phi_j(x) = \psi_j(x) = p_{j-1}(x) e^{-\frac{1}{2} V(x)}, \qquad j=1,2, \ldots, n,
\]
where $p_{j-1}$ is the orthonormal polynomial of degree $j-1$ with respect to
the weight $e^{-V(x)}$ on $\mathbb R$. The kernel \eqref{biorthogonalkernel}
is then the orthogonal polynomial kernel, also called the Christoffel-Darboux kernel
because of the Christoffel-Darboux formula for orthogonal polynomials,
and the ensemble is called an orthogonal polynomial ensemble \cite{Ko}.

Other examples for biorthogonal ensembles arise in the context
of non-intersecting paths as follows.
Consider a one-dimensional diffusion process $X(t)$ (i.e., a strong Markov process on $\mathbb R$
with continuous sample paths) with transition probability functions
$p_t(x,y)$, $t > 0$, $x,y \in \mathbb R$. Take  $n$ independent copies
$X_j(t)$, $j=1, \ldots, n$, conditioned so that
\begin{itemize}
\item $X_j(0) = a_j$, $X_j(T) = b_j$, where $T > 0$, and $a_1 < a_2 < \cdots < a_n$,
    $b_1 < b_2 < \cdots < b_n$ are given values,
\item the paths do not intersect for $0 < t < T$.
\end{itemize}
It then follows from a remarkable theorem of Karlin and McGregor \cite{KM}
that the positions of the paths at any given time $t \in (0,T)$ have
the joint probability density \eqref{biorthogonal} with functions
\begin{equation*} 
    f_j(x) = p_t(a_j,x), \qquad g_j(x) = p_{T-t}(x,b_j),
    \qquad j=1,\ldots, n.
\end{equation*}
[Properly speaking the joint probability density function is first defined
for ordered $n$-tuples $x_1 < x_2 < \cdots < x_n$ only. It is extended
in a symmetric way to all of $\mathbb R^n$.]

An important feature of determinantal point processes is that
they seem to have universal limits. By now, this is well-established
for the eigenvalue distributions of unitary random matrix ensembles.
Indeed if $K_n$ is the eigenvalue correlation kernel for
the random matrix ensemble (note the $n$-dependence of the potential)
\[ \frac{1}{\widetilde{Z}_n} e^{- n \Tr V(M)} dM \]
then we have under mild assumptions on $V$ that
\[ \lim_{n \to \infty} \frac{1}{n} K_n(x,x) =: \rho(x) \]
exists. In addition if $V$ is real analytic, and if $x^*$
is in the bulk of the spectrum (i.e., $\rho(x*) > 0$), then (see \cite{DKMVZ1})
\begin{equation} \label{sinekernel}
    \lim_{n \to \infty} \frac{1}{n \rho(x^*)}
        K_n\left(x^* + \frac{x}{n \rho(x^*)}, x^* + \frac{y}{n \rho(x^*)} \right)
    = \frac{\sin \pi(x-y)}{\pi(x-y)}.
\end{equation}
Universality of local eigenvalue statistics is expressed
by \eqref{sinekernel} in the sense that the sine kernel arises
as the limit regardless of $V$ and $x^*$. The universality \eqref{sinekernel}
is extended in many ways and (as its name suggests) under very mild assumptions (see the recent works \cite{L1, L2}).

The limit \eqref{sinekernel} does not hold at special points $x^*$ of
the spectrum where $\rho(x^*) =0$.
However it turns out that $K_n$ has scaling limits at such special points
that are determined by the macroscopic nature of $x^*$, and in that sense
they are again universal (see e.g.\ \cite{CK1, CK2, CKV, CV, DG}).

It is reasonable to expect that such universal limit results hold generically for
non-intersecting paths as well, although results are more sparse. For recent progress related to
discrete random walks,
random tilings and random matrices with external source see \cite{ABK3, BKMM, BaSu, BK1, BK2, BK4, OR, TW1}.

It is the aim of this paper to study a model of $n$ non-intersecting squared
Bessel processes in the limit $n \to \infty$. Recall that if $\{ \mathbf X(t):\, t\geq 0\}$ is
a $d$-dimensional Brownian motion, then the diffusion process
$$
R(t) = \|\mathbf X(t)\|_2 = \sqrt{ X_1(t)^2 + \dots + X_d(t)^2},\quad t\geq 0,
$$
is the \emph{Bessel process} with parameter $\alpha = \frac{d}{2}-1  $,   while $R^2(t)$ is the
\emph{squared Bessel process} usually denoted by BESQ$^d$ (see e.g. \cite[Ch.\ 7]{KT}, \cite{KO}).
These are an important family of diffusion processes which have applications in finance and other areas.
The well known Cox-Ingersoll-Ross (CIR) model in finance describing the  short term evolution of interest
rates or different models of the growth optimal portfolio (GOP) represent important examples of  squared
Bessel processes \cite{GJY, P}. The Bessel process $R(t) $ for $d=1$ reduces to the Brownian motion
reflected at the origin, while for $d=3$ it is connected with the Brownian motion absorbed at the
origin \cite{KT0, KT1}.

A system of $n$ particles performing BESQ$^d$ conditioned
never to collide with each other and conditioned to start and end
at the origin, can be realized as a process of
eigenvalues of a hermitian matrix-valued diffusion process, known
as the \emph{chiral} or \emph{Laguerre ensemble}, see e.g.
\cite{DF, KIK, KO, TW2} and below. In this paper we consider the case
where all particles start at the same positive value $a > 0$ and end at
$0$. Of particular interest here is the
interaction of the non-intersecting
paths with the hard edge at $0$. 
Due to the nature of the squared Bessel process, the paths starting at a positive value remain
positive, but they are conditioned to end at time $T$ at $0$.
After appropriate rescaling we will see that in the limit $n \to \infty$
the paths fill out a region in the $tx$-plane. The paths start at $t=0$
and initially stay away from the hard edge at $x=0$. At a certain critical time
the smallest paths hit the hard edge and from then on are stuck to it.
The phase transition at the critical time is a new feature of the present
model. It is a new soft-to-hard edge transition.

We are able to analyze the model in great detail since in the confluent case
the biorthogonal ensemble reduces
to a multiple orthogonal polynomial ensemble, as we will show in
Subsection \ref{section2} below.
The correlation kernel for the multiple orthogonal polynomial ensemble
is expressed via a  $3\times 3$ matrix-valued Riemann-Hilbert (RH) problem
\cite{BK1, DK1}.

We analyze the RH problem in the large $n$ limit using the Deift-Zhou
steepest descent method for RH problems \cite{DZ}.
There are some novel ingredients in our analysis which we feel are of
independent interest.
First of all, there is a first preliminary transformation which makes use of the explicit
structure of the RH jump matrix. It contains the modified Bessel functions
$I_{\alpha}$ and $I_{\alpha+1}$ and we use the explicit properties of Bessel functions.
A result of the first transformation is that a jump is created on the negative
real axis, see Section \ref{section3}.

The multiple orthogonal polynomials for modified Bessel functions were
studied before by Coussement and Van Assche \cite{CV1, CV2}. We use their results to make
an ansatz about an underlying Riemann surface that allows us to define
the second transformation in the steepest descent analysis in Section \ref{section4}.
The use of the Riemann surface is similar to what is done in \cite{BK2, KVAW}.
In the appendix we mention an alternative approach via equilibrium measures
and associated $g$-functions.
The further steps in the RH analysis follow the general scheme laid out
by Deift et al.\ \cite{DKMVZ1, DKMVZ2} in the context of orthogonal polynomials.
An important feature of
the present situation is that there is an unbounded cut along the negative
real axis and we have to deal with this technical issue in the construction
of the global parametrix in Section \ref{section6}. The construction of
the local parametrices at the hard edge $0$ also
presents a new technical issue, see Section \ref{section8}.

The main results of the paper are stated in the next section.

\section{Statement of results} \label{section2}

\subsection{Squared Bessel processes} \label{subsection21}

The transition probability density of a squared Bessel
process with parameter $\alpha > -1$ is given by (see \cite{BS, KO})
\begin{align} \label{pt}
   p_{t}^{\alpha} (x,y) & = \frac{1}{2t}\left(\frac{y}{x} \right)^{\alpha/2}
e^{- (x+y)/(2t)}
I_{\alpha}\left(\frac{\sqrt{xy}}{t} \right), &\quad x,y>0, \\
  p_{t}^{\alpha} (0,y) &= \frac{y^{\alpha}}{(2t)^{\alpha+1}\Gamma (\alpha+1) }e^{-y/(2t)}, & \quad  y>0,
 \label{ptx0}
\end{align}
where $I_{\alpha}$ denotes the modified Bessel function of the
first kind of order $\alpha$,
\begin{align} \label{Ialpha}
I_{\alpha} (z)=
    \sum_{k=0}^{\infty}\frac{(z/2)^{2k+\alpha}}{k! \, \Gamma(k+\alpha+1)}\, ;
\end{align}
see \cite[Section 9.6]{AS} for the main properties of the modified Bessel functions.
If $d = 2(\alpha + 1)$ is an integer, then the squared Bessel process can
be seen as the square of the distance to the origin of a $d$-dimensional
standard Brownian motion.

If the starting points $a_j$ and the endpoints $b_j$ are all different,
then (as explained in the introduction) the positions of the paths at a fixed
time $t \in (0,T)$
have a joint probability density
\begin{equation*}
\mathcal P_{n,t} (x_{1},\ldots,x_{n})=\frac{1}{Z_{n,t}}\det \left[p_{t}^{\alpha}
(a_{j},x_{k}) \right]_{j,k=1, \ldots,n}
\det \left[p_{T-t}^{\alpha}
(x_{j},b_{k}) \right]_{j,k=1, \ldots, n},
\end{equation*}
where $Z_{n,t}$ is the normalization constant such that
\begin{equation*}
\int_{( 0,\infty)^{n}} \mathcal P_{n,t} (x_{1},\ldots ,x_{n})dx_{1}\cdots dx_{n}=1.
\end{equation*}
This is a biorthogonal ensemble \eqref{biorthogonal} with functions
\[ f_j(x) = p_t^{\alpha}(a_j, x), \qquad g_j(x) = p_{T-t}^{\alpha}(x,b_j). \]

We are going to take the confluent limit $a_j \to a > 0$, and $b_j \to 0$.
Then the biorthogonal ensemble structure is preserved. In our first result
we identify the functions $f_j$ and $g_j$ for this situation.

\begin{proposition} \label{prop:coalesce}
    In the confluent limit $a_j \to a > 0$, $b_j \to 0$, $j=1, \ldots, n$,
    the positions of the non-intersecting squared Bessel paths at time $t \in (0,T)$ are a biorthogonal
    ensemble with functions
    \begin{align} \label{fjcoalesce1}
       f_{2j-1}(x) & = x^{j-1} p_t^{\alpha}(a,x), \qquad j =1, \ldots, n_1 := \lceil{n/2}\rceil, \\
       \label{fjcoalesce2}
       f_{2j}(x) & = x^{j-1} p_t^{\alpha+1}(a,x), \qquad j=1, \ldots, n_2 := n-n_1, \\
       \label{gjcoalesce}
       g_j(x) & = x^{j-1} e^{-\frac{x}{2(T-t)}}, \qquad j = 1, \ldots, n.
    \end{align}
\end{proposition}

\begin{proof}
In the confluent limit $a_j \to a$,  the linear space spanned by the functions
$y \mapsto p_{t}^{\alpha} (a_{j}, y)$, $j=1, \ldots, n$,
tends to the linear space spanned by
\begin{equation} \label{linearspanf1}
    y \mapsto \frac{\partial^{j-1}}{\partial x^{j-1}} p_{t}^{\alpha} (a, y),
    \qquad j= 1, \ldots, n.
\end{equation}
Using the differential relations
satisfied by the transition probabilities,  (see e.g.\ \cite{AS} or \cite{CV1,CV2}):
\begin{align*}
\frac{\partial}{\partial x}  p_{t}^{\alpha} (x,y) & = \frac{1}{2t}
(p_{t}^{\alpha+1} (x,y)-p_{t}^{\alpha} (x,y)),
\\
x\frac{\partial }{\partial x} p_{t}^{\alpha+1}(x,y) & = \frac{y}{2t} p_{t}^{\alpha} (x,y)
- \left(\frac{x}{2t}+\alpha+1\right)p_{t}^{\alpha+1} (x,y),
\end{align*}
it is easily shown inductively,
that the linear span of \eqref{linearspanf1} is the same as the linear space spanned by
\begin{equation*} 
    \begin{aligned}
    &y \mapsto y^{j-1} p_t^{\alpha}(a, y), \qquad j = 1, \ldots, n_1, \\
    &y \mapsto y^{j-1} p_t^{\alpha+1}(a,y), \qquad j= 1, \ldots, n_2
    \end{aligned}
\end{equation*}
which are exactly the functions in \eqref{fjcoalesce1}, \eqref{fjcoalesce2}.

Next, the linear space spanned by the functions
$x \mapsto p_{T-t}^{\alpha}(x,b_j)$, $j=1, \ldots, n$, tends in the
confluent limit $b_j \to 0$ to the linear space spanned by the functions
\begin{equation} \label{linearspang1}
x \mapsto \frac{\partial^{j-1}}{\partial y^{j-1}} \left. \left[ y^{-\alpha} p_{T-t}^{\alpha} (x,y) \right] \right|_{y=0}.
\end{equation}
By \eqref{pt} and \eqref{Ialpha} we have that
\[ y^{-\alpha} p_{T-t}^{\alpha}(x,y)
    = \frac{1}{(2(T-t))^{\alpha+1}} e^{-(x+y)/(2(T-t))}
        \sum_{k=0}^{\infty} \frac{(xy)^k }{k!\,  \Gamma(k+\alpha+1)(2(T-t))^{2k}} \]
which is an entire function in $y$ of the form
\[ y^{-\alpha} p_{T-t}^{\alpha}(x,y) = e^{-\frac{x}{2(T-t)}}
    \sum_{k=0}^{\infty} P_k(x) y^k \]
where each $P_k(x)$ is a polynomial in $x$ of exact degree $k$.
Thus the linear space spanned by  the functions \eqref{linearspang1}
is equal to the linear space spanned by the functions \eqref{gjcoalesce}, which
completes the proof of the proposition.
\end{proof}

\begin{remark}
In the next subsection we will see how Proposition \ref{prop:coalesce} allows us
to identify the ensemble of non-intersecting squared Bessel paths at any time $ t \in (0,T)$
as a multiple orthogonal polynomial ensemble.
For the transition probability density of the (non-squared) Bessel process the
calculations as in the proof of Proposition \ref{prop:coalesce} would not work
and in fact the positions of non-intersecting Bessel paths are not a multiple
orthogonal polynomial ensemble.
This is the reason why we concentrate on squared Bessel paths.

Of course, by taking square roots we can transplant results on non-intersecting
squared Bessel paths to non-intersecting Bessel paths, see Remark \ref{rem:Besselpaths}
below.
\end{remark}

\subsection{Multiple orthogonal polynomial ensemble} \label{subsection22}

According to Proposition \ref{prop:coalesce} the biorthogonal
ensemble in the confluent case is an example of what we call a multiple
orthogonal polynomial ensemble. A multiple orthogonal polynomial
ensemble in general may involve  an arbitrary number of weights and
an arbitrary multi-index, but we will discuss here the case of
weight functions $\widehat{w}_0, \widehat{w}_1, \widehat{w}_2$ and a
multi-index $(n_1,n_2)$ where $n_1 + n_2 =n$ and $n_1 = \lceil n/2
\rceil$. We take functions
\[ f_{2j-1}(x) = x^{j-1} \widehat{w}_1(x), \qquad f_{2j}(x) = x^{j-1} \widehat{w}_2(x) \]
and
\[ g_j(x) = x^{j-1} \widehat{w}_0(x), \qquad j =1, \ldots, n, \]
and we use these functions for a biorthogonal ensemble \eqref{biorthogonal}.
Note that in the squared Bessel case, we have by Proposition \ref{prop:coalesce}
and \eqref{pt} that (where we drop irrelevant constants)
\begin{align} \label{w1hat}
    \widehat{w}_1(x) & = x^{\alpha/2} e^{-\frac{x}{2t}}
    I_{\alpha}\left(\frac{\sqrt{ax}}{t}\right) \\
    \label{w2hat}
    \widehat{w}_2(x) & = x^{(\alpha+1)/2} e^{-\frac{x}{2t}}
    I_{\alpha+1}\left(\frac{\sqrt{ax}}{t}\right) \\
    \label{w0}
    \widehat{w}_0(x) & = e^{-\frac{x}{2(T-t)}}
    \end{align}
The biorthogonalization process leads to bases $\phi_j$, $\psi_j$,
$j=1, \ldots, n$, and we may take them so that
\[ \phi_j(x) = A_{j-1,1}(x) \widehat{w}_1(x) + A_{j-1,2}(x) \widehat{w}_2(x), \qquad
    \psi_j(x) = B_{j-1}(x) \widehat{w}_0(x), \]
where $A_{j-1,1}$ and $A_{j-1,2}$ are polynomials
of degrees $\lceil (j-1)/2 \rceil$ and $\lfloor (j-1)/2 \rfloor$, respectively,
and $B_{j-1}$ is a monic polynomial of degree $j-1$.
The biorthogonality property is
\begin{equation} \label{biorthprop}
    \int \left( A_{j,1} w_1(x) + A_{j,2} w_2(x) \right) B_k(x) \, dx = \delta_{j,k},
    \qquad j,k = 0, \ldots, n-1, \end{equation}
where we have put
\begin{equation} \label{w1w2}
    w_1(x) = \widehat{w}_0(x) \widehat{w}_1(x), \qquad w_2(x) = \widehat{w}_0(x) \widehat{w}_2(x).
\end{equation}
The polynomials $A_{j,1}$ and $A_{j,2}$ satisfying \eqref{biorthprop}
are called multiple orthogonal polynomials of type I and the polynomials
$B_k$ are called multiple orthogonal polynomials of type II.
The correlation kernel
\[ \widehat{K}_n(x,y) = \sum_{j=1}^n \phi_j(x) \psi_j(y)
    = \sum_{j=0}^{n-1}  \left( A_{j,1} \widehat{w}_1(x) + A_{j,2} \widehat{w}_2(x) \right)
        B_j(y) \widehat{w}_0(y)  \]
is called a multiple orthogonal polynomial kernel.
We will use the equivalent form (it is equivalent since it gives
rise to the same $m$-point correlation functions)
\begin{equation} \label{kernelKn}
    K_n(x,y) = \frac{\widehat{w}_0(x)}{\widehat{w}_0(y)} \widehat{K}_n(x,y) =
    \sum_{j=1}^n  \left( A_{j,1} w_1(x) + A_{j,2} w_2(x) \right)
        B_j(y)  \end{equation}
which has a characterization through a RH problem, \cite{BK1, DK1}
\[ K_n(x,y) = \frac{1}{2\pi i(x-y)} \begin{pmatrix} 0 & w_1(y) & w_2(y) \end{pmatrix}
    Y_+^{-1}(y) Y_+(x) \begin{pmatrix} 1 \\ 0 \\ 0 \end{pmatrix} \]
where $Y$ is a solution of the following $3 \times 3$ matrix valued RH problem.
\begin{enumerate}
\item $Y$ is analytic in $\mathbb{C} \setminus \mathbb R$.
\item On the real axis, $Y$ possesses continuous boundary values $Y_+$ (from the upper half plane)
and $Y_-$ (from the lower half plane), and
\begin{equation}  \label{Yjump}
   Y_+(x) = Y_-(x) \begin{pmatrix}
                   1 & w_{1} (x)& w_{2} (x)\\
                   0 & 1 & 0 \\
                   0 & 0 & 1
                   \end{pmatrix}, \qquad x \in \mathbb{R}.
\end{equation}
\item $Y(z)$ has the following behavior at infinity:
\begin{equation}  \label{Yasym}
    Y(z) = \left( I + \mathcal{O}\left(\frac1z\right) \right)\begin{pmatrix}
            z^{n} & 0 & 0 \\
            0 & z^{-n_1} & 0 \\
            0 & 0 & z^{-n_2}
            \end{pmatrix} , \quad z
        \to \infty, \quad z\in \mathbb{C} \setminus \R.
\end{equation}
\end{enumerate}
If the weight functions are not defined on the whole real line (as it will be
for the non-intersecting squared Bessel paths: the case of interest in this paper), we have to
supplement the RH problem
with appropriate conditions at the endpoints. The RH problem is an extension
of the RH problem for orthogonal polynomials of Fokas, Its, and Kitaev \cite{FIK}
to multiple orthogonal polynomials due to Van Assche et al.\ \cite{VAGK}.

In this paper we have by \eqref{w1hat}, \eqref{w2hat}, \eqref{w0}, and \eqref{w1w2}
\begin{equation}
\begin{aligned} \label{weights1}
w_{1}(x) &= x^{\alpha/2} \exp\left(-\frac{Tx}{2t(T-t)}\right) I_{\alpha}\left(\frac{\sqrt{ax}}{t}
\right), \\
w_{2}(x) &=  x^{(\alpha+1)/2} \exp\left(-\frac{Tx}{2t(T-t)}\right) I_{\alpha+1}\left(\frac{\sqrt{ax}}{t}
\right).
\end{aligned}
\end{equation}
The weights are defined on $[0,\infty)$ so that the jump condition \eqref{Yjump}
only holds for $x \in \mathbb R_+$, and the RH problem \eqref{Yjump}, \eqref{Yasym}
is supplemented with the following endpoint condition.
\begin{enumerate} \setcounter{enumi}{3}
\item $Y(z)$ has the following  behavior near the origin, as $z\to 0$,
$z\in \C\setminus\R_+$,
\begin{equation}  \label{Yedge}
    Y(z) = \mathcal{O}
        \begin{pmatrix}
            1 & h(z) & 1 \\
            1 & h(z) & 1 \\
            1 & h(z) & 1
            \end{pmatrix},
        \mbox{ with }
h(z)=\left\{
        \begin{array}{ccc}
|z|^{\alpha}, & \mbox{ if } & -1<\alpha<0,\\
\log|z|, & \mbox{ if } & \alpha=0,\\
1, & \mbox{ if } & 0<\alpha.\\
            \end{array}\right.
\end{equation}
The $\mathcal{O}$ condition in \eqref{Yedge} is to
be taken entrywise.
\end{enumerate}

\subsection{Multiple orthogonal polynomials for modified Bessel weights}
\label{subsection23}

We are fortunate that the multiple orthogonal polynomials associated
with the weights \eqref{weights1} were studied before by
Coussement and Van Assche \cite{CV1,CV2}. They showed that
all polynomials $A_{j,1}$, $A_{j,2}$ and $B_k$ exist so that
the above RH problem has a unique solution and
\[ \det Y(z) \equiv 1, \qquad \mbox{ for } z\in \mathbb C \setminus \mathbb R_+. \]
In addition $B_k$ satisfies
interesting recurrence and differential relations which they
were able to identify explicitly.

The type II multiple orthogonal polynomials $B_k$ satisfy a
four term recurrence relation
\[ x B_k(x) = B_{k+1}(x) + b_k B_k(x) + c_k B_{k-1}(x) + d_k B_{k-2}(x) \]
with recurrence coefficients that are obtained from
\cite[Theorem 9]{CV2} after appropriate rescaling and
identification of parameters
\begin{align*}
    b_k & = \frac{a(T-t)^2}{T^2} + \frac{2t(T-t)}{T}(2k+ \alpha + 1), \\
    c_k & = \frac{4a t(T-t)^3}{T^3} k + \frac{4t^2(T-t)^2}{T^2} k (k+\alpha), \\
    d_k & = \frac{4a t^2 (T-t)^4}{T^4} k(k-1).
    \end{align*}
In addition $y = B_n$ is a solution of the third order differential equation \cite[Theorem 11]{CV2}
\begin{multline}  \label{DEforMOP}
    xy''' +  \left(- \frac{Tx}{t(T-t)} + \alpha + 2\right) y'' \\
    + \left(\frac{T^2}{4t^2(T-t)^2}x + \frac{(n-\alpha-2)T}{2t(T-t)} - \frac{a}{4t^2}\right) y'
    - \frac{nT^2}{4t^2(T-t)^2} y = 0.
    \end{multline}

\subsection{Time scaling and large $n$ limit} \label{subsection24}

We want to analyze the kernel $K_{n}$ from \eqref{kernelKn} in the large $n$ limit.
To obtain interesting results, we  make the time variable depend on the number $n$ of
paths. Hence, we rescale the time in an appropriate way, namely we
replace the variables $t$ and $T$
\[ t \mapsto \frac{t}{2n}, \qquad T \mapsto \frac{1}{2n}, \]
so that $0 < t < 1$.
Thus, the system of
weights \eqref{weights1} now becomes $n$-dependent
\begin{equation}
\begin{aligned}
\label{weights2}
w_{1}(x) = w_{1,n}(x) & = x^{\alpha/2}\exp\left(-\frac{nx}{t(1-t)}\right)I_{\alpha}\left(\frac{2n\sqrt{ax}}{t}\right), \\
w_{2}(x) = w_{2,n}(x) & = x^{(\alpha+1)/2}\exp\left(-\frac{nx}{t(1-t)}\right)I_{\alpha+1}\left(\frac{2n\sqrt{ax}}{t}
\right).
\end{aligned}
\end{equation}
Alternatively, we could have performed space scaling, putting $T=1$ and
replacing the position variable $x$ with $2nx$ and the starting position $a$ with $2na$.

After the change of time parameters $t \mapsto t/(2n)$, $T \mapsto 1/(2n)$ the differential
equation \eqref{DEforMOP} turns into (with $x$ replaced by $z$)
\begin{multline} \label{DEforMOP2}
  zy'''(z) + \left((2+\alpha) - \frac{2nz}{t(1-t)}\right) y''(z) \\
  +
   \left(\frac{n^2z}{t^2(1-t)^2}  + \frac{n(n-\alpha-2)}{t(1-t)} - \frac{an^2}{t^2}\right)y'(z)
   - \frac{n^3}{t^2(1-t)^2} y (z)=0,
\end{multline}
Expressing \eqref{DEforMOP2} in terms of the scaled logarithmic derivative
$\zeta=y'/(ny)$ and keeping only
the dominant terms with respect to $n$ as $n \to \infty$, we arrive at the algebraic
equation for $\zeta = \zeta(z)$,
\begin{equation}\label{RSequation4}
z\zeta^3 - \frac{2z}{t(1-t)} \zeta^2 +
    \left(\frac{z}{t^2(1-t)^2} + \frac{1}{t(1-t)} - \frac{a}{t^2} \right) \zeta - \frac{1}{t^2(1-t)^2}=0,
\end{equation}
which will play a central role in what follows.
By solving for $z$, it may be written as
\begin{equation} \label{RSequation3}
    z = \frac{1-k \zeta}{\zeta(1-t(1-t)\zeta)^2}, \qquad k = (1-t)(t-a(1-t)).
    \end{equation}

\begin{proposition} \label{prop:RiemannSurface}
For every $t \in (0,1)$ the three-sheeted Riemann surface associated
with \eqref{RSequation3} has four  branch points at $0$, $\infty$, $p$ and $q$ with $p < q$.
There is a critical time
\begin{equation*} 
    t^* = \frac{a}{a+1} \in (0,1)
    \end{equation*}
such that
\begin{description}
\item[\rm Case 1:] for $t < t^*$  we have $0 < p < q$,
\item[\rm Case 2:] for $t > t^*$ we have $p < 0 < q$,
\item[\rm Case 3:] for $t = t^*$ we have $p= 0 < q$.
\end{description}
\end{proposition}
Note that the three cases  correspond to $k < 0$,  $k> 0$, and $k=0$, respectively,
where $k$ is the constant in \eqref{RSequation3}. The proof of Proposition \ref{prop:RiemannSurface}
follows from the discussion in Section \ref{section4}.

In this paper we are going to analyze Case 1 and Case 2. In order to handle the two
cases simultaneously, we shall denote the real branch points by $p_- < p_+ < q$,
where
\begin{equation*} 
    p_- = \min(0,p), \qquad p_+ = \max(0,p).
    \end{equation*}

Functions defined on the Riemann surface associated with \eqref{RSequation3} will play a major
role in the steepest descent analysis.
There is an alternative approach based on an equilibrium problem for
logarithmic potentials and so-called $g$-functions. We briefly
outline this approach in the appendix of this paper.

\subsection{Statement of results} \label{subsection25}

We state our results for the kernel \eqref{kernelKn},
\begin{equation} \label{kernel}
    K_n(x,y) = \frac{1}{2\pi i(x-y)} \begin{pmatrix} 0 & w_{1}(y) & w_{2}(y) \end{pmatrix}
    Y_+^{-1}(y) Y_+(x) \begin{pmatrix} 1 \\ 0 \\ 0 \end{pmatrix}
    \end{equation}
where $Y$ is the solution of the RH problem \eqref{Yjump}, \eqref{Yasym}, \eqref{Yedge}
with weights $w_1$ and $w_2$ as in \eqref{weights2}.
Note that $K_n$ depends on $a> 0$ and
$t \in (0,1)$. In the following $a$ will be fixed. To indicate the dependence on $t$
we occasionally write
\[ K_n(x,y) = K_n(x,y;t). \]
To emphasize the dependence of the branch points on $t$ we may write $p(t)$, $q(t)$,
$p_-(t)$, and $p_+(t)$.

\begin{theorem} 
\label{theo:domain}
Under the rescaling described above, the following hold.

For every $t \in (0,1)$, the limiting mean density of
the positions of the paths at time $t$
\begin{equation*} 
\rho(x) = \rho(x;t) = \lim_{n\to \infty} \frac{1}{n} K_n(x,x;t)
\end{equation*}
exists, and is supported on the interval $[p_+(t), q(t)] \subset [0, \infty)$.
The density $\rho$ satisfies
\begin{equation} \label{density}
\rho (x) = \frac{1}{\pi}\, \left| \Im \zeta(x) \right|, \quad p_+(t)\leq x \leq q(t),
\end{equation}
where $\zeta = \zeta(x)$ is a non-real solution of the equation
\eqref{RSequation3}.
\end{theorem}

From Theorem \ref{theo:domain} it follows that as $n \to \infty$, the
non-intersecting squared Bessel processes fill out a simply connected
region in the $tx$-plane given by
\[ 0 < t < 1, \qquad p_+(t) < x < q(t). \]
This region can be seen in Figure \ref{fig:SqBessel-50paths}.
\begin{figure}[htb]
\centering \begin{overpic}[scale=0.63]%
{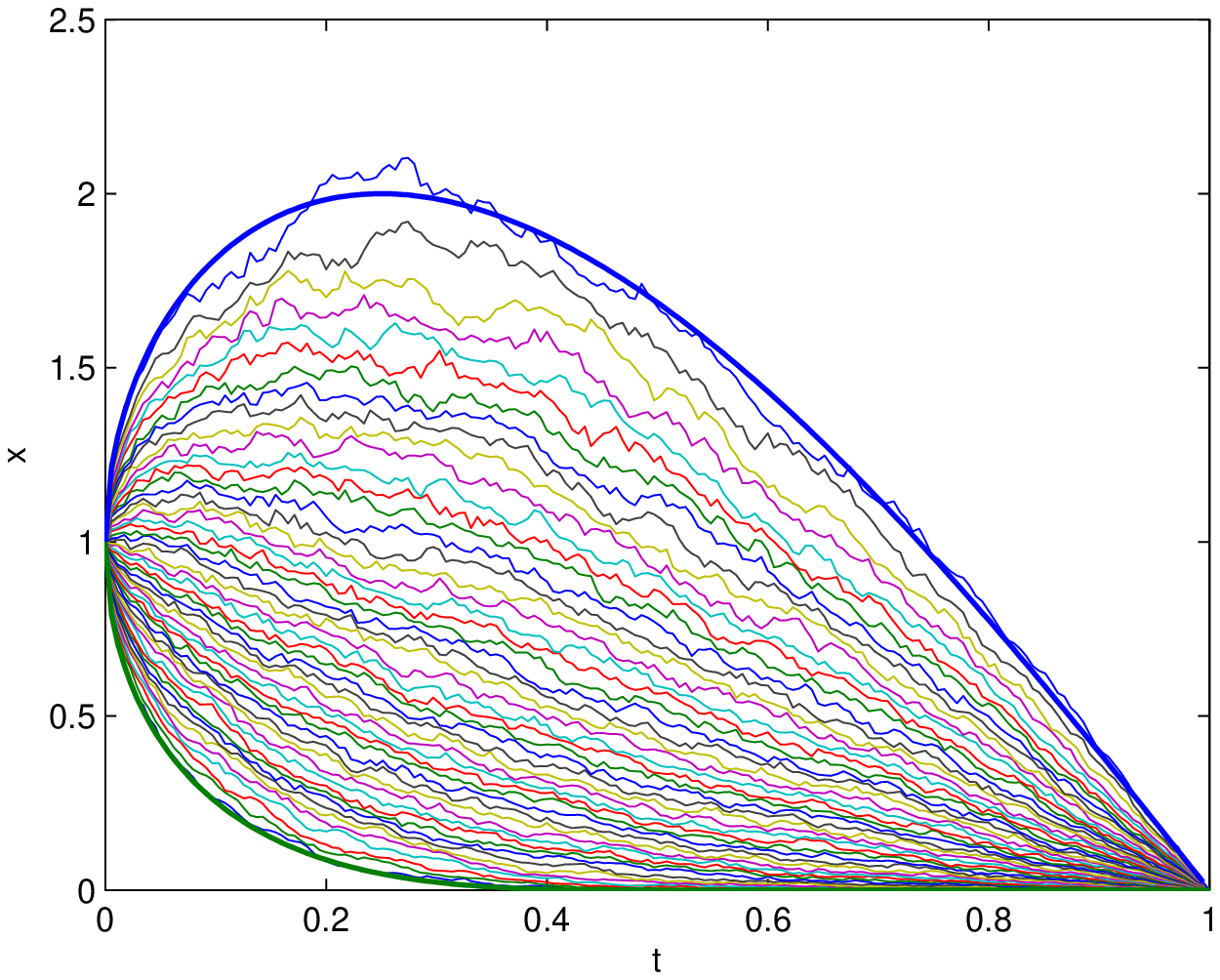}%
\end{overpic}\\
\begin{overpic}[scale=0.63]%
{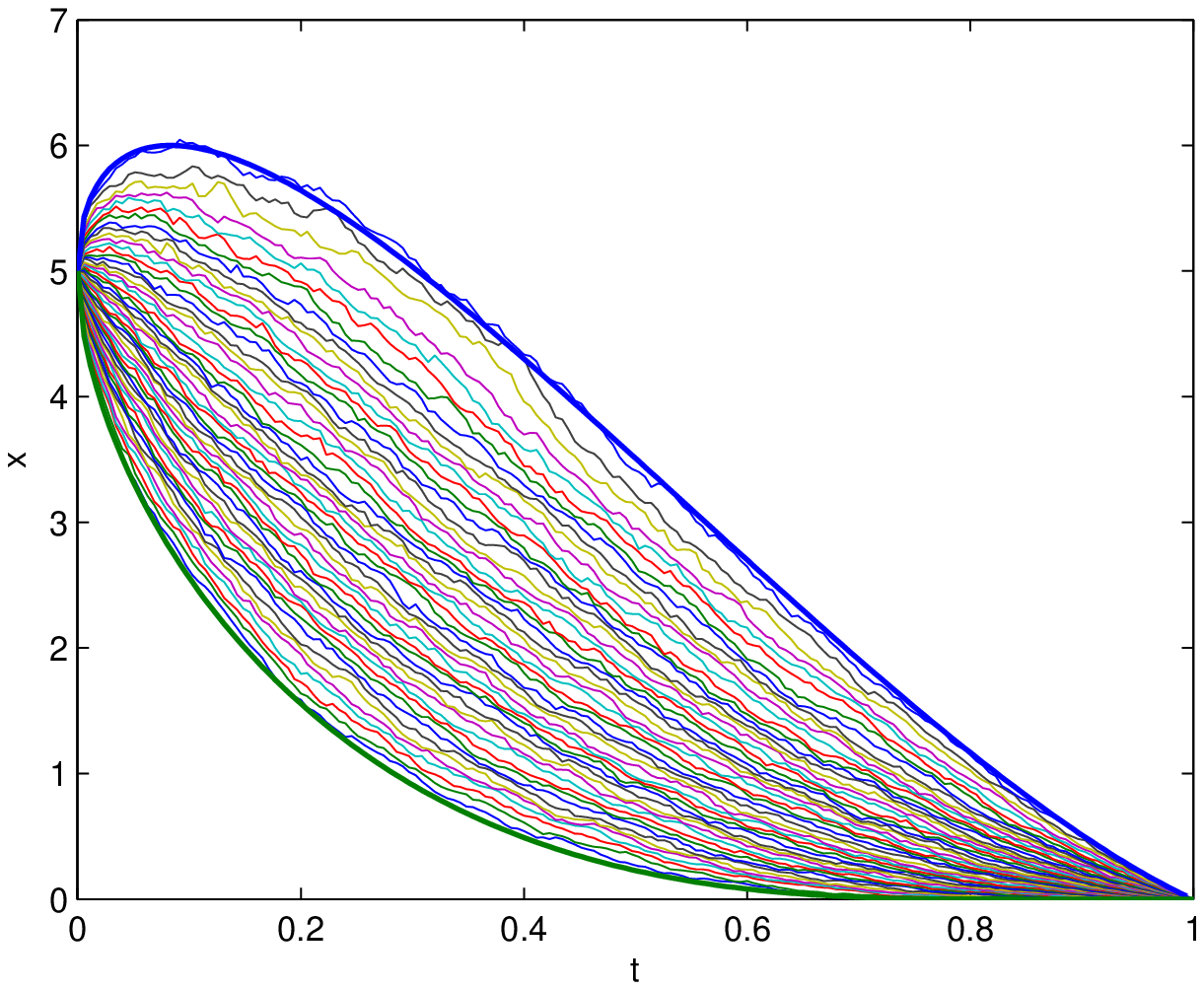}
\end{overpic}
\caption{Numerical simulation of $50$ rescaled non-intersecting BESQ$^2$ with $a=1$ (top) and $a=5$ (bottom).
Bold line is the boundary of the domain described in Theorem \ref{theo:domain}.}
\label{fig:SqBessel-50paths}
\end{figure}

From the definition of $p_+(t)$ and $q(t)$ as branch points of
the Riemann surface for \eqref{RSequation4} it may be shown that
$x = p_+(t)$, $x=q(t)$ are solutions of the algebraic equation
\begin{equation} \label{boundaryequation}
    4ax^3 + x^2(t^2 - 20a t(1-t) - 8a^2(1-t)^2) - 4x(1-t)(t-a(1-t))^3 = 0.
    \end{equation}
The locus of this algebraic curve in $0 < t < 1$, $x > 0$
gives us the boundary curve. Observe that it depends only on $a$, and is independent
from the parameter $\alpha $.

There are some peculiar features of the boundary curve, which may be checked
by direct calculation that we leave to the reader.

\begin{corollary} \label{cor:boundaryCurve}
For every $a> 0$ we have the following.
\begin{enumerate}
\item[\rm (a)] The lower boundary curve $x=p_+(t)$ is positive for $t < t^* = a/(a+1)$
and it is zero for $t \geq t^*$. At $t=t^*$ it has continuous first and second
order derivatives.
\item[\rm (b)] The upper boundary curve $x = q(t)$ has a slope
\[ q'(1) = -4 \]
at $t=1$ which is independent of the value of $a$.
\item[\rm (c)] The upper boundary curve $x=q(t)$ is concave  if $a \leq 1$.
It is not concave on the full interval $[0,1]$ if $a > 1$.
\item[\rm (d)]
The maximum of the upper boundary curve $x=q(t)$ is $a+1$.
\end{enumerate}
\end{corollary}

By continuity the results of Theorem \ref{theo:domain} and Corollary \ref{cor:boundaryCurve}
continue to hold for $a=0$, which is the case of non-intersecting squared Bessel bridges
\cite{KIK}.

\begin{remark}
 The numerical experiments leading to Figure \ref{fig:SqBessel-50paths} have been carried out
 exploiting the connection of the non-intersecting squared Bessel paths with the matrix-valued
 Laguerre process, as described in \cite{KT0, KO}. Indeed, let $\alpha \in \N \cup \{ 0\}$ and
 $b_{jk}, \widetilde b_{jk}$, $1\leq j\leq n+\alpha $, $1\leq k\leq n$, be independent one-dimensional
 standard Brownian motions.  Consider the $(n+\alpha ) \times n$   matrix-valued process
 $M(t)=(m_{jk})$ with entries $m_{jk}(t) = b_{jk}(t) + i \, \widetilde b_{jk}(t)$ and
 define the $n\times n$ symmetric positive definite matrix-valued process, called the Laguerre process, by
$$
\Xi(t) = M(t)^* M(t)\,, \qquad t\in [0, +\infty)\,,
$$
where $M(t)^*$ denotes the conjugate transpose of $M(t)$. Then the process of eigenvalues of
$\Xi(t)$ and the noncolliding $n$-particle system of BESQ$^d$, with $d=2(\alpha +1)$, are equivalent in distribution.
\end{remark}

Finally, in the non-critical case $t \neq t^*$ we find the usual scaling limits from random matrix theory,
namely the sine, Airy, and Bessel kernels.

\begin{theorem} \label{theo:local1}
Let $t \neq t^*$. Then for $x^* \in (p_+(t),q(t))$, we have
\[ \lim_{n \to \infty} \frac{1}{n \rho(x^*)} K_n\left(x^* + \frac{x}{n \rho(x^*)}, x^* + \frac{y}{n \rho(x^*)}\right)
    = \frac{\sin \pi (x-y)}{\pi (x-y)} \]
uniformly for $x$ and $y$ in compact subsets of $\mathbb R$.
\end{theorem}

\begin{theorem} \label{theo:local2}
Let $t \neq t^*$. Then for some constant $c >0$,
\[ \lim_{n \to \infty} \frac{1}{cn^{2/3}} K_n\left(q(t) + \frac{x}{cn^{2/3}}, q(t) + \frac{y}{cn^{2/3}} \right)
    = \frac{\Ai(x) \Ai'(y) - \Ai'(x) \Ai(y)}{x-y}. \]

If $t < t^*$, then for some constant $c > 0$,
\[ \lim_{n \to \infty} \frac{1}{cn^{2/3}} K_n\left(p_+(t) - \frac{x}{cn^{2/3}}, p_+(t) - \frac{y}{cn^{2/3}} \right)
    = \frac{\Ai(x) \Ai'(y) - \Ai'(x) \Ai(y)}{x-y}. \]
\end{theorem}

\begin{theorem} \label{theo:local3}
Let $t > t^*$. Then for some constant $c > 0$, and $x, y >0$,
\[ \lim_{n \to \infty} \frac{1}{cn^2} K_n\left(\frac{x}{cn^2}, \frac{y}{cn^2} \right)
    = \left(\frac{y}{x}\right)^{\alpha/2} \frac{J_{\alpha}(\sqrt{x}) \sqrt{y} J_{\alpha}'(\sqrt{y})
        - \sqrt{x} J_{\alpha}'(\sqrt{x}) J_{\alpha}(\sqrt{y})}{2(x-y)}. \]
\end{theorem}

In the bulk we find the sine kernel, at the soft edges we find the
Airy kernel, and at the hard edge $0$ we find the Bessel kernel
of order $\alpha$. Note that the factor $\left(y/x\right)^{\alpha/2}$
in the Bessel kernel is not important since it will not influence the determinantal
correlation functions. This observation also explains why totally symmetric results are obtained
if we reverse the process and study $n$ non-intersecting BESQ$^d$ paths starting at the origin and
ending at a positive value $a$. Indeed,  $\left(y/x\right)^{\alpha/2}$ is the only factor in the
transition probabilities \eqref{pt} that is not symmetric in its variables.

At $t = t^*$ there is a transition from the Airy
kernel to the Bessel kernel. This is when the non-intersecting squared Bessel paths
first hit the hard edge. The soft-to-hard edge transition is different from
previous ones considered in \cite{BF, CK2}.
We will treat this transition in a separate publication.

Observe also that neither the boundary of the domain filled by the scaled paths, nor the behavior
in the bulk or at the soft edge depends on the parameter $\alpha$ related to the dimension $d$ of
the BESQ$^d$. This dependency appears only in the interaction with the hard edge at $x=0$.
A possible interpretation may be that $\alpha$ is a measure for the interaction with the hard edge.
It does not influence the global behavior as $n \to \infty$, but only the local behavior near $0$.

\begin{remark} \label{rem:Besselpaths}
By taking square roots we can transplant Theorems \ref{theo:domain} and
\ref{theo:local1}--\ref{theo:local3} to the case of non-intersecting
Bessel paths. The correlation kernel for the positions of
non-intersecting Bessel paths, starting at $\sqrt{a}$ and ending
at $0$ is
\[  2 \sqrt{xy} K_n(x^2,y^2) \]
where $K_n$ is the kernel \eqref{kernel} as before. It is then easy to show
from Theorems \ref{theo:local1} and \ref{theo:local2} that the scaling limits
are again the sine kernel in the bulk and the Airy kernel at the soft edges. At the hard
edge however, Theorem \ref{theo:local3} gives the scaling limit
\[  \left(\frac{y}{x} \right)^{\alpha} \frac{\sqrt{xy}}{x+y}
            \frac{J_{\alpha}(x) y J_{\alpha}'(y)
        - x J_{\alpha}'(x) J_{\alpha}(y)}{x-y}. \]
\end{remark}

The proofs of Theorems \ref{theo:domain}, \ref{theo:local1}, \ref{theo:local2}, and \ref{theo:local3}
are given in Section \ref{section10}. They follow from the steepest descent analysis of
the RH problem for $Y$. The steepest descent analysis itself takes most of the paper,
see Sections \ref{section3}--\ref{section9}.

Since we will be dealing extensively with $3 \times 3$ matrices we find it useful
to use the notation $E_{ij}$ to denote
the $3 \times 3$ elementary matrix whose entries are all $0$, except for
the $(i, j)$-th entry, which is $1$. Thus
\begin{equation} \label{Eijdef}
    \left(E_{ij} \right)_{k,l} = \delta_{i,k} \delta_{j,l}
    \end{equation}
for $i,j,k,l \in \{ 1,2,3\}$.
The following properties can be easily checked and will be used without comment.
\begin{lemma} \label{lem:M}
\begin{enumerate}
\item[\rm (a)]
For $i, j, k , l \in \{1, 2, 3\}$,
$$
E_{ij} E_{kl}=\begin{cases}
E_{il}, & \text{if } j=k\,, \\
O, & \text{otherwise.}
\end{cases}
$$
\item[\rm (b)]
If $c\in \C$ and $i, j  \in \{1, 2, 3\}$, $i\neq j$, then $I + c E_{ij}$ is invertible, and
$$
\left(I + c E_{ij} \right)^{-1} = I - c E_{ij}\,.
$$
\end{enumerate}
\end{lemma}

\section{First transformation of the RH problem}\label{section3}

We apply the Deift-Zhou method of steepest descent to the RH problem  \eqref{Yjump}, \eqref{Yasym},
\eqref{Yedge} with weights $w_1$ and $w_2$ given by
\eqref{weights2} and with indices $n_1$ and $n_2$ as follows:
\begin{equation} \label{n1n2}
    n_1 = \begin{cases}
    n/2, & \text{if $n$ is even}, \\
    (n+1)/2, & \text{if $n$ is odd,}
    \end{cases} \qquad
    n_2 = \begin{cases}
    n/2, & \text{if $n$ is even}, \\
    (n-1)/2, & \text{if $n$ is odd.}
    \end{cases}
 \end{equation}
The steepest descent analysis has certain new features that have
not appeared in the literature before.

A possible approach was suggested by Van Assche et al.\ in \cite{VAGK},
since the system of weights \eqref{weights2} is a Nikishin system \cite{NS}. This means (in this case)
that
\begin{equation}\label{discreteCauchy}
    \frac{w_2(x)}{w_1(x)} = x\int_{-\infty}^{0} \frac{d \sigma_n(u)}{x-u}\,,
\end{equation}
where $\sigma_n$ is a discrete measure on the negative real line, see \cite[Theorem 1]{CV2},
with masses at the point
\begin{equation} \label{Besselzeros}
    - (t\, j_{\alpha,k}/(2n\sqrt{a}))^2, \qquad  k=1,2,\ldots,
\end{equation}
where $j_{\alpha,k}$, $k=1,2, \ldots$, are the positive zeros of the Bessel
function $J_{\alpha}$. The approach of \cite{VAGK} would involve a preliminary
transformation
\begin{equation} \label{Xdefold}
    X =
         Y \left( I - \frac{w_2}{w_1} E_{23} \right)
\end{equation}
which would result in a jump condition
\begin{equation} \label{Xjumpold}
    X_+ =
    X_- \left( I + w_1 E_{12} \right)
    \end{equation}
on $\mathbb R_+$. Since $\frac{w_2}{w_1}$ has  poles on the negative real line,
the third column of $X$ has poles on the negative real line, which could be
described by residue conditions as in \cite{BKMM}. We might then continue as in
\cite{BKMM} by turning the residue conditions into jump conditions. However we
will not follow this approach and we will not use the transformation \eqref{Xdefold}.

Instead, our first transformation is based on the special properties
of the modified Bessel functions. We introduce the two functions
\begin{equation} \label{y1y2}
y_{1} (z)=z^{(\alpha+1)/2}I_{\alpha+1} (2\sqrt{z}), \qquad
y_{2} (z)=z^{(\alpha+1)/2}K_{\alpha+1} (2\sqrt{z}),
\end{equation}
where $K_{\alpha+1}$ denotes the modified Bessel function of the second
kind, see \cite[Section 9.6]{AS} for its main
properties. The functions $y_{1}$ and $y_{2}$ are
defined and analytic in the complex plane with a branch cut along the negative real
axis.
The jumps on $\R_-$
can be computed from the formulas 9.6.30 and 9.6.31 of \cite{AS}. We have
\begin{equation}\label{jumpY}
\begin{aligned}
y_{1+} (x)&=e^{2i \alpha\pi}y_{1-} (x),& x<0,\\
y_{2+} (x)&=y_{2-} (x) -i\pi e^{i \alpha\pi}y_{1-} (x), & x<0.
\end{aligned}
\end{equation}

From the expressions for the derivatives of the modified Bessel
functions, see \cite[formulas 9.6.26]{AS}, we deduce that
\begin{equation} \label{y1y2prime}
y_{1}' (z)=z^{\alpha/2}I_{\alpha} (2\sqrt{z}),\qquad
y_{2}' (z)=-z^{\alpha/2}K_{\alpha} (2\sqrt{z}).
\end{equation}
The relations \eqref{y1y2} and \eqref{y1y2prime} imply
that the weights $w_{1}$ and $w_{2}$ defined by \eqref{weights2} can be expressed in
terms of the function $y_{1}$ and its derivative $y_{1}'$ as
\begin{equation}\label{wY}
\begin{aligned}
w_{1} (x) &= \tau^{-\alpha} \exp\left(-\frac{nx}{t (1-t)}\right) y_{1}'(\tau^2 x),\\
w_{2} (x) &= \tau^{-\alpha-1}\exp\left(-\frac{nx}{t (1-t)}\right) y_{1} (\tau^2 x).
\end{aligned}
\end{equation}
where we have put
\begin{equation*} 
    \tau = \tau_n = \frac{n\sqrt{a}}{t}.
    \end{equation*}
We also need the following wronskian relation, see formula 9.6.15 of \cite{AS},
\begin{equation}\label{wronsk}
y_{1} (z)y_{2}' (z)-y_{1}' (z)y_{2} (z)=-\frac{z^{\alpha}}{2},\quad
z\in\C\setminus\R_-.
\end{equation}

Now, we are in a position to define the first transformation of the RH problem
\eqref{Yjump}--\eqref{Yedge}. The aim of the first
transformation is to modify the jump matrix in order to have only one
remaining weight on $\R_+$, as in \eqref{Xjumpold}, which is also simpler than the
weights $w_1$ and $w_2$. Indeed, relations \eqref{wY} and \eqref{wronsk} allow to
remove the modified Bessel functions from the jumps, replacing them by a
simple power function.
The price we have to pay for the simpler jump on $\R_+$ will be a new
jump appearing on $\R_-$ and on two contours $\Delta_2^{\pm}$ that are
taken as in Figure \ref{fig:first_transformation}.
We take $\Delta_2^+$ as an unbounded contour in the second quadrant asymptotic
to a ray $\arg z = \theta$ for some $\theta \in (\pi/2, \pi)$ as $z \to \infty$,
and meeting the real axis at the point $p_- \leq 0$. Its mirror image
in the real axis is the contour $\Delta_2^-$. The contours $\Delta_2^{\pm}$
are the boundary of a domain containing the interval $(-\infty, p_-)$ and we
refer to this domain as the lens around $(-\infty,p_-)$.

\begin{figure}[htb]
\centering \begin{overpic}[scale=1.2]%
{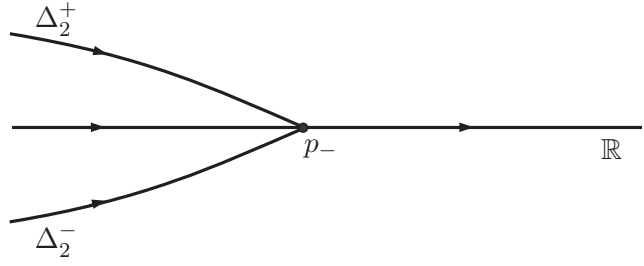}%
      \put(90,14){$\R $}
       \put(8,33){$\Delta_2^+ $}
       \put(8,1){$\Delta_2^- $}
       \put(47,15){$p_-$}
\end{overpic}
\caption{Contour for the first transformation.}
\label{fig:first_transformation}
\end{figure}

We define for $z \in \mathbb C \setminus \mathbb R$,
\begin{multline} \label{Xhatdef}
\widetilde{X} (z) = C_1
        Y (z)
        \begin{pmatrix}
        1 & 0 & 0 \\
        0 & 1 & 0 \\
        0 & 0 & \tau \end{pmatrix}
        \begin{pmatrix}
        1 & 0 & 0\\
        0 & 2y_{2} (\tau^2 z) & - z^{-\alpha}y_{1} (\tau^2 z)\\
        0 & -2y_{2}' (\tau^2 z) & z^{-\alpha} y_{1}' (\tau^2 z)
        \end{pmatrix}
    \begin{pmatrix}
    1 &0&0\\
    0 &  \tau^{-\alpha}& 0\\
    0 & 0 & -2\pi i \tau^{-\alpha}
    \end{pmatrix}\,,
\end{multline}
where $C_1$ is the constant matrix
\begin{equation} \label{C1def}
    C_1 = \begin{cases}
        \begin{pmatrix} 1 & 0 & 0 \\
        0 & 1 & i \frac{4(\alpha+1)^2-1}{16\tau} \\
        0 & 0 & 1
        \end{pmatrix}
     \begin{pmatrix}
        1 & 0 & 0 \\
        0 & (2\pi \tau)^{-1/2} & 0 \\
        0 & 0 & i (2\pi \tau)^{-1/2}
        \end{pmatrix}, & \text{if $n$ is even}, \\[5mm]
        \begin{pmatrix} 1 & 0 & 0 \\
        0 & 1 & 0 \\
        0 & -\frac{4\alpha^2-1}{16\tau} & 1
        \end{pmatrix}
     \begin{pmatrix}
        1 & 0 & 0 \\
        0 & (2\pi \tau)^{-1/2} & 0 \\
        0 & 0 & i (2\pi \tau)^{-1/2}
        \end{pmatrix}, & \text{if $n$ is odd}.
        \end{cases} \end{equation}
 Note that, in view of the wronskian relation
\eqref{wronsk}, the determinant of the fifth matrix in the right-hand side of
\eqref{Xhatdef} is equal to $\tau^{2\alpha}$. Then it is easy to see that
$\det \widetilde{X}(z) \equiv 1$.
The matrix $\widetilde{X} (z)$ is analytic in $\C\setminus\R$ since the matrix $Y (z)$ is
analytic in $\C\setminus\R_{+}$ and
$y_{1}(\tau^2 z)$ and $y_{2}(\tau^2 z)$ are analytic in $\C\setminus\R_-$.
Now define
\begin{equation} \label{Xdef1}
    X(z) = \widetilde{X}(z)
\end{equation}
for $z$  outside the lens around $(-\infty,p_-)$, and
\begin{equation} \label{Xdef2}
    X(z)  = \widetilde{X}(z) \left( I \mp e^{\pm \alpha \pi i} z^{-\alpha} E_{23} \right)
\end{equation}
for $z$ in the part of the lens bounded by $\Delta^{\pm}_2$
and  $(-\infty,p_-)$. [Recall that $E_{ij}$ is used to denote the elementary matrix \eqref{Eijdef}.]

From \eqref{wY}, \eqref{wronsk}, the jump relations \eqref{jumpY}, and
the fact that $Y (z)$ is the solution of the RH problem
\eqref{Yjump}--\eqref{Yedge}, one derives the jump relations
\eqref{Xjump1}--\eqref{Xjump4} below.
As $z \to 0$, we note the following behavior
\begin{align*}
y_{1} (z) & \sim \frac{1}{\Gamma(\alpha+1)} z^{\alpha+1}, &
y_{2} (z) & \sim \frac{1}{2} \Gamma(\alpha+1), \\
y_{1}' (z) & \sim \frac{1}{\Gamma(\alpha)} z^{\alpha}, &
y_{2}' (z) & \sim
\left\{\begin{array}{ll}
-\frac{1}{2}\Gamma (\alpha),& \alpha > 0, \\
\frac{1}{2} \log (z),& \alpha=0,\\
-\frac{1}{2}\Gamma (-\alpha)z^{\alpha},& \alpha<0,
\end{array}\right.
\end{align*}
which is a consequence of the known behavior of the modified Bessel
functions near 0, see formulas 9.6.7--9.6.9 of \cite{AS}.
This shows that $\widetilde X(z)$ has the same kind of behavior as $Y(z)$
at the origin. The behavior of $X(z)$ near the origin is then
also the same, except in case $p_-= 0$ and $\alpha \geq 0$,
see \eqref{Xdef2}.
The result is that
$X(z)$ is  the solution of the following RH problem:

\begin{proposition} \label{prop:RHforX}
The matrix-valued function $X(z)$ defined by \eqref{Xhatdef}, \eqref{Xdef1},
and \eqref{Xdef2} is the unique solution of the following RH problem.
\begin{enumerate}
\item $X(z)$ is analytic in $\mathbb{C} \setminus (\R \cup \Delta_2^{\pm})$.
\item $X(z)$ possesses continuous boundary values on $(\R \cup \Delta_2^{\pm})\setminus\{0 \}$
denoted by $X_+$ and $X_-$, where $X_+$ and $X_-$ denote the limiting
values of $X(z)$ as $z$ approaches the contour from the left and the right,
according to the orientation on $\R$ and $\Delta_2^{\pm}$ as indicated
in Figure~\ref{fig:first_transformation}, and
\begin{equation}
\label{Xjump1}
X_+(x)=X_-(x) \left( I + x^{\alpha} e^{-\frac{nx}{t(1-t)}} E_{12} \right)
     \quad x\in\R_+,
\end{equation}
\begin{equation}  \label{Xjump2}
X_+(x)=X_-(x)
    \begin{pmatrix}
    1 & 0 & 0\\
        0 & 0                                & -|x|^{-\alpha} \\
        0 & |x|^{\alpha}                     & 0
    \end{pmatrix}, \quad x\in (-\infty,p_-),
\end{equation}
\begin{equation}  \label{Xjump3}
X_+(x)=X_-(x) \left( I + |x|^{\alpha} E_{32} \right)
 \quad x\in (p_-,0),
\end{equation}
\begin{equation} \label{Xjump4}
    X_+(z) = X_-(z) \left( I + e^{\pm \alpha \pi i} z^{-\alpha} E_{23} \right)
\quad z \in \Delta_2^{\pm}.
\end{equation}
\item $X(z)$ has the following behavior near infinity:
\begin{multline} \label{Xasymptotics}
   X(z) =
   \left(I + \mathcal{O}\left(\frac{1}{z}\right) \right)
    \begin{pmatrix}
    1 & 0 & 0\\
    0 & z^{(-1)^n/4} & 0 \\
    0 & 0 &  z^{- (-1)^n/4}
    \end{pmatrix}
     \begin{pmatrix}
    1 & 0 & 0 \\
    0 & \frac{1}{\sqrt{2}} & \frac{1}{\sqrt{2}} i \\
    0 & \frac{1}{\sqrt{2}} i & \frac{1}{\sqrt{2}}
    \end{pmatrix} \\
    \begin{pmatrix}
    1 & 0 & 0 \\
    0 & z^{\alpha/2} & 0 \\
    0 & 0 & z^{-\alpha/2}
    \end{pmatrix}
\begin{pmatrix}
            z^{n}  & 0 & 0 \\
            0 & z^{-n/2} e^{-2n\sqrt{az}/t} & 0 \\
            0 & 0 & z^{-n/2} e^{2n\sqrt{az}/t}
            \end{pmatrix},
\end{multline}
uniformly as $z \to \infty$, $z \in \mathbb C \setminus \mathbb R$.
\item $X(z)$ has the same behavior as $Y(z)$ at
the origin, see \eqref{Yedge}, either if $p_- < 0$ or if $z\to p_-= 0$ outside
the lens around $(-\infty, p_-)$. If $p_- = 0$ and  $z \to 0$ in the lens around $(-\infty, p_-)$, then
\begin{equation} \label{Xedge2}
    X(z) = \left\{ \begin{array}{cl}
    \mathcal O \begin{pmatrix} 1 & |z|^{\alpha} & 1 \\  1 & |z|^{\alpha} & 1 \\ 1 & |z|^{\alpha} & 1 \end{pmatrix}
    & \textrm{if } \alpha < 0, \\
    \mathcal O\begin{pmatrix} 1 & \log |z| & \log |z| \\ 1 & \log |z| & \log |z| \\ 1 & \log |z| & \log |z| \end{pmatrix}
    & \textrm{if } \alpha = 0, \\
    \mathcal O\begin{pmatrix} 1 & 1 & |z|^{-\alpha} \\ 1 & 1 & |z|^{-\alpha} \\1 & 1 & |z|^{-\alpha} \end{pmatrix}
    & \textrm{if } \alpha > 0.
    \end{array} \right. \end{equation}
\end{enumerate}
\end{proposition}
\begin{proof}
Statements in items 1, 2 and 4 are proved by straightforward calculations.
It only remains to check the asymptotic behavior at infinity
given in item 3.
This follows from the asymptotic expansions
\begin{equation}\label{asympY1}
\begin{aligned}
y_{1}(z) &
=\frac{1}{2\sqrt{\pi}}z^{\frac{\alpha}{2} +\frac{1}{4}}e^{2\sqrt{z}}
\left(1- \frac{4(\alpha+1)^2-1}{16\sqrt{z}} + \frac{(4(\alpha+1)^2-1)(4(\alpha+1)^2-9)}{512 z}
+ \mathcal{O}\left(\frac{1}{z^{3/2}}\right)\right), \\
y_{1}'
(z) &
=\frac{1}{2\sqrt{\pi}}z^{\frac{\alpha}{2}-\frac{1}{4}}e^{2\sqrt{z}}
\left(1- \frac{4\alpha^2-1}{16\sqrt{z}} + \frac{(4\alpha^2-1)(4\alpha^2-9)}{512 z}
+ \mathcal{O}\left(\frac{1}{z^{3/2}}\right)\right),
\end{aligned}
\end{equation}
as $z \to \infty$, $| \arg z| < \pi$, and
\begin{equation} \label{asympY2}
\begin{aligned}
y_{2} (z) & =\frac{\sqrt{\pi}}{2}
z^{\frac{\alpha}{2}+ \frac{1}{4}}e^{-2\sqrt{z}}
\left(1+\frac{4(\alpha+1)^2-1}{16\sqrt{z}} + \frac{(4(\alpha+1)^2-1)(4(\alpha+1)^2-9)}{512 z}
+ \mathcal{O}\left(\frac{1}{z^{3/2}}\right)\right), \\
y_{2}' (z) & =-\frac{\sqrt{\pi}}{2}
z^{\frac{\alpha}{2}-\frac{1}{4}}e^{-2\sqrt{z}}
\left(1+ \frac{4\alpha^2-1}{16\sqrt{z}} + \frac{(4\alpha^2-1)(4\alpha^2-9)}{512 z}
+ \mathcal{O}\left(\frac{1}{z^{3/2}}\right)\right),
\end{aligned}
\end{equation}
as $z\to\infty$, $|\arg z|< 3\pi$.
These formulas are consequences of the corresponding asymptotic
expansions of the modified Bessel functions, see formulas (9.7.1)--(9.7.4) of \cite{AS}.

It follows from \eqref{asympY1} and \eqref{asympY2} that
\begin{align} \label{Az1}
A(z) & :=
        \begin{pmatrix}
         1 & 0 \\
         0 & \tau \end{pmatrix}
        \begin{pmatrix}
         2y_{2} (\tau^2 z) & - z^{-\alpha}y_{1} (\tau^2 z)\\
        -2y_{2}' (\tau^2 z) & z^{-\alpha} y_{1}' (\tau^2 z)
        \end{pmatrix}
    \begin{pmatrix}
      \tau^{-\alpha}& 0\\
       0 & -2\pi i \tau^{-\alpha}
    \end{pmatrix} \\  \nonumber
    & = \sqrt{\pi \tau}  z^{\sigma_3/4}
    \left[ \begin{pmatrix} 1 & i \\ 1 & -i \end{pmatrix} +
    \frac{D_1}{z^{1/2}} \begin{pmatrix} 1 & -i \\ i & -1 \end{pmatrix}
    + \frac{D_2}{z} \begin{pmatrix} 1 & i \\ i & 1 \end{pmatrix} + \mathcal{O}(z^{-3/2})\right]
    z^{\alpha \sigma_3/2} e^{-2\tau \sqrt{z} \sigma_3}
    \end{align}
as $z \to \infty$, $|\arg z| < \pi$, where
 $D_1$ and $D_2$ are diagonal matrices
\begin{equation*}
\begin{aligned}
    D_1 & = \frac{1}{16\tau} \begin{pmatrix} 4(\alpha+1)^2-1 & 0 \\ 0 & -i(4\alpha^2-1) \end{pmatrix}, \\
    D_2 & = \frac{1}{512 \tau^2}
    \begin{pmatrix} (4(\alpha+1)^2-1)(4(\alpha+1)^2-9) & 0 \\ 0 & -i(4\alpha^2-1)(4\alpha^2-9) \end{pmatrix}
\end{aligned}
\end{equation*}
and $\sigma_3 = \left(\begin{smallmatrix} 1 & 0 \\ 0 & - 1 \end{smallmatrix} \right)$
is the third Pauli matrix.
Thus
\begin{equation} \label{Az2}
    A(z) = \sqrt{\pi \tau} z^{\sigma_3/4}
    \left[ \begin{pmatrix} 1 & 0 \\ 0 & -i \end{pmatrix}
        + \frac{D_1 \sigma_2}{z^{1/2}} + \frac{D_2}{z} + \mathcal{O}(z^{-3/2})\right]
    \begin{pmatrix} 1 & i \\ i & 1 \end{pmatrix}
    z^{\alpha \sigma_3/2} e^{-2\tau \sqrt{z} \sigma_3}
\end{equation}
where $\sigma_2 = \left(\begin{smallmatrix} 0 & -i \\ i & 0 \end{smallmatrix} \right)$.
Now $\frac{D_2}{z}$ commutes with $z^{\sigma_3/4}$ since both are diagonal matrices.
We also have
\[ z^{\sigma_3/4} \frac{D_1 \sigma_2}{z^{1/2}}  = D_1 \begin{pmatrix} 0 & -i \\ i z^{-1} & 0
\end{pmatrix} z^{\sigma_3/4}. \]
The result is that \eqref{Az2} leads to
\begin{align}
\nonumber
    A(z) & = \sqrt{\pi \tau} \left[\begin{pmatrix} 1 & 0 \\ 0 & -i \end{pmatrix} +
    D_1 \begin{pmatrix} 0 & -i \\ 0 & 0 \end{pmatrix}
    + \mathcal{O}(z^{-1})\right] z^{\sigma_3/4}
    \begin{pmatrix} 1 & i \\ i & 1 \end{pmatrix}
    z^{\alpha \sigma_3/2} e^{-2\tau \sqrt{z} \sigma_3} \\
    & = \sqrt{2\pi \tau}
    \begin{pmatrix} 1 & -i\frac{4(\alpha+1)^2-1}{16\tau} \\ 0 & -i \end{pmatrix}
    \left( I + \mathcal{O}(z^{-1})\right)z^{\sigma_3/4}
    \frac{1}{\sqrt{2}}\begin{pmatrix} 1 & i \\ i & 1 \end{pmatrix}
    z^{\alpha \sigma_3/2} e^{-2\tau \sqrt{z} \sigma_3}
    \label{Az3}
\end{align}
as $z \to \infty$, $|\arg z| < \pi$.

Now if $n$ is even we use \eqref{Yasym} with $n_1 = n_2 = n/2$, see \eqref{n1n2},
along with
\eqref{Az1},  \eqref{Az3} in \eqref{Xhatdef}--\eqref{Xdef1} to find
that \eqref{Xasymptotics} holds as $z \to \infty$ in the region exterior
to $\Delta_2^{\pm}$. The asymptotics is uniform in that region.

If $n$ is odd then $n_1 = n/2 + 1/2$ and $n_2 = n/2 - 1/2$, see \eqref{n1n2}.
Then we need to analyze $z^{-\sigma_3/2} A(z)$ with $A$ given by \eqref{Az1}.
A computation similar to the one that led to \eqref{Az3} gives us
\[ z^{-\sigma_3/2} A(z) = \sqrt{2\pi \tau}
    \begin{pmatrix} 1 & 0 \\ \frac{4\alpha^2-1}{16\tau} & -i \end{pmatrix}
    \left( I + \mathcal{O}(z^{-1})\right)z^{\sigma_3/4}
    \frac{1}{\sqrt{2}}\begin{pmatrix} 1 & i \\ i & 1 \end{pmatrix}
    z^{\alpha \sigma_3/2} e^{-2\tau \sqrt{z} \sigma_3}
\]
and \eqref{Xasymptotics} follows as well, taking into account the different
formula \eqref{C1def} for the case $n$ is odd.

The asymptotic formulas \eqref{asympY1} are not valid uniformly up to the negative
real axis. The special combination $y_1 - \frac{1}{\pi i} e^{\alpha \pi i} y_2$
however, does have the asymptotics \eqref{asympY1} uniformly
for $\pi/2 < \arg z \leq \pi$
and $y_1 + \frac{1}{\pi i} e^{-\alpha \pi i} y_2$
has the asymptotics \eqref{asympY1} uniformly for $-\pi \leq \arg z < -\pi/2$.
This can be seen from the formulas that connect the
various Bessel functions (combine formulas 9.1.3-4, 9.1.35, 9.6.3-4 of \cite{AS})
\begin{align*}
    y_1(z) - \frac{1}{\pi i} e^{\alpha \pi i} y_2(z) & =
    z^{(\alpha +1)/2} H_{\alpha+1}^{(1)}(2 \sqrt{z} e^{-\pi i/2}), \\
    y_1'(z) - \frac{1}{\pi i} e^{\alpha \pi i} y_2'(z) & =
    z^{\alpha/2} H_{\alpha}^{(1)}(2 \sqrt{z} e^{-\pi i/2}), \\
    y_1(z) + \frac{1}{\pi i} e^{-\alpha \pi i} y_2(z) & =
    z^{(\alpha +1)/2} H_{\alpha+1}^{(2)}(2 \sqrt{z} e^{\pi i/2}), \\
    y_1'(z) + \frac{1}{\pi i} e^{-\alpha \pi i} y_2'(z) & =
    z^{\alpha/2} H_{\alpha}^{(2)}(2 \sqrt{z} e^{\pi i/2}),
\end{align*}
where $H_{\alpha}^{(1)}$ and $H_{\alpha}^{(2)}$ are the Hankel
functions, and the asymptotic expansions (see \cite[9.2.7-10]{AS})  of the Hankel
functions in the upper and lower half-planes, respectively.
Hence
\begin{equation} \label{asympY3}
\begin{aligned}
y_1(z) & \mp \frac{1}{\pi i} e^{\pm \alpha \pi i} y_2(z) \\
    & =\frac{1}{2\sqrt{\pi}}z^{\frac{\alpha}{2} +\frac{1}{4}}e^{2\sqrt{z}}
      \left(1- \frac{4(\alpha+1)^2-1}{16\sqrt{z}} + \frac{(4(\alpha+1)^2-1)(4(\alpha+1)^2-9)}{512 z}
    + \mathcal{O}\left(\frac{1}{z^{3/2}}\right)\right), \\
y_{1}'(z) & \mp \frac{1}{\pi i} e^{\pm \alpha \pi i} y_2'(z) \\
     & =\frac{1}{2\sqrt{\pi}}z^{\frac{\alpha}{2}-\frac{1}{4}}e^{2\sqrt{z}}
    \left(1- \frac{4\alpha^2-1}{16\sqrt{z}} + \frac{(4\alpha^2-1)(4\alpha^2-9)}{512 z}
    + \mathcal{O}\left(\frac{1}{z^{3/2}}\right)\right),
\end{aligned}
\end{equation}
uniformly as $z \to \infty$ in the region bounded by $\Delta_2^{\pm}$ and the negative
real axis.
Using the asymptotics \eqref{asympY2} and \eqref{asympY3},
and the definition \eqref{Xdef2} of $X(z)$ in the regions bounded by $\Delta_2^{\pm}$
and the negative real axis, we obtain by the same calculations that \eqref{Xasymptotics}
holds uniformly as $z \to \infty$ in these regions as well.

This completes the proof of Proposition \ref{prop:RHforX}.
\end{proof}

\section{The Riemann surface and the second transformation of the RH problem}\label{section4}
The Riemann surface $\RR$ for the algebraic equation \eqref{RSequation4}
plays an important role in the next transformation of the RH problem.
We repeat it here in the form \eqref{RSequation3}
\begin{equation}\label{RSequation5}
    z =  \frac{1- k\zeta}{\zeta (1-t(1-t)\zeta)^{2}}, \qquad
    k= (1-t)(t-a(1-t)).
\end{equation}

There are three inverse functions to \eqref{RSequation5}, which we choose
such that as $z\to\infty$,
\begin{align}
\zeta_{1} (z) & = \frac{1}{z}+ \frac{(1-t)(t+a(1-t))}{z^2} + \mathcal{O} \left(\frac{1}{z^{3}} \right), \label{zeta1}\\
\zeta_{2} (z) & = \frac{1}{t (1-t)}-\frac{\sqrt{a}}{tz^{1/2}}-\frac{1}{2z} -
\frac{t+4a(1-t)}{8\sqrt{a}z^{3/2}}  - \frac{(1-t)(t+a(1-t))}{2z^2}+\mathcal{O} \left(\frac{1}{z^{5/2}} \right), \label{zeta2}\\
\zeta_{3} (z) & = \frac{1}{t (1-t)}+\frac{\sqrt{a}}{tz^{1/2}}-\frac{1}{2z}+
\frac{t+4a(1-t)}{8\sqrt{a}z^{3/2}}-\frac{(1-t)(t+a(1-t))}{2z^2} +\mathcal{O} \left(\frac{1}{z^{5/2}} \right).\label{zeta3}
\end{align}
Here, as in the rest of the paper, all fractional powers are taken
as principal branch, that is, positive on $\R_+$, with the
branch cut along $\R_-$.
The behavior of these functions for real values of $z$ can be deduced from
Figure \ref{fig:mappingzeta} which shows the graph of $z = z(\zeta)$, $\zeta \in \mathbb R$,
and which also indicates the branches of the inverses $\zeta = \zeta_k(z)$ for real $z$.

It is straightforward to check that the discriminant of equation
\eqref{RSequation4} is equal (up to a non-vanishing factor depending only on $t$)
to the polynomial in the left hand side of \eqref{boundaryequation}.
Its three roots along with the point at infinity constitute the four branch points
of the Riemann surface $\RR$. Analyzing the signs of the coefficients in \eqref{boundaryequation}
it is easy to show that, according to the value of $t\in (0,1)$ with respect
to the critical value $t=t^{*}=a/ (a+1)$,
the following cases arise (see Figure \ref{fig:mappingzeta}):
\begin{itemize}
\item Case 1: $t\in(0,t^*)$, i.e., $k < 0$. The Riemann surface $\RR$
has three simple
real branch points $0<p<q$, plus a simple branch point at infinity.
This is the left-most graph  in Figure \ref{fig:mappingzeta}.
\item Case 2: $t\in (t^*,1)$, i.e., $k > 0$. The Riemann surface $\RR$
has three simple branch points $p<0<q$, plus a simple branch point at infinity.
This is the right-most graph in Figure \ref{fig:mappingzeta}.
\item Case 3: $t=t^{*}$, i.e., $k=0$. This is the critical case
where the Riemann surface $\RR$
has two real branch points, 0 and $q>0$, 0 being degenerate (of order 2), and
$q$ being simple. The point at infinity is still a simple branch point
of $\RR$.
\end{itemize}
These assertions coincide with the statement of Proposition \ref{prop:RiemannSurface}.
The rest of the assertions of Corollary \ref{cor:boundaryCurve} is a consequence of
straightforward although tedious computations based on equation \eqref{boundaryequation}.

\begin{figure}[htb]
\centering \begin{overpic}[scale=1]%
{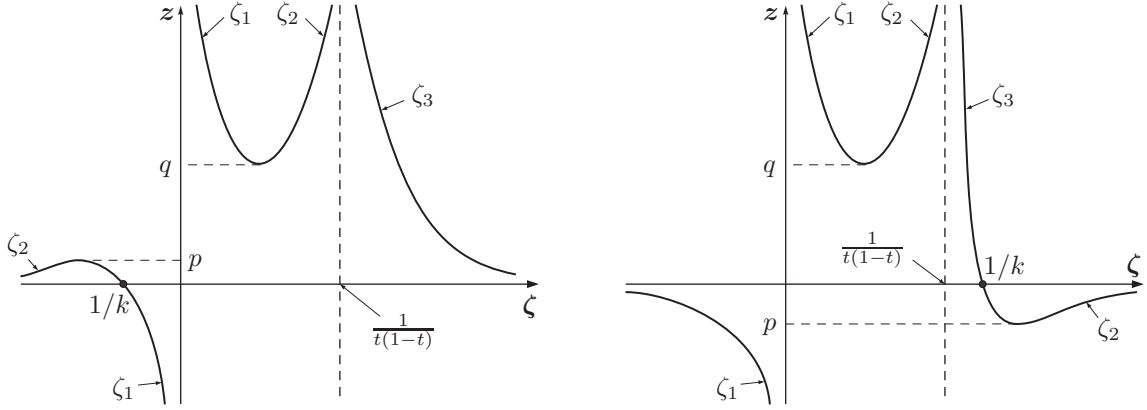}%
      \put(32,8){\small $\frac{1}{t(1-t)}$}
       \put(72,15){\small $\frac{1}{t(1-t)}$}
       \put(16.5,14){\small $p $}
       \put(65.5,8.5){\small $p $}
       \put(14,22){\small $q $}
       \put(65.5,22){\small $q $}
       \put(45,10){\small $\boldsymbol \zeta $}
       \put(96.5,13.5){\small $\boldsymbol \zeta $}
       \put(14,35){\small $\boldsymbol z $}
       \put(65.5,35){\small $\boldsymbol z $}
       \put(20,35){\small $\zeta _1$}
       \put(8,10){\small $1/k$}
       \put(84.3,13.3){\small $1/k$}
       \put(24,35){\small $\zeta _2$}
       \put(10,3){\small $\zeta _1$}
       \put(1,15.3){\small $\zeta _2$}
       \put(35.5,28){\small $\zeta _3$}
       \put(85,28){\small $\zeta _3$}
       \put(61.5,3){\small $\zeta _1$}
       \put(71.5,35){\small $\zeta _1$}
       \put(75.5,35){\small $\zeta _2$}
       \put(94,8){\small $\zeta _2$}
\end{overpic}
\caption{Plots of $z = \frac{1-k\zeta}{\zeta(1-t(1-t)\zeta)^2}$,
$\zeta \in \R$, for Case 1 ($k < 0$; left) and Case 2 ($k > 0$; right).}
\label{fig:mappingzeta}
\end{figure}

In this paper, we shall analyze Case 1 and Case 2. The sheet structure of $\RR$ is shown in Figure \ref{Riemann1}.
\begin{figure}[htb]
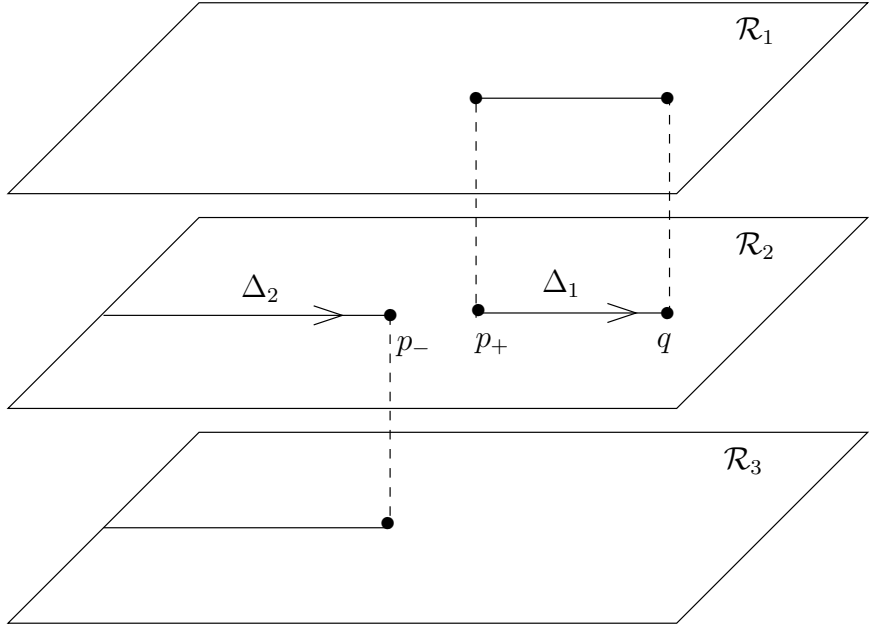

\hfil
 \input riemannsurf.pstex_t
\hfil
\caption{The Riemann surface $\RR$, $p_-=0$ and $p_+=p$ in Case 1,
$p_-=p$ and $p_+=0$ in Case 2.}
 \label{Riemann1}
\end{figure}
As before we use $p_- = \min(p,0)$ and $ p_+ = \max(p,0)$.
The sheets $\RR_{1}$ and $\RR_{2}$ are glued together along the cut
$\Delta_{1}= [p_+,q]$ and the sheets $\RR_{3}$ and $\RR_{2}$ are
glued together along the cut
$\Delta_{2}= (-\infty,p_-]$.
The functions $\zeta_{1}$, $\zeta_{2}$,
$\zeta_{3}$ are defined and analytic on the sheets $\RR_{1}$, $\RR_{2}$, and
$\RR_{3}$ respectively, and we have the jump relations:
\begin{equation}\label{jump-zeta}
\begin{aligned}
\zeta_{1\pm}(x) = \zeta_{2\mp}(x),\qquad  x\in \Delta_{1},\\
\zeta_{2\pm}(x) = \zeta_{3\mp}(x),\qquad  x\in \Delta_{2}.
\end{aligned}
\end{equation}
On $\Delta_{2}$, the function $\zeta_{1}$ is real and
$\zeta_{2}$ and $\zeta_{3}$ are complex conjugate, while on
$\Delta_{1}$, the function $\zeta_{3}$ is real and
$\zeta_{1}$ and $\zeta_{2}$ are complex conjugate, so that
\begin{equation}\label{zetaconj}
\begin{aligned}
    \zeta_{1\pm}(x)=\overline{\zeta_{2\pm} (x)},\qquad x\in \Delta_{1}, \\
    \zeta_{2\pm}(x)=\overline{\zeta_{3\pm}(x)}, \qquad x \in \Delta_2.
\end{aligned}
\end{equation}
Near the origin, one may check from \eqref{RSequation5} that, as $z\to 0$,
\begin{equation}\label{zeta-origin1}
   \zeta_1(z) =  \frac{1}{k} + \mathcal{O}\left({z}\right),
    \quad
   \zeta_2(z) =  -\frac{k_1}{z^{1/2}} +k_2+ \mathcal{O}\left(z^{1/2}\right),
\quad
   \zeta_3(z) = \frac{k_1}{z^{1/2}} + k_2 + \mathcal{O}\left(z^{1/2}\right),
\end{equation}
in Case 1 ($p_-=0$), while
\begin{equation}\label{zeta-origin2}
   \zeta_1(z) =  -\frac{ k_1}{(-z)^{1/2}} + k_2 + \mathcal{O}\left(z^{1/2}\right),
\quad
   \zeta_2(z) =  \frac{ k_1}{(-z)^{1/2}}+k_2 + \mathcal{O}\left(z^{1/2}\right),
\quad
   \zeta_3(z) = \frac{1}{k} + \mathcal{O}\left({z}\right),
\end{equation}
in Case 2 ($p_+=0$), where we have set $k = (1-t)(t-a(1-t))$ as before, and
\[
 k_1=-\frac{\sqrt{|k|}}{t(1-t)},\qquad
 k_2=\frac{1}{t(1-t)}-\frac{1}{2k}.
\]

Next, we introduce the integrals of the $\zeta$-functions,
\begin{align}\label{def-lambda1}
   \lambda_{1}(z) & = \int_{q}^z\zeta_1(s)ds, \\
\label{def-lambda2}
   \lambda_{2}(z) & = \int_{q}^z\zeta_2(s)ds, \\
\label{def-lambda3}
   \lambda_{3}(z) & = \int_{p_-}^z\zeta_3(s)ds+\lambda_{2_{-}} (p_-).
\end{align}
The functions $\lambda_{1}$ and $\lambda_{2}$ are defined and analytic in
$\C\setminus (-\infty, q]$, and the function $\lambda_{3}$ is defined and
analytic in $\C\setminus (-\infty,p_-]$.
By \eqref{zeta-origin1}--\eqref{zeta-origin2},
these functions are bounded in the neighborhood of each branch point $p_-$, $p_+$, $q$. By \eqref{zetaconj},
\begin{equation}\label{lambdaconj}
    \lambda _{1\pm}(x) = \overline{ \lambda _{2\pm}(x)}\,, \quad x \in \Delta_1\,.
\end{equation}

From \eqref{zeta1}--\eqref{zeta3}, it follows that,
as $z\to\infty$,
\begin{align}
\lambda_{1} (z) & = \log z+\ell_1 -\frac{(1-t)(t+a(1-t))}{z}+
\mathcal{O} \left(\frac{1}{z^2} \right), \label{lambda1}\\
\lambda_{2} (z) & = \frac{z}{t (1-t)}-\frac{2\sqrt{az}}{t}
-\frac{1}{2}\log z+\ell_{2} \nonumber \\
& \qquad + \frac{t+4a(1-t)}{4\sqrt{az}}  + \frac{(1-t)(t+a(1-t))}{2z}
+\mathcal{O} \left(\frac{1}{z^{3/2}} \right),\label{lambda2}\\
\lambda_{3} (z) & = \frac{z}{t (1-t)}+\frac{2\sqrt{az}}{t}
-\frac{1}{2}\log z+\ell_{3} \nonumber \\
& \qquad -\frac{t+4a(1-t)}{4\sqrt{az}}  + \frac{(1-t)(t+a(1-t))}{2z}
+\mathcal{O} \left(\frac{1}{z^{3/2}} \right),\label{lambda3}
\end{align}
where $\ell_{j}$, $j=1,2,3$, are certain integration constants.

We will need the following relation between $\ell_2$ and $\ell_3$.
\begin{lemma} \label{lem:ell23}
We have $\ell_3 = \ell_2 + \pi i$.
\end{lemma}
\begin{proof}
By the definition of $\lambda_2$ and $\lambda_3$ we have
for $z=-R$ on the lower side of the cut $\Delta_2$,
\begin{align*}
    (\lambda_2 - \lambda_3)(z) & =
    \int_{(p_-)_-}^z \zeta_2(s) ds - \int_{p_-}^z \zeta_3(s) ds
    = \int_{p_-}^z (\zeta_{2-} - \zeta_{3-})(s) ds  \\
    & = \int_{p_-}^{-R} (\zeta_{3+} - \zeta_{3-})(s) ds
    = \int_{\gamma_{R,\varepsilon}} \zeta_3(s) ds
    \end{align*}
where $\gamma_{R,\varepsilon}$ is a contour from $-R$ to
$p_--\varepsilon$ on the lower side
of $\Delta_2$ continued with the circle around $p_-$
of radius $\varepsilon$, and then back from $p_- - \varepsilon$
to $-R$ on the upper side of $\Delta_2$. Here $\varepsilon > 0$
is taken sufficiently small.
Then we write
\begin{align*}
    (\lambda_2 - \lambda_3)(z) & =
    \int_{\gamma_{R, \varepsilon}}
    \left(\zeta_3(s) - \frac{1}{t(1-t)} - \frac{\sqrt{a}}{t(s-p_-)^{1/2}} + \frac{1}{2(s-p_-)}\right) ds \\
    & \quad + \int_{\gamma_{R,\varepsilon}} \frac{\sqrt{a}}{t(s-p_-)^{1/2}} ds
    - \int_{\gamma_{R,\varepsilon}} \frac{1}{2(s-p_-)} ds.
    \end{align*}
Since the integrand of the first term on the right-hand side is analytic in $\mathbb C \setminus (-\infty,p_-]$
and behaves as $\mathcal{O}(s^{-3/2})$ as $s \to \infty$ (due to \eqref{zeta3}), it follows that the first term tends
to $0$ as $R \to \infty$. For the second term we have
\[ \int_{\gamma_{z,\varepsilon}} \frac{\sqrt{a}}{t(s-p_-)^{1/2}} ds + \frac{4\sqrt{az}}{t} \to 0 \]
as $R \to \infty$, $\varepsilon \to 0$, and the third term is just simply $-\pi i$.
Thus
\begin{equation}\label{labdasatinfty}
     (\lambda_2 - \lambda_3)(z) = -\frac{4\sqrt{az}}{t} - \pi i + o(1)
\end{equation}
which gives the lemma in view of \eqref{lambda2} and \eqref{lambda3}.
\end{proof}

From \eqref{jump-zeta}, the definitions of the $\lambda$-functions, and the
relations
\[
\oint \zeta_{1} (s)ds = 2\pi i,\quad
\oint \zeta_{2} (s)ds = -2\pi i,
\]
(which follow by a residue calculation from the expansion \eqref{zeta1} of $\zeta_1$ at infinity),
where the integrals are taken on a positively oriented closed contour
around $\Delta_{1}$, we check that the following jump relations hold true,
\begin{equation}\begin{aligned}\label{jumplamb}
&\lambda_{1+} (x) =\lambda_{1-} (x) +2\pi i, \quad x\in (-\infty,p_+],\\
&\lambda_{1\pm} (x) =\lambda_{2\mp} (x),  \quad x\in \Delta_{1}= [p_+,q],
\\
&\lambda_{2+}(x)=\lambda_{2-}(x)-2\pi i,  \quad x\in[p_-,p_+],\\
&\lambda_{2+}(x)=\lambda_{3-}(x)-2\pi i,  \quad
\lambda_{2-}(x)=\lambda_{3+}(x), \quad x\in\Delta_{2}= (-\infty,p_-].\\
\end{aligned}
\end{equation}
A straightforward consequence of these relations is the following statement:
\begin{lemma}
\label{lem:analyticity}
Functions $e^{\lambda_{1}(z)}$,
$e^{\lambda_{2}(z)}$ and $e^{\lambda_{3}(z)}$ are analytic and
single--valued outside of $\Delta_1$, $\Delta_1 \cup \Delta_2$,
and $\Delta_{2}$ respectively.
\end{lemma}

In the sequel, we will need to compare the real parts of the
$\lambda$-functions on $\R$ and in neighborhoods of $\Delta_{2}$ and
$\Delta_{1}$. This is the aim of the next lemma.
\begin{lemma} \label{lem:Relambdajs}
\label{signs}
{\rm (a)} The following inequalities hold true,
\begin{align*}
\Re \lambda_{1} <\Re \lambda_{2}\quad &\text{on $(0,p_+)$ in Case 1,}
\\
\Re \lambda_{3} <\Re \lambda_{2}\quad &\text{on $(p_-,0)$ in Case 2,}
\\
\Re \lambda_{1}<\Re \lambda_{2}\quad &\text{on } (q,+\infty).
\end{align*}
{\rm (b)} The open interval $(p_+,q)$ has a neighborhood
$U_{1}$ in the complex plane such that
\begin{equation*}
\Re \lambda_{2} (z) < \Re \lambda_{1} (z),\quad z\in
U_{1}\setminus (p_+,q).
\end{equation*}
{\rm (c)} The open interval $(-\infty,p_-)$ has a neighborhood
$U_{2}$ in the complex plane such that
\begin{equation*} 
\Re \lambda_{2} (z) < \Re \lambda_{3} (z), \quad z\in
U_{2}\setminus (-\infty,p_-).
\end{equation*}
The neigbhorhood $U_2$ is unbounded and contains a full neighborhood
of infinity.
\end{lemma}
\begin{proof}
It is easy to check (see also the left-most picture in
Figure \ref{fig:mappingzeta}) that, in Case 1, $\zeta_{2} (s) <\zeta_{1} (s)
<\zeta_{3} (s)$
for $s\in (0,p_+)$. Hence, from the definitions of the functions
$\lambda_{1}$ and $\lambda_{2}$ and \eqref{lambdaconj}, we may conclude that
$\Re \lambda_{1} <\Re \lambda_{2}$ on $(0,p_+)$. In Case 2, we have
$\zeta_{1} (s) <\zeta_{3} (s)
<\zeta_{2} (s)$
for $s\in (p_-,0)$ (see right-most picture in Figure \ref{fig:mappingzeta}). Moreover, since
\[
(\lambda_{3}-\lambda_{2}) (z)=\int_{(p_-)_+}^{z} (\zeta_{3}-\zeta_{2})
(s)ds+ \lambda_{2-} (p_-)-\lambda_{2+} (p_-),
\]
we get, along with the third jump relation in \eqref{jumplamb}, that
$\Re \lambda_{3} <\Re \lambda_{2}$ on $(p_-,0)$. Finally, on
$(q,\infty)$, $\zeta_{1} (s) <\zeta_{2} (s)
<\zeta_{3} (s)$ so that the third inequality in assertion (a) is
indeed satisfied.

On the $+$ side of $\Delta_{1}$, $(\lambda_{2}-\lambda_{1})_{+}$ is purely
imaginary. Its derivative
$(\zeta_{2}-\zeta_{1})_{+} (z)$ is purely imaginary as
well. By inspection of the Riemann surface $\RR$, it can be shown that
this imaginary part is actually positive. Hence by the Cauchy-Riemann
equations the real part of $(\lambda_{2}-\lambda_{1})(z)$ decreases as
$z$ moves
into the upper half-plane, so that $\Re\lambda_{2} (z) <\Re\lambda_{1}
(z) $ for $z$
near $\Delta_{1}$ in the upper half-plane. Similarly, $\Re\lambda_{2} (z)
<\Re\lambda_{1} (z)$ for $z$
near $\Delta_{1}$ in the lower half-plane, which shows assertion
(b).

The proof of assertion (c) is similar. In order to see that $U_2$
contains a full neighborhood of infinity, it is sufficient to use \eqref{labdasatinfty}, where $\sqrt{a}/t >0$.
\end{proof}

A consequence of  Lemma \ref{lem:Relambdajs} is that we may (and do)
assume that contours $\Delta_2^\pm$,
defined in Section \ref{section3} (and depicted in Figure \ref{fig:first_transformation})
meet the real line at the branch point $p_-$, and lie in the neighborhood $U_2$ of $\Delta_2$
where $\Re (\lambda _2-\lambda _3)<0$.

Using functions $\lambda _j$, we can now define the second transformation
of the RH problem.
\begin{equation}\label{Udef}
U(z)= C_2   X(z)
\begin{pmatrix}
e^{-n \lambda_{1} (z)} & 0 & 0\\
0 & e^{-n (\lambda_{2} (z)- \frac{z}{t (1-t)})} & 0\\
0 & 0 & e^{-n (\lambda_{3} (z)- \frac{z}{t (1-t)})}
\end{pmatrix},
\end{equation}
where
\begin{equation} \label{C2def}
    C_2 = \begin{cases}
    \begin{pmatrix} 1 & 0 & 0 \\
    0 & 1 & -in \frac{t+4a(1-t)}{4\sqrt{a}} \\
    0 & 0 & 1 \end{pmatrix}
    L^{n}, & \text{if $n$ is even}, \\
    \begin{pmatrix} 1 & 0 & 0 \\
    0 & n \frac{t+4a(1-t)}{4\sqrt{a}} & i  \\
    0 &  i & 0 \end{pmatrix}
    L^{n}, & \text{if $n$ is odd},
    \end{cases}
\end{equation}
and $L$ is the constant diagonal matrix
\begin{equation} \label{Ldef}
L=\begin{pmatrix}
e^{\ell_{1}} & 0 & 0\\
0 & e^{\ell_{2}} & 0\\
0 & 0 & e^{\ell_{3}}
\end{pmatrix}.
\end{equation}
By Lemma \ref{lem:analyticity}, the matrix-valued function $U$ is analytic in $\C\setminus\R$.

Making use of the jump relations \eqref{Xjump1}--\eqref{Xjump4}
for $X$ and the definition \eqref{Udef} one easily gets
that the following jump relations for $U$:
\begin{equation*}
U_+(x) = U_-(x)
\begin{pmatrix}
    e^{n (\lambda_{1-} (x) -\lambda_{1+} (x)) }& x^{\alpha}e^{n (\lambda_{1-} (x) -\lambda_{2+} (x)) } & 0 \\
    0 & e^{n (\lambda_{2-} (x) -\lambda_{2+} (x)) } & 0 \\
    0 & 0 & e^{n (\lambda_{3-} (x) -\lambda_{3+} (x)) }
    \end{pmatrix}
\end{equation*}
for $x \in \mathbb R_+$,
\begin{equation*}
U_+(x) = U_-(x)
\begin{pmatrix}
    e^{n (\lambda_{1-} (x) -\lambda_{1+} (x)) }& 0 & 0 \\
    0 & e^{n (\lambda_{2-} (x) -\lambda_{2+} (x)) } & 0 \\
    0 & |x|^{\alpha}e^{n (\lambda_{3-} (x) -\lambda_{2+} (x)) } & e^{n (\lambda_{3-} (x) -\lambda_{3+} (x)) }
    \end{pmatrix}
\end{equation*}
for $x \in (p_-,0)$ (in Case 2),
\begin{equation*}
U_+(x) = U_-(x)
\begin{pmatrix}
    e^{n (\lambda_{1-} (x) -\lambda_{1+} (x)) }&    0 & 0 \\
    0 & 0 & -|x|^{-\alpha}e^{n (\lambda_{3-} (x) -\lambda_{2+} (x)) } \\
    0 & |x|^{\alpha}e^{n (\lambda_{3-} (x) -\lambda_{2+} (x)) } & 0
    \end{pmatrix}
\end{equation*}
for $x \in (-\infty, p_-)$, and
\begin{equation*}
U_+(z) = U_-(z)
 \begin{pmatrix}
    1 & 0 & 0 \\
    0 & 1 & e^{\pm \alpha \pi i} z^{-\alpha} e^{n(\lambda_2-\lambda_3)(z)} \\
    0 & 0 & 1
    \end{pmatrix}, \quad z \in \Delta_2^{\pm}.
\end{equation*}

Using the jump relations \eqref{jumplamb}, one checks easily that
the jump properties for $U$ simplify to the ones stated in
the following proposition with the just introduced notation.

\begin{proposition} \label{prop:RHforU}
The matrix-valued function $U(z)$ defined by \eqref{Udef}
is the unique solution of the following RH problem.
\begin{enumerate}
\item $U(z)$ is analytic in $\mathbb{C} \setminus (\R \cup \Delta_2^{\pm})$.
\item $U(z)$ possesses continuous boundary values at $(\R \cup \Delta_2^{\pm}) \setminus \{p_-, 0 \}$, and
\begin{equation}
\label{Ujump1}
U_+(x)=U_-(x)
    \begin{pmatrix}
1 & 0 & 0\\
0 & 0 & - |x|^{-\alpha} \\
0 & |x|^{\alpha}    & 0
    \end{pmatrix}, \quad x\in\Delta_{2}= (-\infty,p_-),
\end{equation}
\begin{equation}  \label{Ujump2}
U_+(x)=U_-(x) \times \left\{\begin{array}{ll}
    I +  x^\alpha e^{n (\lambda_{1}-\lambda_{2}) (x)} E_{1 2} , & x\in (0,p_+)\text{ in Case 1},\\[1mm]
   I + |x|^{\alpha}e^{n (\lambda_{3}-\lambda_{2})  (x)} E_{3 2}  ,  & x\in (p_-,0)\text{ in Case 2},
\end{array}\right.
\end{equation}
\begin{equation}  \label{Ujump3}
U_+(x)=U_-(x)
    \begin{pmatrix}
    e^{n (\lambda_{2}-\lambda_{1})_{+} (x)} & x^\alpha & 0 \\
        0 & e^{n (\lambda_{2}-\lambda_{1})_{-} (x)} & 0 \\
        0 & 0                    & 1
    \end{pmatrix}, \quad x\in\Delta_{1}= (p_+,q).
\end{equation}
\begin{equation}  \label{Ujump4}
U_+(x)=U_-(x)
\left( I + x^\alpha e^{n (\lambda_{1}-\lambda_{2}) (x)} E_{1 2}\right),
\quad x\in (q,\infty), \end{equation}
\begin{equation} \label{Ujump5}
U_+(z) = U_-(z)
 \left( I + e^{\pm \alpha \pi i} z^{-\alpha} e^{n(\lambda_2-\lambda_3)(z)} E_{2 3}\right),    \quad z \in \Delta_2^{\pm}.
\end{equation}
\item As $z\to \infty$ we have
\begin{multline} \label{Uasymptotics}
U(z) =
   \left(I + \mathcal{O}\left(\frac{1}{z}\right) \right)
    \begin{pmatrix}
    1 & 0 & 0\\
    0 & z^{1/4} & 0 \\
    0 & 0 &  z^{-1/4}
    \end{pmatrix}
     \begin{pmatrix}
    1 & 0 & 0 \\
    0 & \frac{1}{\sqrt{2}} & \frac{1}{\sqrt{2}}i \\
    0 & \frac{1}{\sqrt{2}}i & \frac{1}{\sqrt{2}}
    \end{pmatrix}
    \begin{pmatrix}
    1 & 0 & 0 \\
    0 & z^{\alpha/2} & 0 \\
    0 & 0 & z^{-\alpha/2}
    \end{pmatrix},
\end{multline}
\item $U(z)$ is bounded at $p_-$ if $p_-<0$, and has the same behavior as $X(z)$ at
the origin, see \eqref{Yedge} and \eqref{Xedge2}.
\end{enumerate}
\end{proposition}
\begin{proof}
Jumps \eqref{Ujump1}--\eqref{Ujump5} are result of straightforward calculations and Lemma \ref{lem:analyticity}.

For the proof of the asymptotic condition in item 3 we note that
property \eqref{Xasymptotics} of $X$ and the asymptotic behaviors \eqref{lambda1}-\eqref{lambda3}
of the $\lambda_j$-functions yield
\begin{multline} \label{Xtimeslambda}
    X(z)
\begin{pmatrix}
e^{-n \lambda_{1} (z)} & 0 & 0\\
0 & e^{-n (\lambda_{2} (z)- \frac{z}{t (1-t)})} & 0\\
0 & 0 & e^{-n (\lambda_{3} (z)- \frac{z}{t (1-t)})}
\end{pmatrix} \\
=
   \left(I + \mathcal{O}\left(\frac{1}{z}\right) \right)
    \begin{pmatrix}
    1 & 0 & 0\\
    0 & z^{(-1)^n/4} & 0 \\
    0 & 0 &  z^{-(-1)^n/4}
    \end{pmatrix}
     \begin{pmatrix}
    1 & 0 & 0 \\
    0 & \frac{1}{\sqrt{2}} & \frac{1}{\sqrt{2}} i \\
    0 & \frac{1}{\sqrt{2}} i & \frac{1}{\sqrt{2}}
    \end{pmatrix}
    \begin{pmatrix}
    1 & 0 & 0 \\
    0 & z^{\alpha/2} & 0 \\
    0 & 0 & z^{-\alpha/2}
    \end{pmatrix} \\ \times
    L^{-n}
\begin{pmatrix}
            1+ \mathcal{O}(\frac{1}{z})  & 0 & 0 \\
            0 & 1 - \frac{c_1}{z^{1/2}} + \frac{c_2}{z}  + \mathcal{O}(\frac{1}{z^{3/2}}) & 0 \\
            0 & 0 &  1 + \frac{c_1}{z^{1/2}} + \frac{c_2}{z} + \mathcal{O}(\frac{1}{z^{3/2}})
            \end{pmatrix}
\end{multline}
as $z \to \infty$,
where
\[ c_1 = n\frac{t+4a(1-t)}{4\sqrt{a}}, \quad
     c_2 = \frac{c_1^2 - nk}{2}. \]

If $n$ is even, then by  Lemma \ref{lem:ell23}
we have that $L^{-n}$ commutes with all
matrices before it. The last matrix in the right-hand side of \eqref{Xtimeslambda}
can be moved to the left as in the proof of Proposition \ref{prop:RHforX}.
The result is that \eqref{Xtimeslambda} is equal to
\begin{equation*}
L^{-n} \begin{pmatrix} 1 & 0 & 0 \\
    0 & 1 & i c_1 \\
    0 & 0 & 1 \end{pmatrix}
   \left(I + \mathcal{O}\left(\frac{1}{z}\right) \right)
    \begin{pmatrix}
    1 & 0 & 0\\
    0 & z^{1/4} & 0 \\
    0 & 0 &  z^{-1/4}
    \end{pmatrix}
     \begin{pmatrix}
    1 & 0 & 0 \\
    0 & \frac{1}{\sqrt{2}} & \frac{1}{\sqrt{2}} i \\
    0 & \frac{1}{\sqrt{2}} i & \frac{1}{\sqrt{2}}
    \end{pmatrix}
    \begin{pmatrix}
    1 & 0 & 0 \\
    0 & z^{\alpha/2} & 0 \\
    0 & 0 & z^{-\alpha/2}
    \end{pmatrix},
\end{equation*}
as $z \to \infty$. Then \eqref{Uasymptotics} follows by
the definition \eqref{Udef}--\eqref{Ldef} of $U$.

If $n$ is odd, then by Lemma \ref{lem:ell23}, we have that
\[ L^{-n} \diag(1,1,-1) = \diag\left(e^{-n \ell_1}, e^{-n \ell_2}, e^{-n \ell_2}\right) \]
commutes with all factors before it in \eqref{Xtimeslambda}.
The result now is that \eqref{Xtimeslambda} is equal to
\begin{equation*}
L^{-n}
    \begin{pmatrix} 1 & 0 & 0 \\ 0 & 0 & -i \\ 0 & - i & c_1 \end{pmatrix}
   \left(I + \mathcal{O}\left(\frac{1}{z}\right) \right)
    \begin{pmatrix}
    1 & 0 & 0\\
    0 & z^{1/4} & 0 \\
    0 & 0 &  z^{-1/4}
    \end{pmatrix}
     \begin{pmatrix}
    1 & 0 & 0 \\
    0 & \frac{1}{\sqrt{2}} & \frac{1}{\sqrt{2}} i \\
    0 & \frac{1}{\sqrt{2}} i & \frac{1}{\sqrt{2}}
    \end{pmatrix}
    \begin{pmatrix}
    1 & 0 & 0 \\
    0 & z^{\alpha/2} & 0 \\
    0 & 0 & z^{-\alpha/2}
    \end{pmatrix},
\end{equation*}
as $z \to \infty$, and again \eqref{Uasymptotics} follows by
the definition \eqref{Udef}--\eqref{Ldef} of $U$.

The behavior of $U$ at the origin given in item 4 follows from
the corresponding behavior of $X$, and the fact that the
$\lambda_j$ functions all remain bounded near the origin.
\end{proof}

It follows from Lemma \ref{lem:Relambdajs} that
the jump matrices in \eqref{Ujump2}, \eqref{Ujump4}, and \eqref{Ujump5}
tend to the  identity matrix $I$ as $n \to \infty$ at an exponential rate.
Moreover, by
\eqref{jumplamb}, $(\lambda_{2}-\lambda_{1})_{+} =-
(\lambda_{2}-\lambda_{1})_{-}$ is purely imaginary on
$\Delta_{1}$, so that the first two diagonal elements of the jump matrices in
\eqref{Ujump3} are oscillatory. In the third transformation we open
a lens around $\Delta_1$ and we turn the oscillatory entries into
exponentially small entries.

\section{Third transformation of the RH problem}\label{section5}

Here, the goal is to transform the oscillatory
diagonal terms in the jump matrices on $\Delta_{1}$
into exponentially small off-diagonal terms. This we do by
opening a lens around
$\Delta_{1}$, see Figure
\ref{deform3}. We assume that the lens is contained
in $U_{1}$, see Lemma \ref{signs}.

\begin{figure}[htb]
\centering \begin{overpic}[scale=1.2]%
{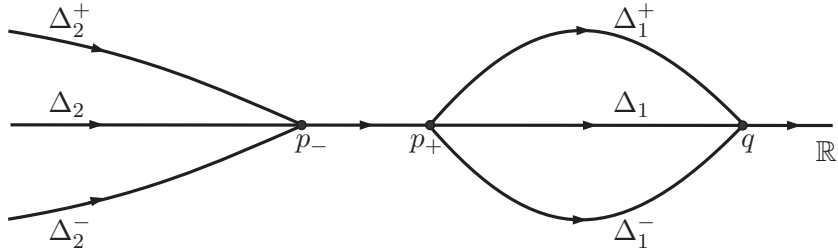}%
      \put(95,14){$\R $}
       \put(8,29){$\Delta_2^+ $}
       \put(8,5){$\Delta_2^- $}
       \put(8,19.5){$\Delta_2 $}
       \put(72,29){$\Delta_1^+ $}
       \put(72,5){$\Delta_1^- $}
       \put(72,19.5){$\Delta_1 $}
       \put(36,15.5){$p_-$}
       \put(49,15.5){$p_+$}
       \put(86.5,15.5){$q$}
\end{overpic}
\caption{Deformation of contours around $\Delta_{1}$.}
\label{deform3}
\end{figure}
We use the following factorizations of the $2\times 2$ non-trivial
block of the jump matrix in \eqref{Ujump3}, \begin{multline*}
\begin{pmatrix}
e^{n (\lambda_{2}-\lambda_{1})_{+} (x)} & x^\alpha  \\
        0 & e^{n (\lambda_{2}-\lambda_{1})_{-} (x)}
\end{pmatrix}\\
=
\begin{pmatrix}
1 & 0\\
x^{-\alpha}e^{n (\lambda_{2}-\lambda_{1})_{-} (x)} & 1\\
\end{pmatrix}
\begin{pmatrix}
0 & x^{\alpha}\\
-x^{-\alpha} & 0
\end{pmatrix}
\begin{pmatrix}
1 & 0\\
x^{-\alpha}e^{n (\lambda_{2}-\lambda_{1})_{+} (x)} & 1
\end{pmatrix}.
\end{multline*}
We set
\begin{equation} \label{T1def}
T(z)=U(z)
\, \left( I   \mp z^{-\alpha}e^{n (\lambda_{2}-\lambda_{1}) (z)} E_{21} \right),
\end{equation}
for $z$ in the domain bounded by $\Delta_{1}^{\pm}$ and $\Delta_1$,
and we let
\begin{equation} \label{T2def}
    T(z) =U(z)
    \end{equation}
for $z$  outside of the lens.

Straightforward calculations show that $T(z)$ is a solution of
the following Riemann--Hilbert problem which is stated in the
next proposition.

\begin{proposition} \label{prop:RHforT}
The matrix-valued function $T(z)$ is the unique solution of the
following RH problem.
\begin{enumerate}
\item $T$ is analytic in $\mathbb C \setminus ( \mathbb R \cup \Delta_1^{\pm} \cup \Delta_2^{\pm})$.
\item $T$ has a jump $T_+(z)=T_-(z)\, j_{T} (z)$ on each of the oriented contours shown in Figure
\ref{deform3}. They are given by
{\allowdisplaybreaks
\begin{align*}
j_{T}(x)&=
    \begin{pmatrix}
    1 & 0                  & 0 \\
0 & 0 & -|x|^{-\alpha}\\
0 & |x|^{\alpha} & 0
    \end{pmatrix}, \quad x\in\Delta_{2}, \\
j_{T}(z)&=
   I + e^{\pm \alpha \pi i} z^{-\alpha}e^{n (\lambda_{2}-\lambda_{3})(z)} E_{23}, \quad z\in\Delta_{2}^{\pm}, \\
j_{T} (x)& =\left\{\begin{array}{ll}
 I+  x^\alpha e^{n (\lambda_{1}-\lambda_{2}) (x)} E_{12}    , & x\in (0,p_+)\text{ in Case 1},\\[2mm]
   I+ |x|^{\alpha}e^{n (\lambda_{3}-\lambda_{2})  (x)} E_{32}  , & x\in (p_-,0)\text{ in Case 2},
\end{array}\right.\\
j_{T} (x) &=
    \begin{pmatrix}
0 &  x^{\alpha}          & 0    \\
-x^{-\alpha} & 0 & 0 \\
0 &  0 & 1\\
    \end{pmatrix}, \quad  x\in\Delta_1,
\\
j_{T}(z)&=
    I + z^{-\alpha} e^{n (\lambda_{2}-\lambda_{1})(z)} E_{2 1}, \quad z\in\Delta_{1}^{\pm}, \\
j_{T}(x)&=
    I + x^\alpha e^{n (\lambda_{1}-\lambda_{2}) (x)} E_{1 2},
    \quad x\in (q,+\infty).
\end{align*}}
\item As $z \to \infty$, we have
\begin{equation} \label{Tinfty}
    T(z) =
   \left(I + \mathcal{O}\left(\frac{1}{z}\right) \right)
    \begin{pmatrix}
    1 & 0 & 0\\
    0 & z^{1/4} & 0 \\
    0 & 0 &  z^{-1/4}
    \end{pmatrix}
     \begin{pmatrix}
    1 & 0 & 0 \\
    0 & \frac{1}{\sqrt{2}} & \frac{1}{\sqrt{2}}i \\
    0 & \frac{1}{\sqrt{2}}i & \frac{1}{\sqrt{2}}
    \end{pmatrix}
    \begin{pmatrix}
    1 & 0 & 0 \\
    0 & z^{\alpha/2} & 0 \\
    0 & 0 & z^{-\alpha/2}
    \end{pmatrix}.
\end{equation}
\item For $-1<\alpha<0$, $T(z)$ behaves near the origin like:
\begin{equation*} 
T(z)=\mathcal{O}\begin{pmatrix}
            1 & |z|^{\alpha} & 1 \\
            1 & |z|^{\alpha} & 1 \\
            1 & |z|^{\alpha} & 1
            \end{pmatrix}, \quad \text{as }z\to 0.
\end{equation*}
For $\alpha=0$, $T(z)$ behaves near the origin like:
\begin{equation*} 
T(z)=\mathcal{O}\begin{pmatrix}
             1 & \log |z| & 1 \\
             1 & \log |z| & 1 \\
             1 & \log |z| & 1
            \end{pmatrix}, \quad \text{as }z\to 0\text{ outside the
lens that ends in 0},
\end{equation*}
and
\begin{equation*} 
T(z)=\begin{cases}
 \mathcal{O}\begin{pmatrix}
            1 & \log |z| & \log |z| \\
            1 & \log |z| & \log |z| \\
            1 & \log |z| & \log |z|
            \end{pmatrix}, & \text{as }z\to 0 \text{ inside the lens
around $\Delta_{2}$ in Case 1},\\[5mm]
\mathcal{O}\begin{pmatrix}
            \log |z| & \log |z| & 1 \\
            \log |z| & \log |z| & 1 \\
            \log |z| & \log |z| & 1
            \end{pmatrix}, & \text{as }z\to 0 \text{ inside the lens
around $\Delta_{1}$ in Case 2}.
\end{cases}
\end{equation*}
For $\alpha > 0$, $T(z)$ behaves near the origin like:
\begin{equation*} 
T(z)=\mathcal{O}\begin{pmatrix}
             1 & 1 & 1 \\
             1 & 1 & 1 \\
             1 & 1 & 1
            \end{pmatrix}, \quad \text{as }z\to 0\text{ outside the
lens that ends in 0},
\end{equation*}
and\begin{equation*} 
T(z)=\left\{\begin{array}{ll}
\mathcal{O}\begin{pmatrix}
            1 & 1 & |z|^{-\alpha} \\
            1 & 1 & |z|^{-\alpha}\\
            1 & 1 & |z|^{-\alpha}
            \end{pmatrix}, & \text{as }z\to 0 \text{ inside the lens
around $\Delta_{2}$ in Case 1},\\
\mathcal{O}\begin{pmatrix}
             |z|^{-\alpha} & 1 & 1 \\
             |z|^{-\alpha} & 1 & 1 \\
             |z|^{-\alpha} & 1 & 1
            \end{pmatrix}, & \text{as }z\to 0 \text{ inside the lens
around $\Delta_{1}$ in Case 2}.
\end{array}
\right.
\end{equation*}
\item
$T$ is bounded at $p$ and $q$.
\end{enumerate}
\end{proposition}
\begin{proof}
All properties follow by straightforward calculations.

Because of the prescribed behavior at the origin, it is not immediate
that the RH problems for $U$ and $T$ are equivalent. Reasoning as in
\cite[Lemma 4.1]{KMVAV} we can still show that they are.  Thus in
particular the solution of the RH problem for $T$ is unique.
\end{proof}

\section{Model RH problem for the global parametrix}\label{section6}

In view of Lemma \ref{signs} the jump matrices in the RH problem for
$T$ all tend to the identity matrix exponentially
fast as $n\to\infty$, except for the jump matrices on  $\Delta_{1}$ and $\Delta_{2}$.
Thus we expect that the main
contribution to the asymptotic behavior of $T$ is described by a solution
$N_{\alpha}$ of the following model RH problem.
\begin{enumerate}
\item $N_{\alpha}$ is analytic in $\C\setminus(\Delta_{1}\cup\Delta_{2})$.
\item $N_{\alpha}$ has continuous boundary values on $\Delta_1\cup \Delta_2$, satisfying the following jump relations:
\begin{align}
\label{eq:Njump1}
N_{\alpha +}(x) & =N_{\alpha -}(x)
    \begin{pmatrix}
    0            & x^\alpha & 0 \\
        -x^{-\alpha} & 0        & 0 \\
        0            & 0        & 1
    \end{pmatrix}, \quad x\in \Delta_1, \\
\label{eq:Njump2}
N_{\alpha +}(x) & =N_{\alpha -}(x)
\begin{pmatrix}
    1            & 0 & 0 \\
        0 & 0        & -|x|^{-\alpha} \\
        0            & |x|^{\alpha}        & 0
    \end{pmatrix}, \quad x\in \Delta_2.
    \end{align}
\item As $z \to \infty$, $z\in \C\setminus \Delta_2$,
\begin{equation} \label{eq:Nasymptotics}
    N_{\alpha}(z) =
   \left(I + \mathcal{O}\left(\frac{1}{z}\right) \right)
    \begin{pmatrix}
    1 & 0 & 0\\
    0 & z^{1/4} & 0 \\
    0 & 0 &  z^{-1/4}
    \end{pmatrix}
    \begin{pmatrix}
    1 & 0 & 0 \\
    0 & \frac{1}{\sqrt{2}} & \frac{1}{\sqrt{2}} i \\
    0 & \frac{1}{\sqrt{2}}i & \frac{1}{\sqrt{2}}
    \end{pmatrix}
    \begin{pmatrix}
    1 & 0 & 0 \\
    0 & z^{\alpha/2} & 0 \\
    0 & 0 & z^{-\alpha/2}
    \end{pmatrix}.
\end{equation}
\end{enumerate}

The asymptotic condition at infinity looks a bit awkward.
However it is consistent with the jump on $\Delta_2$ since
it may be checked that
\[  B(z) = \begin{pmatrix}
     z^{1/4} & 0 \\
     0 &  z^{-1/4}
    \end{pmatrix}
    \frac{1}{\sqrt{2}} \begin{pmatrix}
     1 & i \\
     i & 1
    \end{pmatrix}
    \begin{pmatrix}
    z^{\alpha/2} & 0 \\
    0 & z^{-\alpha/2}
    \end{pmatrix}, \]
satisfies
\[ B_+(x) = B_-(x) \begin{pmatrix} 0 & -|x|^{-\alpha} \\
    |x|^{\alpha} & 0 \end{pmatrix},
        \quad x \in (-\infty, 0). \]

We solve the RH problem for $N_{\alpha}$ in two steps. First
we solve it for the special value $\alpha = 0$ and then we use this
to solve it for general values of $\alpha$. In both steps we
will use the  mapping function \eqref{RSequation5}
\[ z = \frac{1-k\zeta}{\zeta(1-c\zeta)^2} \]
with
\begin{equation} \label{definitionKandC}
    c = t(1-t) \qquad \text{and} \qquad k = (1-t)(t-a(1-t)),
    \end{equation}
 which gives
a bijection between the Riemann surface $\mathcal R$ and the extended $\zeta$-plane.
The mapping properties are summarized in Figure
\ref{fig:Riemannmapping} for the two cases (Case 1 in the upper part
and Case 2 in the lower part of the figure).

\begin{figure}[h]
\centering \begin{overpic}[scale=0.9]%
{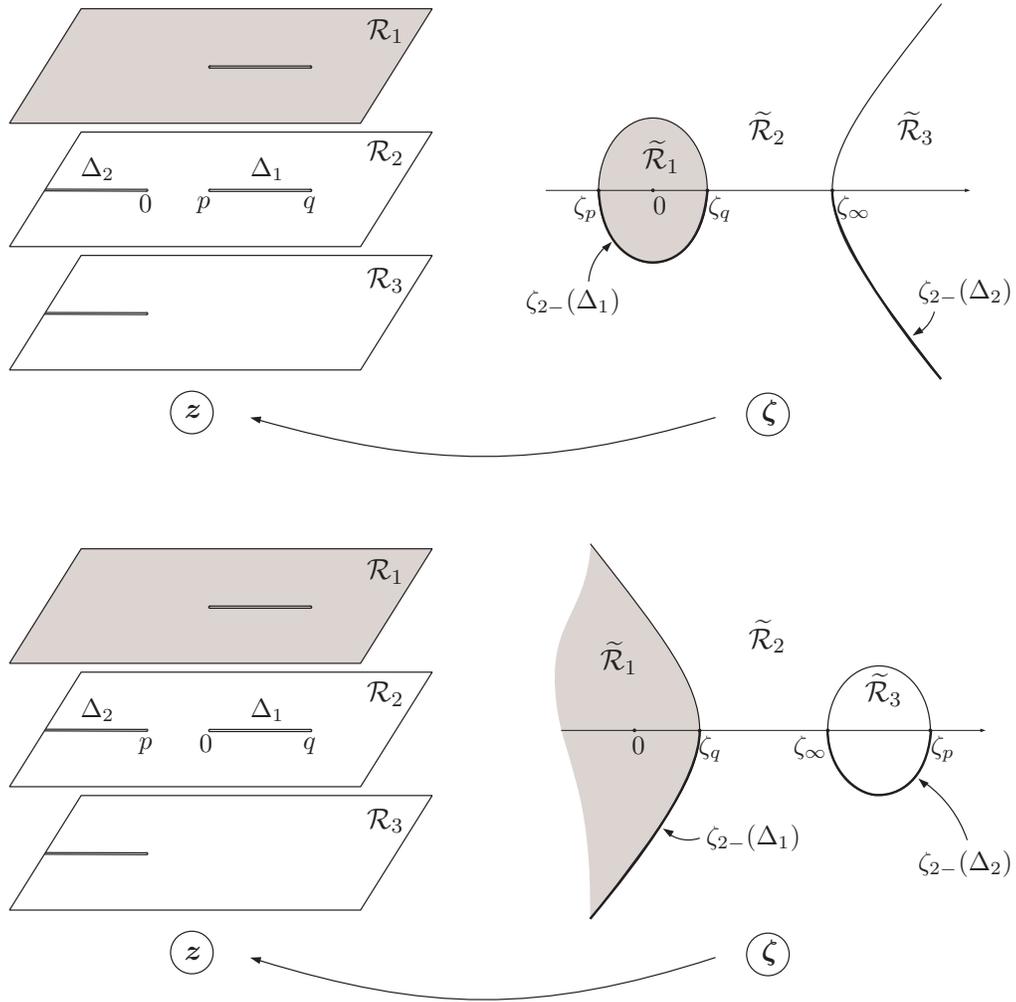}%
\put(72.2,7.6){ $\boldsymbol \zeta $}
\put(72.2,58.6){ $\boldsymbol \zeta $}
\put(17.8,7.6){ $\boldsymbol z $}
\put(17.8,58.7){ $\boldsymbol z $}
\put(35,94.5){ $ \mathcal R_1 $}
\put(35,43.5){ $ \mathcal R_1 $}
\put(35,83){ $ \mathcal R_2 $}
\put(35,32){ $ \mathcal R_2 $}
\put(35,71){ $ \mathcal R_3 $}
\put(35,20){ $ \mathcal R_3 $}
\put(50,69){ \small $ \zeta_{2-}(\Delta_1) $}
\put(87,70){ \small $ \zeta_{2-}(\Delta_2) $}
\put(87,16){ \small $ \zeta_{2-}(\Delta_2) $}
\put(67,18.5){ \small $ \zeta_{2-}(\Delta_1)$}
\put(61,82.3){ $ \widetilde{\mathcal R}_1 $}
\put(71,85){ $ \widetilde{\mathcal R}_2 $}
\put(85,85){ $ \widetilde{\mathcal R}_3 $}
\put(62,78){ \small $ 0 $}
\put(54.5,78){ \small $\zeta_p$}
\put(67.2,78){ \small $\zeta_q$}
\put(79.3,78){ \small $\zeta_{\infty}$}
\put(57,35){ $ \widetilde{\mathcal R}_1 $}
\put(71,37){ $ \widetilde{\mathcal R}_2 $}
\put(82,32){ $ \widetilde{\mathcal R}_3 $}
\put(60,27){ \small $ 0 $}
\put(88.2,27){ \small $\zeta_p$}
\put(66.3,27){ \small $\zeta_q$}
\put(75.2,27){ \small $\zeta_{\infty}$}
\put(13.5,78.2){ \small $ 0 $}
\put(13.5,27.5){ \small $ p $}
\put(19,78.5){ \small $ p $}
\put(19.2,27.2){ \small $ 0 $}
\put(29,78.5){ \small $ q $}
\put(29,27.5){ \small $ q $}
\put(24,81.5){ \small $\Delta_1$}
\put(24,30.5){ \small $\Delta_1$}
\put(8,81.5){ \small $\Delta_2$}
\put(8,30.5){ \small $\Delta_2$}
\end{overpic}
\caption{Bijection \eqref{RSequation5} between the Riemann surface $\mathcal R$ and
the extended $\zeta$-plane in the Case 1 (top) and 2 (bottom).}
\label{fig:Riemannmapping}
\end{figure}

The figure shows the domains
\[ \widetilde{\mathcal R}_j = \zeta_j(\mathcal R_j), \qquad j=1,2,3, \]
where $\mathcal R_j$ is the $j$th sheet of the Riemann surface, and
also the location of the points
\[ \zeta_p = \zeta_2(p), \quad \zeta_q = \zeta_2(q), \quad
    \zeta_{\infty} = \zeta_2(\infty) = \frac{1}{t(1-t)} = \frac{1}{c} \]
for the two cases. We observe that
$\zeta_{2+}(\Delta_1)$ and $\zeta_{2+}(\Delta_2)$ are in the upper half plane,
while $\zeta_{2-}(\Delta_1)$ and $\zeta_{2-}(\Delta_2)$ are in the lower half plane.

To solve the model RH problem for  $\alpha = 0$, we use the polynomial $D(\zeta)$
\begin{equation}
\label{defD}
D(\zeta) = (\zeta-\zeta_p)(\zeta-\zeta_q)(\zeta-\zeta_{\infty}).
\end{equation}
The square root $D (\zeta)^{1/2}$, which branches at these three points,
is defined with a cut on
$\zeta_{2-} (\Delta_1)\cup\zeta_{2-}(\Delta_2)$, which, as noted before,
are the parts of the boundary of $\widetilde{\mathcal R}_2$ that are
in the lower half of the $\zeta$-plane.
We assume that the square root is positive for
large positive $\zeta$.

\begin{proposition}\label{solmrhp1}
A solution of the model RH problem for $N_0$ is given by
\begin{equation}\label{solN0}
    N_0(z)=
    \begin{pmatrix}
    F_1(\zeta_1(z)) & F_1(\zeta_2(z))  & F_1(\zeta_3(z)) \\
        F_2(\zeta_1(z)) & F_2(\zeta_2(z))  & F_2(\zeta_3(z)) \\
        F_3(\zeta_1(z)) & F_3(\zeta_2(z))  & F_3(\zeta_3(z))
    \end{pmatrix}
\end{equation}
where
\begin{equation}
    F_1(\zeta) = K_1 \frac{(\zeta - \zeta_{\infty})^{2}}{D(\zeta)^{1/2}}, \quad
        F_2(\zeta) = K_{2} \frac{\zeta(\zeta-\zeta^*)}{D(\zeta)^{1/2}}, \quad
        F_3(\zeta) =  K_{3}  \frac{\zeta(\zeta - \zeta_{\infty})}{D (\zeta)^{1/2}},
\label{eq:F1F2F3}
\end{equation}
with $D(\zeta)$ given by \eqref{defD}. Furthermore, $\zeta^* \neq \zeta_{\infty}$,
and $K_1$, $K_{2}$, $K_{3}$ are explicitly computable non-zero constants
that depend on $a$ and $t$.
\end{proposition}

\begin{proof}
Note that each of the functions $F_j$, $j=1,2,3$,
defined in \eqref{eq:F1F2F3} satisfies
\begin{equation*} 
    \begin{cases}
    F_{j+}(\zeta) = - F_{j-}(\zeta), & \quad \zeta \in
    \partial \widetilde{\mathcal R}_2 \cap \{ \Im \zeta < 0 \}, \\
    F_{j+}(\zeta) = F_{j-}(\zeta), & \quad \zeta \in
    \partial \widetilde{\mathcal R}_2 \cap \{ \Im \zeta > 0 \},
    \end{cases}
\end{equation*}
because of the choice of the branch cut for $D(\zeta)^{1/2}$.
From this it follows that the $j$th row $(N_{j1}, N_{j2}, N_{j3})$ of $N_0$
given in \eqref{solN0} has the following jumps on $\Delta_{1}$:
\begin{equation*}  
   \left\{ \begin{array}{lll}
   (N_{j1})_+(z) & = & - (N_{j2})_-(z), \\
   (N_{j2})_+(z) & = & (N_{j1})_-(z), \\
   (N_{j3})_+(z) & = &  (N_{j3})_-(z),
   \end{array} \right. \quad z \in \Delta_{1},
\end{equation*}
and the following jumps on $\Delta_{2}$:
\begin{equation*}  
    \left\{ \begin{array}{lll}
     (N_{j1})_+(z) & = & (N_{j1})_-(z), \\
     (N_{j2})_+(z) & = & (N_{j2})_-(z), \\
     (N_{j3})_+(z) & = &  - (N_{j2})_-(z),
     \end{array} \right. \quad z \in \Delta_{2}.
\end{equation*}
These are exactly the jumps required by \eqref{eq:Njump1} and \eqref{eq:Njump2}
when $\alpha = 0$.

It remains to verify the asymptotic condition \eqref{eq:Nasymptotics} with $\alpha = 0$.
Since the computations are straightforward but cumbersome, we give here the outline of the argument.
Observe that
\begin{equation}\label{correspondence}
    \zeta_1(\infty) =0\,, \quad \zeta _2(\infty) = \zeta _3(\infty) = \zeta_\infty\,.
\end{equation}
Function $F_1$ verifies
\begin{align*}
F_1(\zeta ) & = K_1\, \frac{\zeta _\infty^2}{D(0)^{1/2}}+\mathcal O (\zeta)\,, \quad \zeta \to 0\,, \qquad
F_1(\zeta )   = \mathcal O \left( (\zeta -\zeta _\infty)^{3/2}\right)\,, \quad \zeta \to \zeta _\infty\,.
\end{align*}
Taking into account \eqref{correspondence} and \eqref{zeta1}--\eqref{zeta3}, we get that as $z\to \infty$,
\begin{align*}
N_{11}(z) = F_1(\zeta_1(z) ) & = K_1\, \frac{\zeta _\infty^2}{D(0)^{1/2}}+\mathcal O (1/z)\,, \\
N_{12}(z) = F_1(\zeta_2(z) ) & = \mathcal O \left(z^{-3/4} \right)\,, \qquad
N_{13}(z) = F_1(\zeta_3(z) )   =  \mathcal O \left(z^{-3/4} \right) \,.
\end{align*}
With
$ K_1 = \zeta _\infty^{-2}\, D(0)^{1/2} $ it yields
\begin{equation*} 
N_{11}(z) = 1 + \mathcal{O}(1/z), \qquad N_{12}(z) = \mathcal{O}(z^{-3/4}), \qquad N_{13}(z) = \mathcal{O}(z^{-3/4}),
\end{equation*}
as $z \to \infty$, which matches the asymptotic condition for the first row of $N_0$ in \eqref{eq:Nasymptotics}.

Analogously,
\begin{align*}
F_2(\zeta ) & = \mathcal O (\zeta)\,, \quad \zeta \to 0\,, \\
F_2(\zeta )  & =   \beta _1 (\zeta -\zeta _\infty)^{-1/2}+ \beta _2 (\zeta -\zeta _\infty)^{1/2}+
    \mathcal O \left( (\zeta -\zeta _\infty)^{3/2}\right)  \,, \quad \zeta \to \zeta _\infty\,,
\end{align*}
where $\beta _1$ and $\beta _2$ are explicitly computable in terms of $K_2$, $\zeta ^*$ and
the rest of the parameters of $\mathcal R$. By \eqref{zeta1}--\eqref{zeta3}, and taking
into account the second relation in \eqref{zetaconj}, we have
\begin{align*}
N_{21}(z) = F_2(\zeta_1(z) ) & =\mathcal O (1/z)\,, \quad z\to \infty\,, \\
N_{22}(z) = F_2(\zeta_2(z) ) & =    z^{1/4}\, \left( \widetilde \beta _1 +
    \widetilde \beta _2 z^{-1/2}+ \widetilde \beta _3 z^{-1}+  \mathcal O \left(z^{-3/2}
    \right)\right)\,, \quad z\to \infty\,, \\
N_{23}(z) = F_2(\zeta_3(z) ) & = i \, z^{1/4}\, \left( \widetilde \beta _1 -
    \widetilde \beta _2 z^{-1/2}+ \widetilde \beta _3 z^{-1}+  \mathcal O \left(z^{-3/2}
    \right)\right)\,, \quad z\to \infty\,,
\end{align*}
where again $\widetilde \beta _j$'s are explicit. Imposing the condition that $
\widetilde \beta_1=1/\sqrt{2}$ and $\widetilde \beta_2=0$, which determines $K_2$ and $\zeta ^*$,
we obtain that for a certain constant $a_2$,
\begin{align*} 
    N_{21}(z) & = \mathcal{O}(1/z),  \\
    N_{22}(z) & = \frac{1}{\sqrt{2}} z^{1/4}\left(1+ \frac{a_2}{z} + \mathcal{O}(z^{-3/2})\right), \\
    N_{23}(z) & =  \frac{i}{\sqrt{2}} z^{1/4}\left(1+ \frac{a_2}{z} + \mathcal{O}(z^{-3/2})\right),
\end{align*}
matching the asymptotic condition for the second row of $N_0$ in \eqref{eq:Nasymptotics}.

Finally,
\begin{align*}
F_3(\zeta ) & = \mathcal O (\zeta)\,, \quad \zeta \to 0\,, \\
F_3(\zeta )  & =   \gamma _1 (\zeta -\zeta _\infty)^{1/2} + \mathcal O \left( (\zeta -\zeta _\infty)^{3/2}\right)
    \,, \quad \zeta \to \zeta _\infty\,,
\end{align*}
where $\gamma _1$ is explicitly computable in terms of $K_3 $  and the rest of the parameters of $\mathcal R$.
By \eqref{zeta1}--\eqref{zeta3}, and taking again into account the second relation in \eqref{zetaconj}, we have
\begin{align*}
N_{31}(z) = F_3(\zeta_1(z) ) & =\mathcal O (1/z)\,, \quad z\to \infty\,, \\
N_{32}(z) = F_3(\zeta_2(z) ) & =    z^{-1/4}\, \left( \widetilde \gamma _1 + \widetilde \gamma _2 z^{-1/2}+
    \mathcal O \left(z^{-1} \right)\right)\,, \quad z\to \infty\,, \\
N_{33}(z) = F_3(\zeta_3(z) ) & = -i \, z^{1/4}\, \left( \widetilde \gamma _1 - \widetilde \gamma _2 z^{-1/2} +
    \mathcal O \left(z^{-1} \right)\right)\,, \quad z\to \infty\,,
\end{align*}
where again $\widetilde \gamma _j$'s are explicit. Imposing the condition that $
\widetilde \gamma_1=i/\sqrt{2}$, which determines $K_3$, we obtain that for a certain constant $a_3$,
\begin{align*} 
    N_{31}(z) & = \mathcal{O}(1/z), \\
    N_{32}(z) & = \frac{i}{\sqrt{2}} z^{-1/4}\left(1+ \frac{a_3}{z^{1/2}} + \mathcal{O}(z^{-1})\right), \\
    N_{33}(z) & = \frac{1}{\sqrt{2}} z^{-1/4}\left(1-  \frac{a_3}{z^{1/2}} + \mathcal{O}(z^{-1})\right),
\end{align*}
as $z \to \infty$. This is precisely the asymptotic condition for the third row of $N_0$ in
\eqref{eq:Nasymptotics}. This concludes the proof.
\end{proof}

To construct the solution for general $\alpha$, we use functions
\begin{equation} \label{r123def}
    \begin{aligned}
    r_1(\zeta) & = \log(1-c\zeta), & \zeta \in \widetilde{\mathcal R}_1, \\
    r_2(\zeta) & = \log(1-k\zeta) - \log \zeta - \log(1-c\zeta), & \zeta \in \widetilde{\mathcal R}_2, \\
    r_3(\zeta) & = \log(1-c\zeta) + i \pi, & \zeta \in \widetilde{\mathcal R}_3.
    \end{aligned}
\end{equation}
where $c$ and $k$ are as in \eqref{definitionKandC}. The branches
of the logarithm are defined as follows.
\begin{itemize}
\item $\log(1-c\zeta)$ vanishes for $\zeta = 0$, and has a branch
cut along $\zeta_{2-}(\Delta_2)$ in Case 1, and along
$\zeta_{2-}(\Delta_2) \cup [\zeta_p, +\infty)$ in Case 2.
\item $\log(1-k\zeta)$ vanishes for $\zeta = 0$, and has a branch
cut along $(-\infty,1/k]$ in case 1 (when $k < 0$), and
along $[1/k,+\infty)$ in Case 2 (when $k > 0$).
\item $\log \zeta$ is the principal branch of the logarithm, i.e.,
with a cut along $(-\infty, 0]$.
\end{itemize}
With these definitions of the branches we have that $r_j$ is defined and analytic on
$\widetilde{\mathcal R}_j$ for $j=1,2,3$. To see this for $j=2$, it is important
to note that $1/k$ is in  $\widetilde{\mathcal R}_1$ in Case 1,
and $1/k$ is in $\widetilde{\mathcal R}_3$ in Case 2.

\begin{proposition} \label{solmrhp2}
 A solution of the model RH problem for general $\alpha$ is given
by
\begin{equation} \label{Nalpha}
    N_{\alpha}(z) = C_{\alpha} N_0(z) \begin{pmatrix}
    e^{\alpha G_1(z)} & 0 & 0 \\
    0 & e^{\alpha G_2(z)} & 0 \\
    0 & 0 & e^{\alpha G_3(z)}
    \end{pmatrix} \end{equation}
where $N_0$ is given by \eqref{solN0},
\begin{equation} \label{G1G2G3def}
    G_j(z) = r_j(\zeta_j(z)), \qquad j=1,2,3, \quad z \in \mathcal R_j,
    \end{equation}
with $r_1$, $r_2$, $r_3$ defined
in \eqref{r123def}, and $C_{\alpha}$ is a constant matrix
given explicitly in \eqref{Calpha} below.
\end{proposition}
\begin{proof}
From the definitions \eqref{r123def} with the specified branches of
the logarithm it follows that the functions $r_j$, $j=1,2,3$,
satisfy the following boundary conditions
\begin{equation} \label{r123jumps}
    \begin{aligned}
    r_2(\zeta) & = r_1(\zeta) + \log z, \quad \zeta \in \partial \widetilde{\mathcal R}_1, \\
    r_2(\zeta) & = r_3(\zeta) + \log |z|, \quad \zeta \in \partial \widetilde{\mathcal R}_3,
    \end{aligned}
\end{equation}
where $z = z(\zeta)$ is given by \eqref{RSequation5}.
Then by \eqref{r123jumps} and \eqref{G1G2G3def} we obtain
\begin{equation}
\begin{aligned} \label{G1G2G3jumps}
    G_{2\pm}(x) & = \log x + G_{1\mp}(x), \qquad x \in \Delta_1, \\
    G_{2\pm}(x) & = \log|x| + G_{3\mp}(x), \qquad x \in \Delta_2.
    \end{aligned}
    \end{equation}
Using \eqref{G1G2G3jumps} and the jump propertes \eqref{eq:Njump1}, \eqref{eq:Njump2}
for $\alpha = 0$,
it is then an easy calculation to show that \eqref{Nalpha} satisfies the
jump conditions \eqref{eq:Njump2} and \eqref{eq:Njump1}.

We note that by \eqref{r123def} and \eqref{G1G2G3def},
\[ e^{G_1(z)} = 1- c \zeta_1(z), \quad e^{G_2(z)} = z(1-c\zeta_2(z)),
    \quad e^{G_3(z)} = c\zeta_3(z) - 1, \]
and thus as $z \to \infty$ by \eqref{zeta1}, \eqref{zeta2}, \eqref{zeta3},
\begin{equation*}
\begin{aligned} 
    e^{G_1(z)}  & = 1 + \mathcal O(1/z), \\
    e^{G_2(z)}  & =  z^{1/2} \left(\sqrt{a}(1-t) + \frac{t(1-t)}{2\sqrt{z}} +
        \frac{t(1-t)(t+4a(1-t))}{8\sqrt{a}z} + \mathcal{O}(z^{-3/2}) \right), \\
    e^{G_3(z)}  & = z^{-1/2} \left(\sqrt{a}(1-t) - \frac{t(1-t)}{2\sqrt{z}} +
       \frac{t(1-t)(t+4a(1-t))}{8\sqrt{a}z} + \mathcal{O}(z^{-3/2})\right).
\end{aligned}
\end{equation*}
To obtain \eqref{eq:Nasymptotics} we should then take the constant prefactor $C_{\alpha}$
in \eqref{Nalpha} as
\begin{equation} \label{Calpha}
    C_{\alpha} = \begin{pmatrix} 1 & 0 & 0 \\
    0 & \sqrt{a} (1-t) & \frac{it(1-t)}{2} \\
    0 & 0 & \sqrt{a} (1-t)  \end{pmatrix}^{-\alpha}
    = \begin{pmatrix} 1 & 0 & 0 \\
    0 & (\sqrt{a} (1-t))^{-\alpha} & -\frac{i\alpha t(1-t)}{2} (\sqrt{a} (1-t))^{-\alpha-1} \\
    0 & 0 & (\sqrt{a} (1-t))^{-\alpha} \end{pmatrix}.
    \end{equation}
Then with the choice of \eqref{Calpha},
we indeed have that $N_{\alpha}$ defined in \eqref{Nalpha} satisfies the conditions in the
model RH problem for general $\alpha$.
\end{proof}

\begin{lemma}
\label{lemma:localN}
The solution $N_{\alpha}$ of the model RH problem given
in Proposition \ref{solmrhp2}
has the following behavior near the branch points
\begin{enumerate}
\item[\rm (a)] In Case 1 we have
\begin{equation} \label{asymptoticsPandQ1} N_{\alpha}(z) = \mathcal O \begin{pmatrix}
    |z-q|^{-1/4} & |z-q|^{-1/4} & 1 \\
    |z-q|^{-1/4} & |z-q|^{-1/4} & 1 \\
    |z-q|^{-1/4} & |z-q|^{-1/4} & 1 \end{pmatrix} \quad \textrm{as } z \to q,
\end{equation}
\begin{equation} \label{asymptoticsPandQ2} N_{\alpha}(z) = \mathcal O \begin{pmatrix}
    |z-p|^{-1/4} & |z-p|^{-1/4} & 1 \\
    |z-p|^{-1/4} & |z-p|^{-1/4} & 1 \\
    |z-p|^{-1/4} & |z-p|^{-1/4} & 1
\end{pmatrix} \quad \textrm{as } z \to p,
\end{equation}
and
\begin{equation}\label{asymptoticsPandQ3} N_{\alpha}(z)
\begin{pmatrix} 1 & 0 & 0 \\
    0 & z^{-\alpha/2} & 0 \\ 0 & 0 & z^{\alpha/2}  \end{pmatrix}
    = \mathcal O \begin{pmatrix} 1 & |z|^{-1/4} & |z|^{-1/4} \\
    1 & |z|^{-1/4} & |z|^{-1/4} \\
    1 & |z|^{-1/4} & |z|^{-1/4}
\end{pmatrix} \quad \textrm{as } z \to 0,
\end{equation}
\item[\rm (b)] In Case 2 we have
\begin{equation}\label{asymptoticsPandQ4}
 N_{\alpha}(z) = \mathcal O \begin{pmatrix} |z-q|^{-1/4} & |z-q|^{-1/4} & 1 \\
    |z-q|^{-1/4} & |z-q|^{-1/4} & 1 \\
    |z-q|^{-1/4} & |z-q|^{-1/4} & 1
    \end{pmatrix} \quad \textrm{as } z \to q,
\end{equation}
\begin{equation} N_{\alpha}(z) = \mathcal O \begin{pmatrix} 1 & |z-p|^{-1/4} & |z-p|^{-1/4} \\
    1 & |z-p|^{-1/4} & |z-p|^{-1/4} \\
    1 & |z-p|^{-1/4} & |z-p|^{-1/4}
\end{pmatrix} \quad \textrm{as } z \to p,
\end{equation}
and
\begin{equation} N_{\alpha}(z)
\begin{pmatrix} z^{\alpha/2} & 0 & 0 \\
    0 & z^{-\alpha/2} & 0 \\ 0 & 0 & 1  \end{pmatrix}
    = \mathcal O \begin{pmatrix}
    |z|^{-1/4} & |z|^{-1/4} & 1 \\
    |z|^{-1/4} & |z|^{-1/4} & 1 \\
    |z|^{-1/4} & |z|^{-1/4} & 1
\end{pmatrix}
 \quad \textrm{as } z \to 0.
\end{equation}
\end{enumerate}
\end{lemma}
\begin{proof}
Observe that for $j=1, 2, 3$,
$F_j(\zeta )=\mathcal O \left( (\zeta -\zeta _q)^{-1/2}\right)$ as $\zeta \to \zeta_q$,
where $F_j$'s are defined in \eqref{eq:F1F2F3}. Furthermore, for the mapping
\eqref{RSequation5}, $\zeta_1^{-1}(\zeta _q)=\zeta_3^{-1}(\zeta _q)=q$, and $\zeta_3^{-1}(\zeta _q)$
is a regular point of $\RR$. Since functions $\zeta _1$ and $\zeta _2$ are bounded and have a square
root branch at $q$, by definition \eqref{solN0} we obtain \eqref{asymptoticsPandQ1} for $N_0$.
Since the transformation in \eqref{Nalpha} does not affect the behavior at $q$, this proves the
first identity of the Lemma. The rest of the conditions is analyzed in a similar fashion, and we omit the details.
\end{proof}

By \eqref{eq:Njump1}--\eqref{eq:Njump2}, we have that $\det N_\alpha $ is analytic in
$\C\setminus \{ 0, p, q\}$, and by Lemma \ref{lemma:localN},
$\det N_\alpha(z)=\mathcal{O}\left(|z-z_0|^{-1/2}\right)$ as $z\to z_0$ where
$z_0$ is any one of the branch points $0$, $p$ and $q$. Hence, $\det N_\alpha$ is entire.
From \eqref{eq:Nasymptotics}, $\lim_{z\to \infty} \det N_\alpha(z)=1$, and we conclude that
\begin{equation}\label{det1}
    \det N_\alpha(z)\equiv 1\,, \quad z\in \C\,.
\end{equation}
Comparing the local behavior in Proposition \ref{prop:RHforT} and Lemma \ref{lemma:localN}, we see that near
the branch points, the matrix $TN_{\alpha}^{-1}$ is not bounded which
means that $N_{\alpha}$ is not a good approximation to $T$. Hence we
need a local analysis around these points.

\section{Parametrices near the branch points $p$ and $q$ 
(soft edges)}\label{section7}

We are going to construct a local parametrix $P$ around $q$. The local parametrix around $p$ can be
built in a similar way, and is not further discussed here. Consider a
small fixed disk $B_\delta $ with radius $\delta >0$ and center at $q$ that does not
contain any other branch point. We look for a $3 \times 3$ matrix valued function $P$ such that
\begin{figure}[htb]
\centering \begin{overpic}[scale=1.4]%
{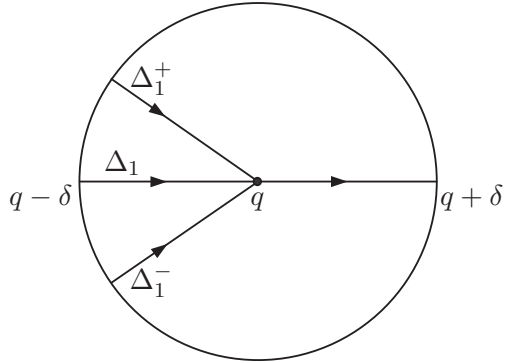}%
\put(46,44.5){  $ q $}
\put(90,44.5){  $ q+\delta $}
\put(-10,44.5){  $ q-\delta $}
\put(12,52.5){  $  \Delta_1 $}
\put(18,72){  $  \Delta_1^+ $}
\put(18,25){  $  \Delta_1^- $}
\end{overpic}
\caption{Construction of a parametrix around $q$.}
\label{fig:parametrix1}
\end{figure}
\begin{enumerate}
\item $P$ is analytic in $B_\delta \setminus (\R\cup \Delta_1^\pm)$.
\item $P$ has a
jump $P_+(z)=P_-(z)\, j_{T} (z)$ on each of the oriented contours shown in Figure
\ref{fig:parametrix1}, given by the restriction of $j_T$ in Proposition \ref{prop:RHforT} to
these contours. Namely,
{\allowdisplaybreaks
\begin{align*}
j_{T} (x) &=
    \begin{pmatrix}
0 &  x^{\alpha}          & 0    \\
-x^{-\alpha} & 0 & 0 \\
0 &  0 & 1\\
    \end{pmatrix}, \quad  x\in (q-\delta , q)= \Delta_1 \cap B_\delta ,
\\
j_{T}(z)&=
  I + z^{-\alpha} e^{n (\lambda_{2}-\lambda_{1})(z)} E_{21}  , \quad z\in\Delta_{1}^{\pm} \cap B_\delta , \\
j_{T}(x)&=
    I + x^\alpha e^{n (\lambda_{1}-\lambda_{2}) (x)} E_{1 2},
    \quad x\in (q,q+\delta).
\end{align*}}
\item As $n\to \infty$,
\[
P(z)=N_{\alpha}(z) (I+\mathcal{O}(1/n)) \quad \text{uniformly for $z \in \partial B_\delta \setminus
    (\R \cup \Delta_1^\pm)$,}
\]
where $N_{\alpha}$ is the global parametrix built in Section \ref{section6}.
\item $P$ is bounded as $z\to q$, $z\in \R\setminus \Delta_1^\pm$.
\end{enumerate}

The solution of the RH problem 1.--4.) can be built in a standard way using the Airy functions; we follow the
scheme proposed in \cite{De,DKMVZ1,DKMVZ2} and developed, for instance, in \cite{BK2,DKV,KVAW}.
The function
\begin{align} \label{fqdef}
    f(z) =
    \left[\frac{3}{4}(\lambda_2-\lambda_1)(z)\right]^{2/3}
\end{align}
is a biholomorphic (conformal) map of a neighborhood of $q$ onto a
neighborhood of the origin such that $f(z)$ is real and positive
for $z>q$. We may deform the contours $\Delta_1^\pm$ near $q$ in such a way that $f$ maps
$\Delta_1^\pm \cap B_\delta $ to the rays with angles
$\frac{2\pi}{3}$ and $-\frac{2\pi}{3}$, respectively.
We put
\begin{equation*} 
    y_0(s)=\Ai(s), \quad  y_1(s) = \omega \Ai(\omega s),
    \quad y_2(s) = \omega^{ 2 } \Ai(\omega^{ 2 } s), \quad \omega=e^{2\pi i/3}\,,
\end{equation*}
where $\Ai$ is the usual Airy function. Define the matrix $\Psi$ by
\begin{align*}
    \Psi(s) & =
    \begin{pmatrix}
        y_0(s) &  -y_2(s) & 0 \\
        y_0'(s) &  -y_2'(s) & 0 \\
        0 & 0  & 1
    \end{pmatrix},
    \quad \arg s \in (0, 2 \pi/3),
    \\
    \Psi(s) & =
    \begin{pmatrix}
        -y_1(s) &  -y_2(s) & 0 \\
        -y_1'(s) &  -y_2'(s) & 0 \\
        0 & 0   & 1
    \end{pmatrix},
     \quad  \arg s \in (2\pi/3,\pi),
     \\
    \Psi(s) & =
    \begin{pmatrix}
        -y_2(s) &   y_1(s) & 0 \\
        -y_2'(s) &  y_1'(s) & 0 \\
         0 & 0 &   1
    \end{pmatrix},
    \quad   \arg s \in (-\pi,-2\pi/3), 
    \\
    \Psi(s) & =
    \begin{pmatrix}
        y_0(s) & y_1(s) & 0 \\
        y_0'(s) & y_1'(s) & 0 \\
        0 & 0 &   1
    \end{pmatrix},
    \quad   \arg s \in (-2\pi/3,0). 
\end{align*}
Then (see e.g.\ \cite[Section 7.6]{De}), for any analytic prefactor $E$, we have that
\begin{align} \label{Psoft}
    P(z) = E(z)
    \Psi\left(n^{2/3}f (z)\right)\diag\left(z^{-\alpha /2} e^{\frac{n}{2}(\lambda_2-\lambda_1)(z)},
      z^{\alpha /2} e^{-\frac{n}{2}(\lambda_2-\lambda_1)(z)},1\right)
\end{align}
satisfies the parts 1.--3.\ of the RH problem for $P$. The freedom in $E$ can
be used to satisfy also the matching condition (4). The construction of $E$ uses
the asymptotics of the Airy function $\Ai(s)$  as $s \to \infty$, and follows the
scheme, exposed in the literature (see e.g.\ \cite{K}), and we omit the details here.
The result is the following.

\begin{proposition} \label{prop:Psoft}
The matrix-valued function $P$ given in \eqref{Psoft} with
$E$ given by
\begin{multline} \label{eq:defE}
    E(z) =   N_{\alpha}(z)\, \begin{pmatrix}
        z^{\alpha /2} &   0 & 0 \\
        0 &   z^{- \alpha /2} & 0 \\
        0 &   0 & 1
    \end{pmatrix} \\  \times  \begin{pmatrix}
   \sqrt{\pi} & -\sqrt{\pi} & 0 \\
   -i \sqrt{\pi} & -i \sqrt{\pi} & 0 \\
   0 & 0 & 1
   \end{pmatrix} \,
    \begin{pmatrix}
        n^{1/6}f ^{1/4}(z) &   0 & 0 \\
        0 &   n^{-1/6}f ^{-1/4}(z) & 0 \\
        0 &   0 & 1
    \end{pmatrix}\,,
\end{multline}
satisfies all conditions 1.--4. in the RH problem for $P$.
\end{proposition}

%

\section{Parametrix near the branch point $0$ (hard edge)}\label{section8}

From the local behavior of $T(z)$ as $z\to \infty$, described in Proposition \ref{prop:RHforT},
it follows that the local parametrix $P$ at the origin  will be different from
the parametrices at the other branch points.
Fortunately, this kind of behavior has been analyzed in \cite{KMVAV} (for a $2\times 2$
matrix valued RH problem),
and in \cite{LW} (for a $3 \times 3$ matrix valued RH problem)
and we use the construction from these papers.
There will be a new feature though in the case $-1<\alpha <0$.

\subsection{Case 1.} Let $B_\delta $ be a small fixed disk with radius $\delta >0$, now centered
at the origin, that does not contain any other branch point.
Consider all the jumps matrices $j_T$ of matrix $T$ on curves meeting at $0$,
see item 2.\ in Proposition \ref{prop:RHforT}.
The off-diagonal entry in $j_T$ on $(0,\delta)$ is $x^{\alpha} e^{n(\lambda_1-\lambda_2)(x)}$,
which is exponentially small since $\Re (\lambda_1 - \lambda_2) < -c < 0$ on $[0,\delta)$.
This suggests that we may ignore the jump on $(0,\delta)$ in the construction
of the local parametrix.
Note however, that $x^{\alpha}$ is not bounded as $x \to 0$ in case $-1 < \alpha < 0$, which
explains why we need an extra argument for this case.

\begin{figure}[htb]
\centering \begin{overpic}[scale=1.4]%
{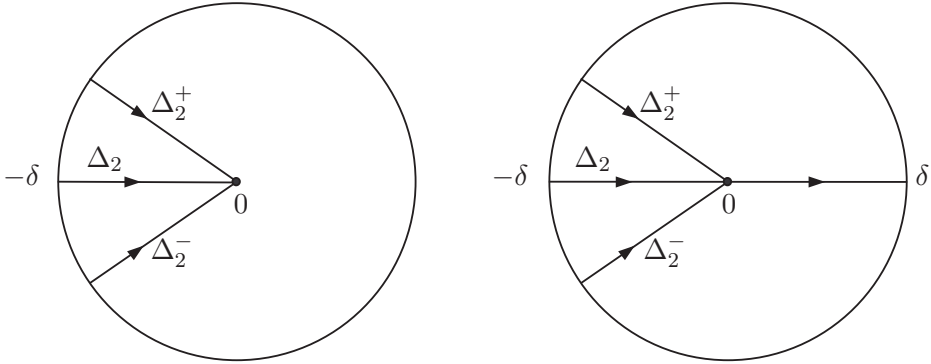}%
\put(22,20){  $ 0 $}
\put(75,20){  $ 0 $}
\put(-3,23){  $  -\delta $}
\put(50,23){  $  -\delta $}
\put(96,23){  $  \delta $}
\put(59,25){  $  \Delta_2 $}
\put(6,25){  $  \Delta_2 $}
\put(13,31){  $  \Delta_2^+ $}
\put(66,31){  $  \Delta_2^+ $}
\put(13,15){  $  \Delta_2^-$}
\put(66.5,15){  $  \Delta_2^- $}
\end{overpic}
\caption{Contours for the local parametrix around $0$ in the Case 1,
for $\alpha \geq 0$ (left picture) and for $-1< \alpha < 0$ (right picture).}
\label{fig:parametrix2}
\end{figure}

\subsubsection{First part of the construction (which works for $\alpha \geq 0$)}
In this part we simply
disregard the jump matrix on $(0, \delta )$. Taking into account
Proposition \ref{prop:RHforT}, we thus look for a $3 \times 3$ matrix valued function $Q$ such that
\begin{enumerate}
\item $Q$ is analytic in $B_\delta \setminus (\Delta_2 \cup \Delta_2^\pm)$.
\item $Q$ has a
jump $Q_+(z)=Q_-(z)j_{T} (z)$ on $(\Delta_2 \cup \Delta_2^{\pm}) \cap B_{\delta}$,
see the left picture  in Figure \ref{fig:parametrix2}. The jump matrices are given by
{\allowdisplaybreaks
\begin{align*}
j_{T}(x)&=
    \begin{pmatrix}
    1 & 0                  & 0 \\
0 & 0 & -|x|^{-\alpha}\\
0 & |x|^{\alpha} & 0
    \end{pmatrix}, \quad x\in (-\delta , 0)=\Delta_{2} \cap B_\delta , \\
j_{T}(z)&=
    I + e^{\pm \alpha \pi i} z^{-\alpha}e^{n (\lambda_{2}-\lambda_{3})(z)} E_{2 3},
    \quad z\in\Delta_{2}^{\pm} \cap B_\delta . 
\end{align*}}
\item
For $-1<\alpha<0$, $Q(z)$ behaves near the origin like:
\begin{equation}\label{P01}
Q(z)=\mathcal{O}\begin{pmatrix}
            1 & |z|^{\alpha} & 1 \\
            1 & |z|^{\alpha} & 1 \\
            1 & |z|^{\alpha} & 1
            \end{pmatrix}, \quad \text{as }z\to 0.
\end{equation}
For $\alpha=0$, $Q(z)$ behaves near the origin like:
\begin{equation} \label{P02}
Q(z)=
\mathcal{O}\begin{pmatrix}
            1 & \log |z| & \log |z| \\
            1 & \log |z| & \log |z| \\
            1 & \log |z| & \log |z|
            \end{pmatrix}, \quad \text{as }z\to 0.  
\end{equation}
For $\alpha>0$, $Q(z)$ behaves near the origin like:
\begin{align} \label{P03}
Q(z) & =
\mathcal{O}\begin{pmatrix}
            1 & 1 & |z|^{-\alpha} \\
            1 & 1 & |z|^{-\alpha}\\
            1 & 1 & |z|^{-\alpha}
            \end{pmatrix}, \quad \text{as }z\to 0 \text{ in the lens around $\Delta_2$, bounded by $\Delta_2^\pm$},
\\ \label{P03bis}
Q(z) &=\mathcal{O}\begin{pmatrix}
             1 & 1 & 1 \\
             1 & 1 & 1 \\
             1 & 1 & 1
            \end{pmatrix}, \quad \text{as }z\to 0\text{ outside the lens.}
\end{align}
\item As $n\to \infty$,
\begin{equation} \label{P0match}
Q(z)=N_{\alpha}(z) (I+\mathcal{O}(1/n)) \quad \text{uniformly for
$z \in \partial B_\delta \setminus (\Delta_2 \cup \Delta_2^\pm)$,}
\end{equation}
where $N_{\alpha}$ is the parametrix built in Section \ref{section6}.
\end{enumerate}

Consider
\begin{equation} \label{tildeQdefCase1}
    \widetilde{Q} (z)= Q(z)
    \diag \left ( 1,\;  (\pm 1)^n  z^{-\alpha /2} e^{\frac{n}{2}(\lambda _2-\lambda _3)(z)}, \;
    (\pm 1)^n z^{\alpha /2} e^{-\frac{n}{2}(\lambda _2-\lambda _3)(z)}  \right), \quad
    \text{for $\pm \Im z > 0$},
\end{equation}
where $z^{ \alpha /2}$ denotes the principal branch, as usual.
By Lemma \ref{lem:analyticity}, the diagonal factor in \eqref{tildeQdefCase1}
is analytic in
 $B_\delta \setminus (-\delta ,0)$.
It follows that the matrix valued function $\widetilde{Q}$ should satisfy:
\begin{enumerate}
\item $\widetilde{Q}$ is analytic in $B_\delta \setminus (\R\cup \Delta_2^\pm)$.
\item $\widetilde{Q}$ has a jump $\widetilde{Q}_+(z)= \widetilde{Q}_-(z)j_{\widetilde{Q}} (z)$
on each of the oriented contours shown in Figure \ref{fig:parametrix2}, left. They are given by
{\allowdisplaybreaks
\begin{align*}
j_{\widetilde{Q}}(x)&=
    \begin{pmatrix}
    1 & 0                  & 0 \\
    0 & 0 & -1\\
    0 & 1 & 0
    \end{pmatrix}, \quad x\in (-\delta , 0)=\Delta_{2} \cap B_\delta , \\
j_{\widetilde{Q}}(z)&=
    I + e^{\pm \alpha \pi i} E_{2 3}, \quad z\in\Delta_{2}^{\pm} \cap B_\delta ,
\end{align*}}
where we have used the last identity in \eqref{jumplamb}.
\item
For $-1<\alpha<0$, $\widetilde{Q}(z)$ behaves near the origin like:
\begin{equation}\label{P01tilde}
\widetilde Q(z)=\mathcal{O}\begin{pmatrix}
            1 & |z|^{\alpha/2} & |z|^{\alpha/2} \\
            1 & |z|^{\alpha/2} & |z|^{\alpha/2} \\
            1 & |z|^{\alpha/2} & |z|^{\alpha/2}
            \end{pmatrix}, \quad \text{as }z\to 0.
\end{equation}
For $\alpha=0$, $\widetilde Q(z)$ behaves near the origin like:
\begin{equation} \label{P02tilde}
\widetilde Q(z)=
\mathcal{O}\begin{pmatrix}
            1 & \log |z| & \log |z| \\
            1 & \log |z| & \log |z| \\
            1 & \log |z| & \log |z|
            \end{pmatrix}, \quad \text{as }z\to 0.
\end{equation}
For $ \alpha>0$, $\widetilde Q(z)$ behaves near the origin like:
\begin{equation} \label{P03tilde}
\widetilde Q(z)=\mathcal{O}\begin{pmatrix}
             1 & |z|^{\alpha/2} & |z|^{-\alpha/2} \\
             1 & |z|^{\alpha/2} & |z|^{-\alpha/2} \\
             1 & |z|^{\alpha/2} & |z|^{-\alpha/2}
            \end{pmatrix}, \quad \text{as }z\to 0\text{ outside $\Delta_2^\pm$},
\end{equation}
and
\begin{equation} \label{P03bistilde}
\widetilde Q(z)=\mathcal{O}\begin{pmatrix}
            1 & |z|^{-\alpha/2} & |z|^{-\alpha/2} \\
            1 & |z|^{-\alpha/2} & |z|^{-\alpha/2} \\
            1 & |z|^{-\alpha/2} & |z|^{-\alpha/2}
            \end{pmatrix}, \quad \text{as }z\to 0 \text{ inside $\Delta_2^\pm$}.
\end{equation}
\end{enumerate}
Although we have different expressions for $\widetilde{Q}$ for
the cases $n$ even and $n$ odd, there is no distinction between these
two cases in the conditions on $\widetilde{Q}$.

The problem for $\widetilde{Q}$ has a solution in terms of the modified Bessel functions of order $\alpha $
see \cite[Section 6]{KMVAV}.
Namely, with the modified Bessel functions $I_{\alpha}$ and $K_{\alpha}$, and the
Hankel functions $H_{\alpha}^{(1)}$ and $H_{\alpha}^{(2)}$ (see \cite[Chapter 9]{AS}), we define a
$2 \times 2 $ matrix $\Psi(\zeta)$ for $|\arg \zeta| < 2 \pi/3$ as
\begin{equation}\label{RHPPSIsolution1}
    \Psi(\zeta) =
    \begin{pmatrix}
        I_{\alpha} (2 \zeta^{1/2}) & \frac{i}{\pi} K_{\alpha}(2 \zeta^{1/2}) \\[1ex]
        2\pi i \zeta^{1/2} I_{\alpha}'(2\zeta^{1/2}) & -2 \zeta^{1/2} K_{\alpha}'(2\zeta^{1/2})
    \end{pmatrix}.
\end{equation}
For $2\pi/3 < \arg \zeta < \pi$ we define it as
\begin{equation}\label{RHPPSIsolution2}
    \Psi(\zeta) =
    \begin{pmatrix}
        \frac{1}{2} H_{\alpha}^{(1)}(2(-\zeta)^{1/2}) &
        \frac{1}{2} H_{\alpha}^{(2)}(2(-\zeta)^{1/2}) \\[1ex]
        \pi \zeta^{1/2} \left(H_{\alpha}^{(1)}\right)'(2(-\zeta)^{1/2}) &
        \pi \zeta^{1/2} \left(H_{\alpha}^{(2)}\right)'(2(-\zeta)^{1/2})
    \end{pmatrix} e^{\frac{1}{2}\alpha\pi i \sigma_3 }.
\end{equation}
And finally for $- \pi < \arg \zeta  < -2\pi/3$ it is defined as
\begin{equation}\label{RHPPSIsolution3}
    \Psi(\zeta) =
    \begin{pmatrix}
        \frac{1}{2} H_{\alpha}^{(2)}(2(-\zeta)^{1/2}) &
        -\frac{1}{2} H_{\alpha}^{(1)}(2 (-\zeta)^{1/2}) \\[1ex]
        -\pi \zeta^{1/2} \left(H_{\alpha}^{(2)}\right)'(2 (- \zeta)^{1/2}) &
        \pi \zeta^{1/2} \left(H_{\alpha}^{(1)}\right)'(2 (-\zeta)^{1/2})
    \end{pmatrix}
    e^{-\frac{1}{2} \alpha \pi i \sigma_3}.
\end{equation}

Then we define a $3\times 3$ matrix $\widetilde{\Psi}$, given in block form by
\begin{equation} \label{PsiTilde}
    \widetilde{\Psi}(\zeta ) =  \left(\begin{MAT}{c|c}
  1 & 0 \\-
  0 & \sigma_1 \Psi(\zeta) \sigma_1 \\
\end{MAT}
\right)\,, \qquad \sigma_1 = \begin{pmatrix} 0 & 1 \\ 1 & 0 \end{pmatrix}.
\end{equation}
[The conjugation by $\sigma_1$ is needed to interchange the second and third
rows and columns.]
The function
\begin{align*} 
    f(z) =
    \left[\frac{1}{2}(\lambda_2-\lambda_3)(z)\right]^{2} =
    \left[\frac{1}{2}\, \int_0^z (\zeta_2-\zeta_3)(s)\, ds \right]^{2}
\end{align*}
can be continued analytically from $B_\delta \setminus (-\delta, 0]$ to the full neighborhood $B_\delta $,
giving a biholomorphic (conformal) homeomorphism of a neighborhood of the origin onto itself
(see \eqref{zeta-origin1}) such that $f(x)$ is real and positive
for $x \in (0, \delta )$. Again, we may deform the contours $\Delta_2^\pm$ near $0$ in such a way
that $f$ maps $\Delta_2^\pm \cap B_\delta $ to the rays with angles
$\frac{2\pi}{3}$ and $-\frac{2\pi}{3}$, respectively. It follows from \cite{KMVAV} that
for any analytic prefactor $E$, we have that
\[
    \widetilde Q(z) = E(z) \widetilde{\Psi}(n^2 f(z) )
\]
satisfies the conditions (1), (2), (3) needed for $\widetilde{Q}$.
So we complete the construction of $Q$ by defining
\begin{equation} \label{P0def}
Q (z)= E(z)  \widetilde{\Psi}(n^2 f(z) )
    \diag \left ( 1,\;  z^{\alpha /2} e^{-\frac{n}{2}(\lambda _2-\lambda _3)(z)}, \;
     z^{-\alpha /2} e^{\frac{n}{2}(\lambda _2-\lambda _3)(z)}  \right )\,,
\end{equation}
where $E$, analytic in $B_\delta $, is chosen to satisfy the matching condition on $\partial B_\delta $.
Using again the results of \cite{KMVAV}, and taking into account that
we have to interchange the second and third rows and columns, we define
\begin{multline}
    E(z)
      =  N_{\alpha}(z) \diag \left ( 1,\;    z^{ -\alpha /2} , \;   z^{ \alpha /2}   \right )
      \diag \left(1, \frac{1}{\sqrt{2}}
    \begin{pmatrix}
    1 & -i \\
      -i & 1
    \end{pmatrix} \right)   \\
     \times \diag \left( 1, \; \left(2\pi n \right)^{-1/2} f(z)^{-1/4}, \;  \left(2\pi n \right)^{1/2} \, f(z)^{1/4}
    \right)  \label{defEHardEdge2}  \,.
\end{multline}
Here the branch of $f ^{1/4}(z)$ is positive for $z\in (0,  \delta )$.
Observe that $f ^{1/4}(z)=\mathcal O \left( z^{1/4} \right)$ as $z\to 0$, so by \eqref{asymptoticsPandQ3},
\[
E(z)= \mathcal O \begin{pmatrix} 1 & 1 & 1 \\
    1 & z^{-1/2} & 1 \\
    1 & z^{-1/2} & 1
\end{pmatrix} \quad \textrm{as } z \to 0.
\]
It is easy to check that
\[
E_+(x)= E_-(x) \, \begin{pmatrix}
    1 &   0 & 0 \\
     0 &  i (f_+/f_-) ^{-1/4}(z)   &   0  \\
      0 &  0 &     -i (f_+/f_-) ^{1/4}(z)
    \end{pmatrix}\,, \quad x\in (-\delta , 0)\,.
\]
Since $f_+ ^{1/4}(x)= i f_- ^{1/4}(x)$ for $x\in (-\delta , 0)$ and $E$ cannot have a pole at the
origin, we conclude that $E$ is analytic in $B_\delta $.

Finally, the matching condition \eqref{P0match} in condition (4) of the
RH problem for $Q$ is satisfied by results of \cite{KMVAV}. We have thus established
the following.
\begin{proposition} \label{prop:PhardCase1}
The matrix-valued function $Q$ defined by \eqref{P0def}, \eqref{defEHardEdge2},
with $\widetilde{\Psi}$ as in \eqref{PsiTilde}
satisfies the conditions 1.--4. of the RH problem for $Q$.
\end{proposition}

Taking into account \eqref{det1} and that $\det \widetilde{\Psi}=1$ (see \cite{KMVAV}) we
also conclude that
\begin{equation}\label{det2}
    \det Q(z)\equiv 1\,, \quad z\in B_{\delta}.
\end{equation}
If we would take $Q$ as the local parametrix for $T$, we would define
the final transformation as
\[ R(z) = T(z) Q(z)^{-1}, \quad z \in B_{\delta}. \]
Then $R$ would be analytic in $B_{\delta} \setminus (0,\infty)$
with  the following jump for  $x \in (0,\delta)$,
\begin{align} \nonumber
    R_-(x)^{-1} R_+(x) & =
    Q(x) T_-(x)^{-1} T_+(x) Q(x)^{-1} \\
    \nonumber
    & = Q(x) \left( I + x^{\alpha} e^{n(\lambda_1-\lambda_2)(x)} E_{12} \right) Q(x)^{-1} \\
    & = I + x^{\alpha} e^{n(\lambda_1-\lambda_2)(x)} Q(x) E_{12} Q(x)^{-1}.
    \label{Rjump}
    \end{align}

\begin{lemma} \label{lem:P0bounded}
For $\alpha \geq 0$, the matrix $Q(x) E_{12} Q(x)^{-1}$
is bounded as $x \to 0$, $x > 0$.
\end{lemma}
\begin{proof}
If $\alpha > 0$, then it follows from \eqref{P03bis} and \eqref{det2}
that both $Q(x)$ and $Q(x)^{-1}$ are bounded as $x \to 0$, $x > 0$,
and the lemma follows.

For $\alpha = 0$, the above argument, now based on \eqref{P02}
instead of \eqref{P03bis}, does not work, since it would lead
to a bound  $\mathcal O(\log|x|)$ as $x \to 0$. To prove the lemma
for $\alpha = 0$, we look at the precise construction of $Q$.
From \eqref{RHPPSIsolution1} and the known behavior of $I_0(\zeta)$
and $K_0(\zeta)$ as $\zeta \to 0$, we obtain
\[ \Psi(\zeta) = \mathcal O\begin{pmatrix} 1 & \log |\zeta| \\ 1 & \log |\zeta|
    \end{pmatrix} \]
Since $\det \Psi(\zeta) = 1$, it then follows by \eqref{PsiTilde}
that
\[ \widetilde{\Psi}(\zeta)^{-1} = \begin{pmatrix} 1 & 0 & 0 \\
    0 & \mathcal{O}(1) & \mathcal{O}(1) \\
    0 & \mathcal{O}(\log |\zeta|) & \mathcal{O}(\log |\zeta|) \end{pmatrix}
    \qquad \text{ as } \zeta \to 0. \]
Using this in \eqref{P0def} we obtain
\begin{equation} \label{lemPinverse}
    Q(x)^{-1} = \begin{pmatrix} 1 & 0 & 0 \\
    0 & \mathcal{O}(1) & \mathcal{O}(1) \\
    0 & \mathcal{O}(\log |x|) & \mathcal{O}(\log |x|) \end{pmatrix} E^{-1}(x),
    \quad \text{as } x \to 0, x > 0,
    \end{equation}
where $E^{-1}(x)$ is bounded near $x=0$.
Since
\[ Q(x) E_{12} Q(x)^{-1} = Q(x) \begin{pmatrix} 1 \\ 0 \\ 0 \end{pmatrix}
    \begin{pmatrix} 0 & 1 & 0 \end{pmatrix} Q(x)^{-1} \]
and $Q(x) \begin{pmatrix} 1 \\ 0 \\ 0 \end{pmatrix}$ is
bounded by \eqref{P02} and
$ \begin{pmatrix} 0 & 1 & 0 \end{pmatrix} Q(x)^{-1}$ is bounded
by \eqref{lemPinverse}, the lemma follows for $\alpha = 0$ as well.
\end{proof}

From Lemma \ref{lem:P0bounded}
and the fact that $\Re (\lambda_1 - \lambda_2) < -c < 0$, for some $c >0$,
it follows that the jump matrix \eqref{Rjump} is exponentially close to the
identity matrix as $n \to \infty$, uniformly for $x \in (0,\delta)$,
in case $\alpha \geq 0$. We take the parametrix $P = Q$ in case $\alpha \geq 0$.

This does not work if $\alpha < 0$, since then we would get that
$Q(x) E_{12} Q(x)^{-1}$
is of order $x^{\alpha}$ as $x \to 0$. Then for any fixed $x > 0$, the jump matrix
is close to the identity matrix as $n \to \infty$, but it
is not valid uniformly for $x \in (0,\delta)$.


\subsubsection{Second part of the construction, for $-1 < \alpha < 0$}
Let us analyze now the case when $-1<\alpha <0$.
Now we cannot simply ignore the jump matrix of $T$ on
$(0, \delta )$, so we will try to match all four jumps.
Namely, we build a $3 \times 3$
matrix valued function $P$ such that
\begin{enumerate}
\item[1.] $P$ is analytic in $B_\delta \setminus \left(\Delta_2 \cup \Delta_2^\pm \cup (0,\delta) \right)$.
\item[2.] $P$ has a
jump $P_+(z)=P_-(z)\, j_{T} (z)$ on each of the oriented contours shown in
the right picture of Figure \ref{fig:parametrix2}.  The jump matrices are given by
{\allowdisplaybreaks
\begin{align*}
j_{T}(x)&=
    \begin{pmatrix}
    1 & 0                  & 0 \\
0 & 0 & -|x|^{-\alpha}\\
0 & |x|^{\alpha} & 0
    \end{pmatrix}, \quad x\in (-\delta , 0)=\Delta_{2} \cap B_\delta , \\
j_{T}(z)&=
    I + e^{\pm \alpha \pi i} z^{-\alpha}e^{n (\lambda_{2}-\lambda_{3})(z)} E_{2 3},
    \quad z\in\Delta_{2}^{\pm} \cap B_\delta , \\
j_{T} (x)& =
     I+  x^\alpha e^{n (\lambda_{1}-\lambda_{2}) (x)} E_{12} , \quad x\in (0,\delta ).
\end{align*}}
\item[3.]
$P(z)$ behaves near the origin like:
\begin{equation}\label{Q01}
P(z)=\mathcal{O}\begin{pmatrix}
            1 & |z|^{\alpha} & 1 \\
            1 & |z|^{\alpha} & 1 \\
            1 & |z|^{\alpha} & 1
            \end{pmatrix}, \quad \text{as }z\to 0.
\end{equation}
\item[4.] As $n\to \infty$,
\begin{equation} \label{Q0match}
P(z)=N_{\alpha}(z) (I+\mathcal{O}(1/n)) \quad \text{uniformly for
$z \in \partial B_\delta \setminus \left(\Delta_2 \cup \Delta_2^\pm \cup (0,\delta) \right)$,}
\end{equation}
where $N_{\alpha}$ is the parametrix built in Section \ref{section6}.
\end{enumerate}

We use the matrix-valued function $Q$ given by formulas \eqref{P0def}
and \eqref{defEHardEdge2}, that worked as a parametrix for the case $\alpha \geq 0$.
We take $P$ in the form
\begin{equation}\label{defQ1}
 P(z) = Q(z) S(z),
\end{equation}
where $S$ is given in the four components of $B_{\delta} \setminus
(\Delta_2 \cup \Delta_2^{\pm} \cup (0,\delta))$ as follows:
\begin{align} \nonumber
    S(z) & = I + \frac{1}{1-e^{2 \alpha \pi i}} z^{\alpha} e^{n(\lambda_1-\lambda_2)(z)} E_{12},
    \\ & \label{Sdef1}
    \qquad \text{for $z$ in the region bounded by $(0,\delta)$ and $\Delta_2^{+}$}, \\
    S(z) & = I + \frac{e^{2\alpha \pi i}}{1-e^{2 \alpha \pi i}} z^{\alpha}
    e^{n(\lambda_1-\lambda_2)(z)} E_{12}, \nonumber
    \\ & \label{Sdef2}
    \qquad \text{for $z$ in the region bounded by $(0,\delta)$ and $\Delta_2^{-}$}, \\
    S(z) & = I + \frac{1}{1-e^{2 \alpha \pi i}} z^{\alpha} e^{n(\lambda_1-\lambda_2)(z)} E_{12}
    - \frac{e^{\alpha \pi i}}{1-e^{2\alpha \pi i}} e^{n(\lambda_1-\lambda_3)(z)} E_{13}, \nonumber
    \\ & \label{Sdef3}
    \qquad \text{for $z$ in the region bounded by $\Delta_2$ and $\Delta_2^{+}$}, \\
    S(z) & = I + \frac{e^{2\alpha \pi i}}{1-e^{2 \alpha \pi i}} z^{\alpha} e^{n(\lambda_1-\lambda_2)(z)} E_{12}
    + \frac{e^{\alpha \pi i}}{1-e^{2\alpha \pi i}} e^{n(\lambda_1-\lambda_3)(z)} E_{13}, \nonumber
    \\ & \label{Sdef4}
    \qquad \text{for $z$ in the region bounded by $\Delta_2$ and $\Delta_2^{-}$}.
    \end{align}
This construction is actually valid for any non-integer $\alpha$.

It is a straightforward, although somewhat lengthy, calculation to show
that $P$ satisfies all the jump conditions from item 2. in the RH problem for
$P$. To check the jump on $\Delta_2 = (-\delta, 0)$
one has to keep in mind that $\lambda_{2+} = \lambda_{3-} - 2\pi i$ on
$\Delta_2$, see \eqref{jumplamb},
and that $z^{\alpha}$ is defined with a cut on $(-\infty,0]$.
The conditions 1., 3., and 4. in the RH problem for $P$ are easy to verify
from the above definitions and the corresponding conditions in the RH problem for $Q$.
For condition 4. we also need to note that $\Re (\lambda_1  -\lambda_j)(z) < - c < 0$
for $j=2,3$ and $z \in B_{\delta}$.

In order to unify notation for $\alpha \geq 0$ and $-1 < \alpha <0$, we take as the parametrix in
$B_\delta $ in the Case 1 the matrix valued function $P=QS$, where $S=I$ if $\alpha \geq 0$, and
$S$ is given by \eqref{Sdef1}--\eqref{Sdef4}, if $-1 < \alpha <0$.

\subsection{Case 2.}
The construction of the local parametrix $P$  near the origin in Case 2
follows along similar lines as the construction in Case 1.
In Case 2 the geometry of the curves in the RH problem for $T$ is
shown in the right picture of Figure \ref{fig:parametrix3}. Now
the jump matrix on $(-\delta, 0)$ is exponentially close to the
identity matrix if $n$ is large, and in the first step of the
construction we ignore the jump on $(-\delta,0)$, thereby giving
us the contours as in the left picture of Figure \ref{fig:parametrix3}.

\begin{figure}[htb]
\centering \begin{overpic}[scale=1.4]%
{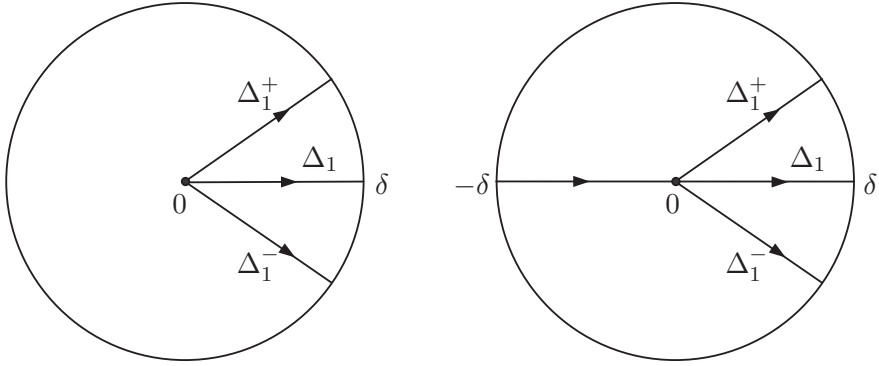}%
\put(21,20){  $ 0 $}
\put(74.5,20){  $ 0 $}
\put(51.5,22){  $   -\delta $}
\put(43,22){  $   \delta $}
\put(96,22){  $   \delta $}
\put(35,25){  $  \Delta_1 $}
\put(88,25){  $  \Delta_1 $}
\put(28,32){  $  \Delta_1^+ $}
\put(81,32){  $  \Delta_1^+ $}
\put(28,13.5){  $  \Delta_1^- $}
\put(81,13.5){  $  \Delta_1^- $}
\end{overpic}
\caption{Contours for the local parametrix around $0$ in the Case 2,
for $\alpha \geq 0$ (left picture) and for $-1< \alpha < 0$ (right picture).}
\label{fig:parametrix3}
\end{figure}

\subsubsection{Construction for $\alpha \geq 0$}
We start by constructing a solution to the following RH problem (see left picture
of Figure \ref{fig:parametrix3}).
\begin{enumerate}
\item[1.] $Q$ is analytic in $B_\delta \setminus (\Delta_1 \cup \Delta_1^\pm)$.
\item[2.] $Q$ has a
jump $Q_+(z)=Q_-(z)\, j_{T} (z)$ on each of the oriented contours shown in Figure
\ref{fig:parametrix3}. They are given by
{\allowdisplaybreaks
\begin{align*}
    j_{T}(x) &=
    \begin{pmatrix}
    0 & x^{\alpha}                  & 0 \\
    -x^{-\alpha} & 0 & 0\\
    0 & 0 & 1
    \end{pmatrix}, \quad x\in (0,\delta) =\Delta_{1} \cap B_\delta , \\
    j_{T}(z) &=
    I + z^{-\alpha}e^{n (\lambda_{2}-\lambda_{1})(z)} E_{2 1}, \quad z\in\Delta_{1}^{\pm} \cap B_\delta.
\end{align*}}
\item[3.]
For $-1<\alpha<0$, $Q(z)$ behaves near the origin like:
\begin{equation}\label{P01Case2}
Q(z)=\mathcal{O}\begin{pmatrix}
            1 & |z|^{\alpha} & 1 \\
            1 & |z|^{\alpha} & 1 \\
            1 & |z|^{\alpha} & 1
            \end{pmatrix}, \quad \text{as }z\to 0.
\end{equation}
For $\alpha=0$, $Q(z)$ behaves near the origin like:
\begin{equation} \label{P02Case2}
Q(z)=\mathcal{O}\begin{pmatrix}
             \log |z| & \log |z| & 1 \\
             \log |z| & \log |z| & 1 \\
             \log |z| & \log |z| & 1
            \end{pmatrix}, \quad \text{as }z\to 0,
\end{equation}
For $0<\alpha$, $Q(z)$ behaves near the origin like:
\begin{align} \label{P03Case2}
Q(z) & =
\mathcal{O}\begin{pmatrix}
            |z|^{-\alpha} & 1 & 1 \\
            |z|^{-\alpha} & 1 & 1 \\
            |z|^{-\alpha} & 1 & 1
            \end{pmatrix}, \quad \text{as }z\to 0 \text{ in the lens around $\Delta_1$}, \\
 \label{P03bisCase2}
Q(z) &=\mathcal{O}\begin{pmatrix}
             1 & 1 & 1 \\
             1 & 1 & 1 \\
             1 & 1 & 1
            \end{pmatrix}, \quad \text{as }z\to 0\text{ outside the lens}.
\end{align}
\item[4.] As $n\to \infty$,
\begin{equation} \label{P0matchCase2}
Q(z)=N_{\alpha}(z) (I+\mathcal{O}(1/n)) \quad \text{uniformly for
$z \in \partial B_\delta \setminus (\Delta_1 \cup \Delta_1^\pm)$,}
\end{equation}
where $N_{\alpha}$ is the parametrix built in Section \ref{section6}.
\end{enumerate}

With $\Psi$ built in \eqref{RHPPSIsolution1}--\eqref{RHPPSIsolution3} we define
a $3 \times 3$ matrix-valued function
\begin{equation} \label{Psihat}
    \widehat{\Psi}(\zeta ) =  \left(\begin{MAT}{c|c}
    \sigma_3 \Psi(-\zeta) \sigma_3 & 0 \\-
    0 & 1 \\
\end{MAT}
\right)\,, \qquad \sigma_3 = \begin{pmatrix} 1 & 0 \\ 0 & -1 \end{pmatrix},
\end{equation}
where now $\Psi$ is in the upper left block, and
\begin{multline}\label{P0defCase2}
Q (z)=E(z) \, \widehat{\Psi}(n^2 f(z) )
\diag \left((\pm 1)^n (-z)^{- \alpha /2} e^{\frac{n}{2}(\lambda _2-\lambda _1)(z)}, \;
    (\pm 1)^n (-z)^{\alpha /2} e^{-\frac{n}{2}(\lambda _2-\lambda _1)(z)}, \; 1  \right),
\end{multline}
for $\pm \Im z > 0$,
where $(-z)^{\alpha /2}$ is positive for $z \in (-\delta, 0)$ and is defined with a cut
on $(0,+\infty)$.
Here $f$ is the conformal map
\begin{equation} \label{fdefCase2}
f(z)= \left[\frac{1}{2}(\lambda_2-\lambda_1)(z)- \frac{1}{2}(\lambda_2-\lambda_1)(0)\right]^{2} =
\left[\frac{1}{2}\, \int_0^z (\zeta_2-\zeta_1)(s)\, ds \right]^{2},
\end{equation}
and the analytic prefactor $E$ is
\begin{multline} \label{EdefCase2}
    E(z) = N_{\alpha}(z) \diag \left((-z)^{\alpha/2}, (-z)^{-\alpha/2} , 1\right)
    \diag \left(\frac{1}{\sqrt{2}} \begin{pmatrix} 1 & i \\ i & 1 \end{pmatrix}, 1 \right) \\
    \times
     \diag \left( (2\pi n)^{1/2} f(z)^{1/4},\; (2\pi n)^{-1/2} f(z)^{-1/4}, \; 1
    \right).
\end{multline}

Then we find the following analogue of Proposition \ref{prop:PhardCase1}.
\begin{proposition} \label{prop:PhardCase2}
The matrix-valued function $Q$ defined by \eqref{P0defCase2}, \eqref{EdefCase2},
with $\widehat{\Psi}$ as in \eqref{Psihat} and $f$ as in \eqref{fdefCase2},
satisfies the conditions 1.--4. of the RH problem for $Q$.
\end{proposition}

\subsubsection{Construction for $-1 < \alpha < 0$}
The above constructed $Q$ can be used  as a parametrix $P$ for $T$ in case $\alpha \geq 0$.
For $-1 < \alpha < 0$, the parametrix should also have the same jump as $T$  on $(-\delta, 0)$,
and we seek a $3 \times 3$ matrix valued function $P$ such that
\begin{enumerate}
\item[1.] $P$ is analytic in $B_\delta \setminus \left(\Delta_1 \cup \Delta_1^\pm \cup (-\delta,0) \right)$.
\item[2.] $P$ has a
jump $P_+(z)=P_-(z)\, j_{T} (z)$ on each of the oriented contours shown in Figure
\ref{fig:parametrix3}, right. They are given by
{\allowdisplaybreaks
\begin{align*}
    j_{T}(x) &=
    \begin{pmatrix}
    0 & x^{\alpha}                  & 0 \\
    -x^{-\alpha} & 0 & 0\\
    0 & 0 & 1
    \end{pmatrix}, \quad x\in (0,\delta) =\Delta_{1} \cap B_\delta , \\
    j_{T}(z) &=
    I + z^{-\alpha}e^{n (\lambda_{2}-\lambda_{1})(z)} E_{2 1}, \quad z\in\Delta_{2}^{\pm} \cap B_\delta, \\
    j_{T}(x) &= I+ |x|^{\alpha}e^{n (\lambda_{3}-\lambda_{2})  (x)} E_{32} \,, \quad x\in (-\delta, 0) \,.
\end{align*}}
\item[3.]
$P(z)$ behaves near the origin like:
\begin{equation}\label{Q01Case2}
P(z)=\mathcal{O}\begin{pmatrix}
            1 & |z|^{\alpha} & 1 \\
            1 & |z|^{\alpha} & 1 \\
            1 & |z|^{\alpha} & 1
            \end{pmatrix}, \quad \text{as }z\to 0.
\end{equation}
\item[4.] As $n\to \infty$,
\begin{equation} \label{Q0matchCase2}
P(z)=N_{\alpha}(z) (I+\mathcal{O}(1/n)) \quad \text{uniformly for $z \in \partial B_\delta \setminus
\left(\Delta_1 \cup \Delta_1^\pm \cup (-\delta, 0) \right)$,}
\end{equation}
where $N_{\alpha}$ is the parametrix built in Section \ref{section6}.
\end{enumerate}

Just as in Case 1, we build $P$ in the form \eqref{defQ1},
\begin{equation}\label{defQ1Case2}
 P(z) = Q(z) S(z),
\end{equation}
where $Q$ is the matrix valued function constructed by formulas \eqref{P0defCase2}--\eqref{EdefCase2},
and $S$ is now explicitly given in each of the four components of
$B_{\delta} \setminus (\Delta_1 \cup \Delta_1^{\pm} \cup (-\delta,0))$ by
\begin{align} \nonumber
    S(z) & = I - \frac{e^{\alpha \pi i}}{1-e^{2 \alpha \pi i}} z^{\alpha} e^{n(\lambda_3-\lambda_2)(z)} E_{32},
    \\ & \label{Sdef1Case2}
    \qquad \text{for $z$ in the region outside the lens}, \\
    S(z) & = I - \frac{e^{\alpha \pi i}}{1-e^{2 \alpha \pi i}} z^{\alpha}
    e^{n(\lambda_3-\lambda_2)(z)} E_{32}
    + \frac{e^{\alpha \pi i}}{1-e^{2\alpha \pi i}} e^{n(\lambda_3-\lambda_1)(z)} E_{31}, \nonumber
    \\ & \label{Sdef3Case2}
    \qquad \text{for $z$ in the  upper part of the lens around $\Delta_1$}, \\
    S(z) & = I - \frac{e^{\alpha \pi i}}{1-e^{2 \alpha \pi i}} z^{\alpha} e^{n(\lambda_3-\lambda_2)(z)} E_{32}
    - \frac{e^{\alpha \pi i}}{1-e^{2\alpha \pi i}} e^{n(\lambda_3-\lambda_1)(z)} E_{31}, \nonumber
    \\ & \label{Sdef4Case2}
    \qquad \text{for $z$ in the lower part of the lens around $\Delta_1$}.
    \end{align}

Then by straightforward calculations it can again be checked
that all conditions 1.--4. of the RH problem for $P$ are satisfied.

In order to unify notation for $\alpha \geq 0$ and $-1 < \alpha <0$, we take as the parametrix in
$B_\delta $ in the Case 2 the matrix valued function $P=QS$, where $S=I$ if $\alpha \geq 0$, and
$S$ is given by \eqref{Sdef1Case2}--\eqref{Sdef4Case2}, if $-1 < \alpha <0$.

\section{Final transformation}\label{section9}

We denote generically by $B_{\delta}$ the small disks around the branch points $0$, $p$ and $q$,
and by $P$ the local parametrices built in $B_{\delta}$. We define the matrix valued function $R$ as
\begin{equation} \label{Rdef}
R(z) = \begin{cases}
    T(z) P^{-1}(z), & \text{in the  neighborhoods $B_{\delta}$,   }   \\
    T(z) N_{\alpha}^{-1}(z), & \text{elsewhere.}
    \end{cases}
\end{equation}
Then $R$ is defined and analytic outside the real line, the lips $\Delta_{1}^{\pm}$
and $\Delta_{2}^{\pm}$ of the lenses and the circles around the three branch points.
If $\alpha \geq 0$, the jump matrices of $T$ and $N_{\alpha}$ coincide on $\Delta_1$ and $\Delta_2$
and the jump matrices of $T$ and $P$ coincide inside the three disks with the
exception of the interval $(0,\delta)$ in Case 1, and $(-\delta, 0)$ in Case 2.
It follows that $R$ has an analytic continuation to the complex plane
minus the
contours shown in Figure \ref{fig:FinalContours}.

\begin{figure}[htb]
\centering \begin{overpic}[scale=1.4]%
{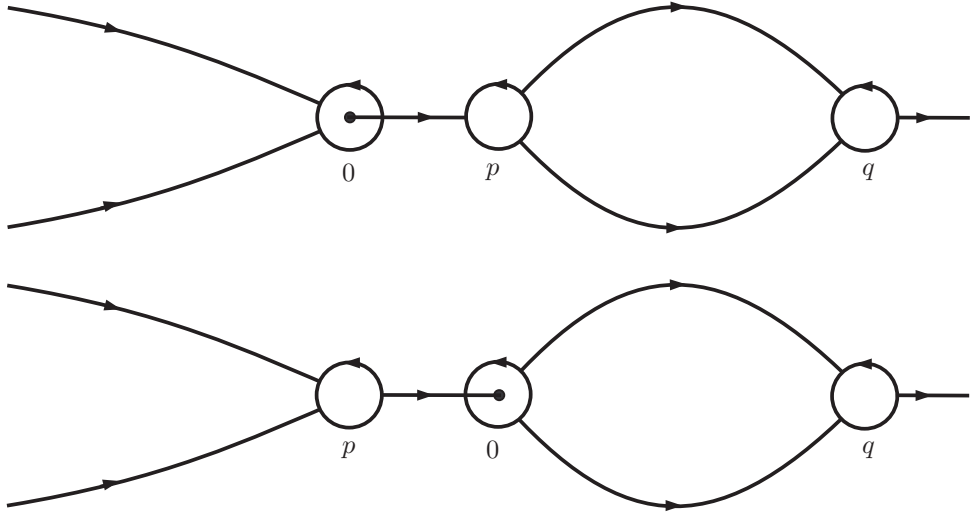}%
\put(35,39){ \small $ 0 $}
\put(49,39.5){ \small $ p $}
\put(85.5,39.5){ \small $ q $}
\put(35,12.5){ \small $ p $}
\put(49,12){ \small $ 0 $}
\put(85.5,12.5){ \small $ q $}
\end{overpic}
\caption{Jump contours for the RH problem for $R$, when $\alpha \geq 0$: Cases 1 (top) and 2 (bottom).}
\label{fig:FinalContours}
\end{figure}

We find that $R$ satisfies the following RH problem, that we describe explicitly only in the Case 1
(Case 2 is similar):
\begin{enumerate}
\item $R$ is analytic outside of the  contours in Figure \ref{fig:FinalContours}.
\item $R$ has a
jump $R_+(z)=R_-(z)\, j_{R} (z)$ on each of the oriented contours in Figure
\ref{fig:FinalContours}, with jump matrix
{\allowdisplaybreaks
\begin{align}
    j_{R} (z) & = N_{\alpha}(z)\, j_T(z) N_{\alpha}^{-1}(z),
    \quad  z \in \Delta_1^{\pm}  \cup \Delta_2^{\pm} \cup (\delta,p-\delta) \cup (q+\delta,\infty), \label{jumpR1}\\
    j_{R}(z) &= N_{\alpha}(z) P^{-1}(z),
    \quad  z \in \partial B_{\delta}, \label{jumpR2} \\
    j_R(z) & = P(z) \, j_T(z) P^{-1}(z),
    \quad z \in (0,\delta).\label{jumpR3}
    \end{align}}
\item
$R(z) = I + O(1/z)$ as $z \to \infty$.
\end{enumerate}

\begin{figure}[htb]
\centering \begin{overpic}[scale=1.4]%
{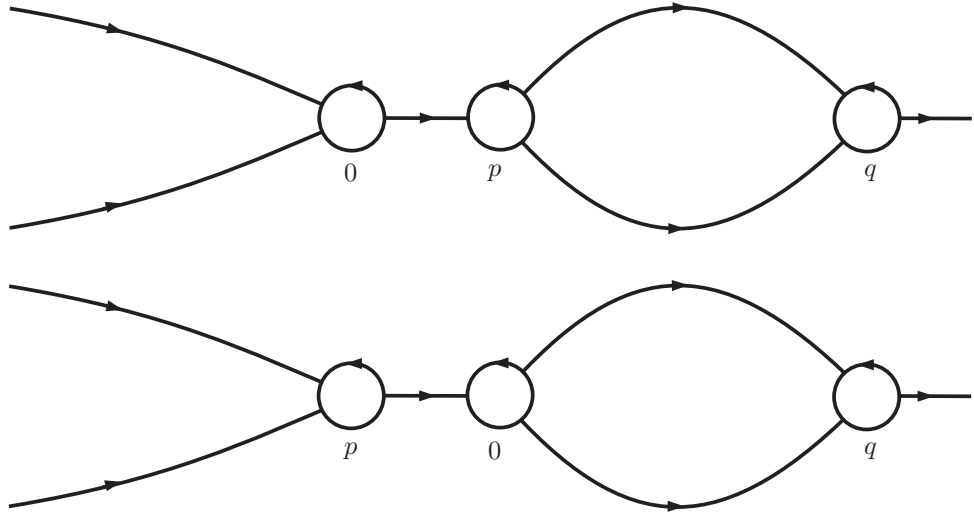}%
\put(35,39){ \small $ 0 $}
\put(49,39.5){ \small $ p $}
\put(85.5,39.5){ \small $ q $}
\put(35,12.5){ \small $ p $}
\put(49,12){ \small $ 0 $}
\put(85.5,12.5){ \small $ q $}
\end{overpic}
\caption{Jump contours for the RH problem for $R$, when $-1<\alpha < 0$: Cases 1 (top) and 2 (bottom).}
\label{fig:FinalContours2}
\end{figure}

Note that it is only after this final transformation that the RH problem
is normalized at infinity. Item 3. follows from \eqref{Tinfty}
and \eqref{eq:Nasymptotics} and the definition \eqref{Rdef} of $R$.

If $-1<\alpha <0$, the situation is even simpler, since now $R$ has an analytic continuation to the
complex plane minus the contours shown in Figure \ref{fig:FinalContours2}, so that only
jumps \eqref{jumpR1}--\eqref{jumpR2} remain. By \eqref{Q01} and \eqref{Q01Case2},
$R(z)$ is at most $\mathcal{O}\left(|z|^\alpha \right)$ as $z\to 0$, so that the singularity
at $0$ is removable.

From the matching conditions for the local parametrices it follows that
\[ j_R(z) = I + \mathcal{O}(1/n) \qquad \textrm{ as $n \to \infty$
    uniformly for $z$ on the boundary of the disks.} \]
If $\alpha \geq 0$, for $x$ in the interval $(0,\delta)$ (in Case 1)
or $(-\delta,0)$ (in Case 2), we have for some $c > 0$,
\[ j_R(x) = I + \mathcal{O}(x^{\alpha} e^{-cn}). \]
On the remaining contours we have for some $c > 0$,
\[ j_R(z) = I + \mathcal{O}(e^{-cn|z|}) \qquad \textrm{ as $n \to \infty$} \,.\]

We can use standard arguments (see e.g.\ \cite{BK2})
to conclude that
\begin{equation} \label{RnearI}
R(z)= I + \mathcal{O} \left( \frac{1}{n (|z|+1)}\right)\,,
    \qquad n \to \infty\,,
\end{equation}
uniformly for $z$ in the complex plane outside of these contours.
Then by Cauchy's theorem also
\begin{equation} \label{Rprimenear0}
R'(z)= \mathcal{O} \left( \frac{1}{n (|z|+1)}\right)\,, \qquad n \to \infty\,.
\end{equation}
Thus, we obtain the following estimate which will be useful in the next section
\begin{equation}
\label{RinvR}
R^{-1}(y) R(x) = I + R^{-1}(y)\, (R(x)-R(y))= I + \mathcal{O} \left( \frac{x-y}{n}\right)\,.
\end{equation}

%
%

\section{Proofs of the theorems}\label{section10}

The proofs of Theorems \ref{theo:domain}--\ref{theo:local3}
are based on the asymptotic analysis of the kernel $K_n(x,y)$.
If we use \eqref{kernel} and follow the steps of the RH steepest descent analysis, we
find that for $x, y > 0$ and $x, y \in \Delta_1$,
\begin{align}
    K_n(x,y) & = \frac{1}{2\pi i(x-y)}
    \begin{pmatrix} 0 & w_1(y) & w_2(y) \end{pmatrix}
    Y_+^{-1}(y) Y_+(x) \begin{pmatrix} 1 \\ 0 \\ 0 \end{pmatrix} \nonumber \\
    & = \frac{1}{2\pi i(x-y)}
    \begin{pmatrix} 0 & y^{\alpha} e^{-n \frac{y}{t(1-t)}} & 0 \end{pmatrix}
    X_+^{-1}(y) X_+(x) \begin{pmatrix} 1 \\ 0 \\ 0 \end{pmatrix}  \nonumber \\
    & = \frac{1}{2\pi i(x-y)}
    \begin{pmatrix} 0 & y^{\alpha} e^{-n \lambda_{2,+}(y)} & 0 \end{pmatrix}
    U_+^{-1}(y) U_+(x) \begin{pmatrix} e^{n \lambda_{1,+}(x)} \\ 0 \\ 0 \end{pmatrix}  \nonumber \\
    & = \frac{1}{2\pi i(x-y)}
    \begin{pmatrix} - e^{-n\lambda_{1,+}(y)} & y^{\alpha} e^{-n \lambda_{2,+}(y)} & 0 \end{pmatrix}
    T_+^{-1}(y) T_+(x) \begin{pmatrix} e^{n \lambda_{1,+}(x)} \\ x^{-\alpha} e^{n \lambda_{2,+}(x)} \\ 0 \end{pmatrix}.
    \label{KernelAnalysis}
\end{align}
This will be our basic formula for the kernel.

\begin{proof}[Proof of Theorem \ref{theo:domain}]
We take $x$ and $y$ in the interior of $\Delta_1$, and we may assume that
the circles around the branch points are such that $x$ and $y$ lie outside
of these disks, so that
\[
T(x)= R(x) N_{\alpha}(x)\,, \qquad T(y) = R(y) N_{\alpha}(y)\,.
\]
Thus, by \eqref{RinvR}
\begin{align*}
 T_+^{-1}(y) T_+(x) & = N_{\alpha,+}^{-1}(x) R_+^{-1}(x)  R_+(y) N_{\alpha,+}(y) \\
    & = N_{\alpha,+}^{-1}(x) \left( I + \mathcal{O} \left( \frac{x-y}{n}\right) \right) N_{\alpha,+}(y) \\
    & = I + \mathcal{O} \left( x-y \right) \quad \text{as } y \to x\,,
\end{align*}
and also
\begin{align*}
 \begin{pmatrix} 1 & 0 & 0 \\ 0 & y^{\alpha} & 0 \\ 0 & 0 & 1 \end{pmatrix}
 T_+^{-1}(y) T_+(x) \begin{pmatrix} 1 & 0 & 0 \\ 0 & x^{-\alpha} & 0 \\ 0 & 0 & 1 \end{pmatrix}
 & = I + \mathcal{O} \left( x-y \right) \quad \text{as } y \to x\,.
\end{align*}

Taking into account that on $\Delta_1$ both $\lambda _1$ and $\lambda _2$ are
purely imaginary on $\Delta_1$ and $\lambda _{2+} =\overline{\lambda_{1+}}$ on
$\Delta_1$, we can rewrite \eqref{KernelAnalysis} as
\begin{align}
    K_n(x,y) & =  \frac{1}{2\pi i(x-y)}
    \begin{pmatrix} - e^{-n i \, \Im \lambda_{1,+}(y)} & e^{n i\,  \Im \lambda_{1,+}(y)} & 0 \end{pmatrix}
    \left( I + \mathcal{O} \left( x-y \right) \right) \begin{pmatrix} e^{n i\,  \Im \lambda_{1,+}(x)} \\
        e^{- n i \, \Im \lambda_{1,+}(x)} \\ 0 \end{pmatrix} \nonumber \\
    & = \frac{1}{2\pi i(x-y)} \, \left( e^{ n i\,  \Im ( \lambda_{1,+}(y) - \lambda_{1,+}(x))} -
        e^{-n i \, \Im(\lambda_{1,+}(y) - \lambda_{1,+}(x))}   + \mathcal{O} (x-y) \right) \nonumber \\
    & = \frac{\sin \left( n  \,  \Im ( \lambda_{1,+}(y) - \lambda_{1,+}(x)) \right) }{ \pi (x-y)}
      + \mathcal{O} (1) \,, \quad \text{as } y \to x\,, \label{asymptoticsK}
\end{align}
where $\mathcal{O}(1)$ holds uniformly in $n$. Now we let $y\to x$.
Using \eqref{def-lambda1} and the L'Hopital rule, we get that
\[
    K_n(x,x) = - \frac{n}{\pi}\, \Im \zeta _{1,+}(x) + \mathcal O(1)=
    \frac{n}{\pi}\,\left|  \Im \zeta _{1,+}(x)\right| + \mathcal O(1)\,, \quad n \to \infty\,,
\]
(see e.g.\ \eqref{zeta-origin2}), and
so
\[ \lim_{n \to \infty} \frac{1}{n} K_n(x,x) =
     \frac{1}{\pi} \left| \Im \zeta_{1,+}(x) \right|. \]
If $x\in \R_+\setminus \Delta_1$, then it can be proved analogously that
\[
    \lim_{n\to \infty} \frac{1}{n} K_n(x,x) = 0\,.
\]
This proves that the limiting mean density of paths exists and is
supported on $[p_+,q]$. This proves Theorem \ref{theo:domain}.
\end{proof}

\begin{proof}[Proof of Theorem \ref{theo:local1}]

Let $x^*\in (p_+(t), q(t))$, where $p_+(t)< q(t)$ are the end points of the interval $\Delta_1$,
described in Theorem \ref{theo:domain}.
Then $\rho(x^*) > 0$, where $\rho$ is the density given in \eqref{density}. For
given $x, y \in \mathbb R$, we take
\[ x_n = x^* + \frac{x}{n \rho(x^*)}, \qquad y_n = x^* + \frac{y}{n \rho(x^*)}. \]
Then for $n$ large enough, we have $x_n, y_n \in (p_+(t),q(t))$, so that \eqref{asymptoticsK} holds.
Then by Taylor expansion,
\begin{align*}
    \Im (\lambda_{1,+}(y_n) - \lambda_{1,+}(x_n)) & =  (y_n-x_n) \Im \zeta_{1,+}(x^*) + \mathcal{O}(y_n-x_n)^2 \\
    & = \frac{y-x}{n \rho(x^*)} \cdot (- \pi \rho(x^*)) + \mathcal{O}\left(\frac{1}{n^2}\right) \\
    & = \frac{\pi (x-y)}{n} + \mathcal{O}\left(\frac{1}{n^2}\right),
    \end{align*}
and therefore
\begin{align*}
    \frac{1}{n \rho(x^*)} K_n(x_n,y_n) & =
    \frac{\sin(n \Im (\lambda_{1,+}(y_n) - \lambda_{1,+}(x_n)))}{\pi(x-y)} + \mathcal{O}(1/n) \\
    & = \frac{\sin \pi(x-y)}{\pi(x-y)} + \mathcal{O}(1/n),
\end{align*}
    which proves Theorem \ref{theo:local1}.
\end{proof}

\begin{proof}[Proof of Theorem \ref{theo:local2}]
Take $c = f'(q)$ where $f$ is the conformal map from \eqref{fqdef}.
For $x, y \in \mathbb R$ we put
 $x_n = q + \frac{x}{c n^{2/3}}$ and $y_n = q + \frac{y}{cn^{2/3}}$.
 This implies that
 \[ n^{2/3} f(x_n) \to x, \qquad n^{2/3} f(y_n) \to y. \]
If $x, y < 0$, then we still can apply \eqref{KernelAnalysis}, but now,
for $n$ large enough, $x_n, y_n$ belong to the small disk $B_\delta $ around $q$, so that
\begin{align*}
 T(x_n) & = R(x_n) P(x_n) \\ & = R(x_n) E(x_n)
    \Psi\left(n^{2/3}f (x_n)\right)\diag\left(x_n^{-\alpha /2} e^{\frac{n}{2}(\lambda_2(x_n)-\lambda_1(x_n))},
      x_n^{\alpha /2} e^{-\frac{n}{2}(\lambda_2(x_n)-\lambda_1(x_n))},1\right)\,,
\end{align*}
and similarly for $T(y_n)$.
Therefore,
\begin{align*}
 T_+(x_n)\, \begin{pmatrix} e^{n \lambda_{1,+}(x_n)} \\ x_n^{-\alpha} e^{n \lambda_{2,+}(x_n)} \\ 0 \end{pmatrix}
    & = x_n^{-\alpha /2} e^{\frac{n}{2}(\lambda_{1,+}(x_n)+\lambda_{2,+}(x_n))}  R(x_n) E(x_n)
    \Psi_+\left(n^{2/3}f (x_n)\right) \, \begin{pmatrix} 1 \\ 1 \\ 0 \end{pmatrix}\,,
\end{align*}
and
\begin{multline*}
 \begin{pmatrix} - e^{-n\lambda_{1,+}(y_n)} & y_n^{\alpha} e^{-n \lambda_{2,+}(y_n)} & 0 \end{pmatrix} T_+^{-1}(y_n)
  \\ = y_n^{ \alpha /2} e^{-\frac{n}{2}(\lambda_{1,+}(y_n)+\lambda_{2,+}(y_n))}
  \begin{pmatrix} - 1 & 1 & 0 \end{pmatrix}    \Psi_+^{-1}\left(n^{2/3}f (y_n)\right)
  E^{-1}(y_n) R^{-1}(y_n) \,.
\end{multline*}
As in \cite[Section 9]{BK2}, we can show that
\[ E^{-1}(y_n) R^{-1}(y_n) R(x_n) E(x_n) \to I. \]
Thus,
\begin{align*}
    \lim_{n \to \infty} \frac{1}{cn^{2/3}} K_n(x_n,y_n)
    & = \frac{1}{2\pi i(x-y)} \begin{pmatrix} -1 & 1 & 0 \end{pmatrix}
        \Psi_+^{-1}(x) \Psi_+(y) \begin{pmatrix} 1 \\ 1 \\ 0 \end{pmatrix} \\
    & = \frac{\Ai(x) \Ai'(y) - \Ai'(x) \Ai(y)}{x-y}.
    \end{align*}

Similar calculations give the same result if $x$ and/or $y$ are positive.

The scaling limit near $p$ in case $t < t^*$ follows in a similar way.
\end{proof}

\begin{proof}[Proof of Theorem \ref{theo:local3}]
Now we assume $t > t^*$ so that we are in Case 2.
For  $x$ and $y$ are in the $\delta$-neighborhood $B_{\delta}$ of $0$,
we use the expression \eqref{KernelAnalysis} for $K_n(x,y)$
with $T = RP = RQS$, where $S = I$ in case $\alpha \geq 0$, or $S$
is given by \eqref{Sdef3Case2} in case $-1 < \alpha < 0$.
In both cases it follows that
\[
    \begin{pmatrix} e^{-n\lambda_{1,+}(y)} & y^{\alpha} e^{-n \lambda_{2,+}(y)} \end{pmatrix}   S_+^{-1}(y)
    =
    \begin{pmatrix} e^{-n\lambda_{1,+}(y)} & y^{\alpha} e^{-n \lambda_{2,+}(y)} \end{pmatrix}, \]
\[ S_+(x) \begin{pmatrix} e^{n \lambda_{1,+}(x)} \\ x^{-\alpha} e^{n \lambda_{2,+}(x)} \\ 0 \end{pmatrix}
    = \begin{pmatrix} e^{n \lambda_{1,+}(x)} \\ x^{-\alpha} e^{n \lambda_{2,+}(x)} \\ 0 \end{pmatrix}, \]
so that by \eqref{KernelAnalysis}
\begin{multline}
    K_n(x,y)  = \frac{1}{2\pi i(x-y)}
    \begin{pmatrix} - e^{-n\lambda_{1,+}(y)} & y^{\alpha} e^{-n \lambda_{2,+}(y)} & 0 \end{pmatrix}
    Q_+^{-1}(y) R_+(y)^{-1} \\
    \times R_+(x) Q_+(x) \begin{pmatrix} e^{n \lambda_{1,+}(x)} \\ x^{-\alpha} e^{n \lambda_{2,+}(x)} \\ 0 \end{pmatrix}.
    \label{KernelAnalysis2}
\end{multline}

Let now $x, y> 0$  be arbitrary. Let $c = -f'(0) > 0$ where $f$ is the
conformal map from \eqref{fdefCase2}
and take $x_n = \frac{x}{4cn^2}$, $y_n = \frac{y}{4cn^2}$
so that
\[ n^2 f(x_n) \to -x/4, \qquad n^2 f(y_n) \to -y/4 \]
as $n \to \infty$.
Then for $n$ large enough, we have that $x_n$ and $y_n$ are in the
$\delta$-neighborhood $B_{\delta}$ of $0$, so that we can use
\eqref{KernelAnalysis2} with $x$ and $y$ replaced by $x_n$ and $y_n$.
We then have
\begin{align*} R_+(x_n) Q_+(x_n)
    & = R(x_n) E(x_n) \widehat{\Psi}_+(n^2 f(x_n)) \\
    & \quad \times
        \diag \left(e^{\alpha \pi i/2} x_n^{-\alpha/2} e^{\frac{n}{2}(\lambda_{2,+}(x_n) - \lambda_{1,+}(x_n))},\;
        e^{-\alpha \pi i/2} x_n^{\alpha/2} e^{-\frac{n}{2}(\lambda_{2,+}(x_n) - \lambda_{1,+}(x_n))},\; 1 \right),
        \end{align*}
and similarly for $R_+(y_n) Q_+(y_n)$.
Thus,
\begin{align*}
 R_+(x_n) Q_+(x_n)\, \begin{pmatrix} e^{n \lambda_{1,+}(x_n)} \\ x_n^{-\alpha} e^{n \lambda_{2,+}(x_n)} \\ 0 \end{pmatrix}
    & = x_n^{-\alpha /2} e^{\frac{n}{2}(\lambda_{1,+}(x_n)+\lambda_{2,+}(x_n))}  R(x_n) E(x_n)
    \widehat{\Psi}_+\left(n^2 f (x_n)\right) \,
    \begin{pmatrix} e^{\alpha \pi i/2} \\ e^{-\alpha \pi i/2} \\ 0 \end{pmatrix}\,,
\end{align*}
and
\begin{multline*}
 \begin{pmatrix} - e^{-n\lambda_{1,+}(y_n)} & y_n^{\alpha} e^{-n \lambda_{2,+}(y_n)} & 0 \end{pmatrix}
    Q_+^{-1}(y_n) R_+^{-1}(y_n)
  \\ = y_n^{ \alpha /2} e^{-\frac{n}{2}(\lambda_{1,+}(y_n)+\lambda_{2,+}(y_n))}
  \begin{pmatrix} - e^{-\alpha \pi i/2} & e^{\alpha \pi i/2} & 0 \end{pmatrix}
  \widehat{\Psi}_+^{-1}\left(n^2 f (y_n)\right)
  E^{-1}(y_n) R^{-1}(y_n) \,.
\end{multline*}
Then it may be shown (see \eqref{RinvR} and \cite{BK2}) that
\[ E^{-1}(y_n) R^{-1}(y_n) R(x_n) E(x_n) \to I, \]
and we arrive at
\[ \lim_{n \to \infty} \frac{1}{cn^2} K_n(x_n,y_n)
    = \frac{1}{2\pi i(x-y)} \left(\frac{y}{x}\right)^{\alpha/2}
  \begin{pmatrix} - e^{-\alpha \pi i/2} & e^{\alpha \pi i/2} & 0 \end{pmatrix}
  \widehat{\Psi}_+^{-1}(y/4)
  \widehat{\Psi}_+(x/4)
  \begin{pmatrix} e^{\alpha \pi i/2} \\ e^{-\alpha \pi i/2} \\ 0 \end{pmatrix}.
\]
To evaluate this further, we first note that by definition of $\widehat{\Psi}$,
\begin{align*}
  \begin{pmatrix} - e^{-\alpha \pi i/2} & e^{\alpha \pi i/2} & 0 \end{pmatrix} &
  \widehat{\Psi}_+^{-1}(y/4)
  \widehat{\Psi}_+(x/4)
  \begin{pmatrix} e^{\alpha \pi i/2} \\ e^{-\alpha \pi i/2} \\ 0 \end{pmatrix} \\
  = &
  \begin{pmatrix} - e^{-\alpha \pi i/2} & e^{\alpha \pi i/2} \end{pmatrix}
  \sigma_3 \Psi_-^{-1}(-y/4) \Psi_-(-x/4) \sigma_3
  \begin{pmatrix} e^{\alpha \pi i/2} \\ e^{-\alpha \pi i/2}  \end{pmatrix} \\
  = &
  \begin{pmatrix} e^{-\alpha \pi i/2} & e^{\alpha \pi i/2} \end{pmatrix}
   \Psi_-^{-1}(-y/4) \Psi_-(-x/4)
  \begin{pmatrix} - e^{\alpha \pi i/2} \\ e^{-\alpha \pi i/2}  \end{pmatrix} \\
  = &
  \begin{pmatrix} 1 & 1 \end{pmatrix}
  \begin{pmatrix} \frac{1}{2} H_{\alpha}^{(2)}(\sqrt{y}) & -\frac{1}{2} H_{\alpha}^{(1)}(\sqrt{y}) \\
  \frac{1}{2} \pi i \sqrt{y} \left( H_{\alpha}^{(2)}\right)'(\sqrt{y}) &
  - \frac{1}{2} \pi i \sqrt{y} \left(H_{\alpha}^{(1)}\right)'(\sqrt{y})
  \end{pmatrix}^{-1} \\
  & \times
  \begin{pmatrix} \frac{1}{2} H_{\alpha}^{(2)}(\sqrt{x}) & -\frac{1}{2} H_{\alpha}^{(1)}(\sqrt{x}) \\
  \frac{1}{2} \pi i \sqrt{x} \left( H_{\alpha}^{(2)}\right)'(\sqrt{x}) &
  -\frac{1}{2} \pi i \sqrt{x} \left(H_{\alpha}^{(1)}\right)'(\sqrt{x})
  \end{pmatrix}
  \begin{pmatrix} -1 \\ 1 \end{pmatrix}.
  \end{align*}
where for the last line we used the definition of $\Psi(\zeta)$ in terms of the
Hankel functions that is valid for $-\pi < \arg \zeta < - 2\pi/3$.
Since
\[ \frac{1}{2} \left( H_{\alpha}^{(1)} + H_{\alpha}^{(2)}\right) = J_{\alpha} \]
and since the above matrices with the Hankel functions have determinant one, it
follows that the above expression is equal to
\begin{multline*}
  \begin{pmatrix} - \pi i \sqrt{y} J_{\alpha}'(\sqrt{y}) & J_{\alpha}(\sqrt{y}) \end{pmatrix}
  \begin{pmatrix} - J_{\alpha}( \sqrt{x})  \\
  - \pi i \sqrt{x} J_{\alpha}'(\sqrt{x})
  \end{pmatrix} \\
  = \pi i \left( J_{\alpha}(\sqrt{x}) \sqrt{y} J_{\alpha}'(\sqrt{y}) -
    \sqrt{x} J_{\alpha}'(\sqrt{x}) J_{\alpha}(\sqrt{y}) \right).
    \end{multline*}
Using this in the expression for the scaling limit we obtain the theorem.
\end{proof}

\section{Appendix: approach via equilibrium measures} \label{appendix}

In the appendix we indicate an approach via equilibrium measures.
Our starting point is the RH problem for $X$, see Proposition \ref{prop:RHforX}. Instead of the
$\lambda$-functions that come from the Riemann surface we use the
so-called $g$-functions to make the second transformation of the
RH problem.

As an intermediate step we first define
\begin{equation} \label{Utildedef}
    \widetilde{U}(z) = X(z) \diag (1, e^{2n \sqrt{az}/t}, e^{-2n \sqrt{az}/t})
\end{equation}
with the usual principal branch of the square root function. Then
$\widetilde{U}$ satisfies the following RH problem.

\begin{enumerate}
\item $\widetilde U(z)$ is analytic in $\mathbb{C} \setminus (\R \cup
\Delta_2^{\pm})$.
\item $\widetilde U(z)$ possesses continuous
boundary values on $\R \cup \Delta_2^{\pm}$ denoted by $\widetilde
U_+$ and $\widetilde U_-$, and
\begin{equation}
\label{Utildejump1} \widetilde U_+(x)= \widetilde U_-(x)
    \left( I +
    x^\alpha e^{-n \left(\frac{x}{t(1-t)} - \frac{2 \sqrt{ax}}{t}\right)} E_{12} \right)
\quad x\in\R_+,
\end{equation}
\begin{equation}  \label{Utildejump2}
\widetilde U_+(x) = \widetilde U_-(x)
    \begin{pmatrix}
    1 & 0 & 0\\
        0 & 0                                & -|x|^{-\alpha} \\
        0 & |x|^{\alpha}                     & 0
    \end{pmatrix}, \quad x\in (-\infty,p_-),
\end{equation}
\begin{equation}  \label{Utildejump3}
\widetilde U_+(x)= \widetilde U_-(x)
    \begin{pmatrix}
    1 & 0 & 0\\
        0 & e^{4i n |ax|^{1/2}/t}            & 0 \\
        0 & |x|^{\alpha}                     & e^{-4i n |ax|^{1/2}/t}
    \end{pmatrix}, \quad x \in (p_-,0),
\end{equation}
\begin{equation} \label{Utildejump4}
    \widetilde U_+(z) = \widetilde U_-(z) \left( I + e^{\pm \alpha \pi i} z^{-\alpha} e^{-4 n (az)^{1/2}/t}  E_{23} \right)
\quad z \in \Delta_2^{\pm}.
\end{equation}
\item $\widetilde U(z)$ has the following behavior near infinity:
\begin{multline} \label{Utildeasymptotics}
   \widetilde U(z) =
   \left(I + \mathcal{O}\left(\frac{1}{z}\right) \right)
    \begin{pmatrix}
    1 & 0 & 0\\
    0 & z^{1/4} & 0 \\
    0 & 0 &  z^{-1/4}
    \end{pmatrix}
     \begin{pmatrix}
    1 & 0 & 0 \\
    0 & \frac{1}{\sqrt{2}} & \frac{1}{\sqrt{2}} i \\
    0 & \frac{1}{\sqrt{2}} i & \frac{1}{\sqrt{2}}
    \end{pmatrix}
    \begin{pmatrix}
    1 & 0 & 0 \\
    0 & z^{\alpha/2} & 0 \\
    0 & 0 & z^{-\alpha/2}
    \end{pmatrix} \\
\begin{pmatrix}
            z^{n}  & 0 & 0 \\
            0 & z^{-n/2}  & 0 \\
            0 & 0 & z^{-n/2}
            \end{pmatrix}, \quad z
        \to \infty, \quad z\in \mathbb{C} \setminus (\R \cup \Delta_2^{\pm}),
\end{multline}
\item $\widetilde U(z)$ has the same behavior as $X(z)$ at the origin,
see \eqref{Xedge2}.
\end{enumerate}

Now we consider the following variational problem for two measures
$\mu_1$ and $\mu_2$. Minimize
\begin{multline} \iint \log \frac{1}{|x-y|} d\mu_1(x) d\mu_1(y)
    - \iint \log \frac{1}{|x-y|} d\mu_1(x) d\mu_2(y) \\
    + \iint \log \frac{1}{|x-y|} d\mu_2(x) d\mu_2(y)
    + \int \left(\frac{x}{t(1-t)} - \frac{2\sqrt{ax}}{t}\right) d\mu_1(x)
\end{multline}
over all pairs $(\mu_1, \mu_2)$ such that
\begin{equation} \label{mu1mu2normalization}
\begin{aligned}
    \supp(\mu_1) & \subset [0, \infty), \qquad \int d\mu_1 = 1, \\
    \supp(\mu_2) & \subset (-\infty,0], \qquad \int d\mu_2 = 1/2,
    \end{aligned}
    \end{equation}
and
\begin{equation} \label{mu2constraint}
    \mu_2 \leq \sigma,
\end{equation}
where $\sigma$ is the (unbounded) measure on $(-\infty,0]$ with
density
\begin{equation} \label{sigmadef}
\frac{d\sigma}{dx} = \frac{\sqrt{a}}{\pi t} |x|^{-1/2},
    \quad x \in (-\infty,0].
    \end{equation}

It is possible to show that there is a unique minimizing pair
$(\mu_1, \mu_2)$. The measures are absolutely continuous with
respect to Lebesgue measure and their densities are related to the
functions $\zeta_1$ and $\zeta_3$ coming from the Riemann surface
as follows
\begin{equation}
    \begin{aligned}
    \frac{d \mu_1}{dx} & = - \frac{1}{2\pi i} \left( \zeta_{1+} - \zeta_{1-}\right), \\
    \frac{d \mu_2}{dx} & = \frac{d \sigma}{dx}
        + \frac{1}{2\pi i} \left( \zeta_{3+}(x) - \zeta_{3-}(x) \right).
    \end{aligned}
\end{equation}
Thus
\[ \supp(\mu_1) = \Delta_1, \qquad \supp(\mu_2) = (-\infty,0],
    \qquad \supp(\sigma-\mu_2) = \Delta_2, \]
and the constraint \eqref{mu2constraint} on $\mu_2$ is active only
in Case 2.

The following variational equalities and inequalities hold for
certain Lagrange multipliers $l_1$ and $l_2$:
\begin{equation} \label{variational1}
     2\int \log |x-s| d\mu_1(s) - \int \log |x-s| d\mu_2(s)
    - \frac{x}{t(1-t)} + \frac{2\sqrt{ax}}{t}
    \left\{ \begin{array}{cl} = l_1, & x \in \Delta_1, \\
    < l_1, & x \in \R_+ \setminus \Delta_1,
    \end{array} \right.
\end{equation}
\begin{equation} \label{variational2}
    2\int \log |x-y| d\mu_2(s) - \int \log |x-y| d\mu_1(s)
    \left\{ \begin{array}{cl} = l_2, & x \in \Delta_2, \\
    > l_2, & x \in \R_- \setminus \Delta_2.
    \end{array} \right.
\end{equation}
This is a vector equilibrium for the pair of measures $ \mu _1$ and $\mu_2$, supported on  $\R_+$ and
$\R_-$, respectively, with the matrix of interaction
$$
\begin{pmatrix}
2 & -1 \\ -1 & 2
\end{pmatrix},
$$
characteristic of a Nikishin system \cite{BL, NS} (see \cite{Apt} for a survey),
but with two additional features:
\begin{enumerate}
\item[(i)] there is an external field
$$
\varphi(x) =  \frac{x}{t(1-t)} - \frac{2\sqrt{ax}}{t}
$$
acting on $\R_+$, motivated by the varying character of the orthogonality weights in \eqref{weights2};
\item[(ii)] there is an upper constraint \eqref{mu2constraint} originated in the fact that
$w_2/w_1$ is the Cauchy transform of a discrete measure on $\R_-$, see \eqref{discreteCauchy}.
The upper constraint \eqref{sigmadef} is equal to the limiting distribution of the points
\eqref{Besselzeros} that are related to the positive zeros of the Bessel function $J_\alpha $.
\end{enumerate}

We introduce the $g$-functions
\begin{equation} \label{defgj}
    g_j(z) = \int \log(z-s) d\mu_j(s), \qquad j=1,2,
\end{equation}
and define the transformation
\begin{equation} \label{defUwithg}
    U(z) = C_n \diag \left(e^{-nl_1}, 1, e^{nl_2}\right) \widetilde U(z)
    \diag \left( e^{-n (g_1(z)- l_1)}, e^{n(g_1(z)- g_2(z)}, e^{n (g_2(z)-l_2)}\right)
\end{equation}
where $l_1$ and $l_2$ are the constants from \eqref{variational1}
and \eqref{variational2} and $C_n$ is a constant matrix (see the
first matrix in the right-hand side of \eqref{Udef}). Then $U$
satisfies a RH problem.
\begin{enumerate}
\item $U(z)$ is analytic in $\mathbb{C} \setminus (\R \cup
\Delta_2^{\pm})$. \item $U(z)$ possesses continuous boundary
values on $\R \cup \Delta_2^{\pm}$ denoted by $U_+$ and $U_-$, and
\begin{equation}
\label{Ualtjump1} U_+(x)= U_-(x)
    \begin{pmatrix}
    e^{-n(g_{1+}(x)-g_{1-}(x))} &
    x^\alpha e^{n \left(g_{1+}(x) + g_{1-}(x) - g_2(x)-\frac{x}{t(1-t)} + \frac{2 \sqrt{ax}}{t} - l_1\right)}& 0  \\
            0                    & e^{n(g_{1+}(x)-g_{1-}(x))}& 0 \\
            0                    & 0    & 1
    \end{pmatrix}, \quad x\in\R_+,
\end{equation}
\begin{equation}  \label{Ualtjump2}
U_+(x) = U_-(x)
    \begin{pmatrix}
    1 & 0 & 0\\
        0 & 0                                & -|x|^{-\alpha} \\
        0 & |x|^{\alpha}                     & 0
    \end{pmatrix}, \quad x\in (-\infty,p_-),
\end{equation}
\begin{equation}  \label{Ualtjump3}
U_+(x)= U_-(x)
    \begin{pmatrix}
    1 & 0 & 0\\
        0 & e^{4i n |ax|^{1/2}/t} e^{- n(g_{2+}(x)-g_{2-}(x))}           & 0 \\
        0 & |x|^{\alpha} e^{n(g_{1+}(x) - g_{2+}(x) - g_{2-}(x) + l_2)}  &
        e^{-4i n |ax|^{1/2}/t} e^{ n(g_{2+}(x)-g_{2-}(x))}
    \end{pmatrix}, \quad x \in (p_-,0),
\end{equation}
\begin{equation} \label{Ualtjump4}
    U_+(z) = U_-(z) \left( I + e^{\pm \alpha \pi i} z^{-\alpha} e^{-4 n (az)^{1/2}/t} e^{n(2g_2(z) - g_1(z)-l_2)}  E_{23} \right)
\quad z \in \Delta_2^{\pm}.
\end{equation}
\item $U(z)$ has the following behavior as $z
        \to \infty$, $z\in \mathbb{C} \setminus (\R \cup \Delta_2^{\pm})$:
\begin{multline} \label{Ualtasymptotics}
    U(z) =
   \left(I + \mathcal{O}\left(\frac{1}{z}\right) \right)
    \begin{pmatrix}
    1 & 0 & 0\\
    0 & z^{1/4} & 0 \\
    0 & 0 &  z^{-1/4}
    \end{pmatrix}
     \begin{pmatrix}
    1 & 0 & 0 \\
    0 & \frac{1}{\sqrt{2}} & \frac{1}{\sqrt{2}} i \\
    0 & \frac{1}{\sqrt{2}} i & \frac{1}{\sqrt{2}}
    \end{pmatrix}
    \begin{pmatrix}
    1 & 0 & 0 \\
    0 & z^{\alpha/2} & 0 \\
    0 & 0 & z^{-\alpha/2}
    \end{pmatrix}.
\end{multline}
\item $U(z)$ has the same behavior as $X(z)$ at the origin, see
\eqref{Xedge2}.
\end{enumerate}

Due to the equilibrium conditions we have that the jump
\eqref{Ualtjump1} simplifies on the interval $\Delta_1$ to
\begin{equation}  \label{Ualtjump1a}
U_+(x)= U_-(x)
    \begin{pmatrix}
    e^{-n(g_{1+}(x)-g_{1-}(x))} &
    x^\alpha & 0  \\
            0                    & e^{n(g_{1+}(x)-g_{1-}(x))}& 0 \\
            0                    & 0    & 1
    \end{pmatrix}, \quad x\in\Delta_1.
\end{equation}
A calculation that uses the fact that $\mu_2 = \sigma$ on $(p_-,0)$
shows that the diagonal entries of the jump matrix
\eqref{Ualtjump3} on $(p_-,0)$ are equal to $1$, so that
\begin{equation} \label{Ualtjump3a}
U_+(x)= U_-(x) \left( I + |x|^{\alpha} e^{n(g_{1+}(x) - g_{2+}(x) - g_{2-}(x) + l_2)} E_{32} \right)
\quad x \in (p_-,0),
\end{equation}
with an off-diagonal entry that is tending to $0$ as $n \to
\infty$. Of course the jump \eqref{Ualtjump3a} is only relevant in
Case 2.

We can then go on by opening a lens around $\Delta_1$ as discussed
in the main part of the text.

We conclude this appendix by giving the relation between the
$g$-functions and the $\lambda$-functions coming from the Riemann
surface. We have
\begin{align}
    g_1(z) & = \lambda_1(z) - \ell_1, \\
    g_1(z) - g_2(z) & = - \lambda_2(z) + \frac{z}{t(1-t)} - \frac{2\sqrt{az}}{t} + \ell_2, \\
    g_2(z) & = -\lambda_3(z) + \frac{z}{t(1-t)} + \frac{2 \sqrt{az}}{t} + \ell_3.
\end{align}
with constants $\ell_1$, $\ell_2$, and $\ell_3$ appearing in
\eqref{lambda1}--\eqref{lambda3}.
These relations and \eqref{lambda1}--\eqref{lambda3} show that
\begin{align}
g_1(z) & = \log z  -\frac{(1-t)(t+a(1-t))}{z}+
\mathcal{O} \left(\frac{1}{z^2} \right), \label{g1asymptotics} \\
g_2(z) & =  \frac{1}{2}\log z + \frac{t+4a(1-t)}{4\sqrt{az}}  -
\frac{(1-t)(t+a(1-t))}{2z} +\mathcal{O} \left(\frac{1}{z^{3/2}}
\right),\label{g2asymptotics}
\end{align}
as $z \to \infty$.

\section*{Acknowledgements}

ABJK is supported by FWO-Flanders project G.0455.04,
by K.U. Leuven research grant OT/04/21, by the Belgian Interuniversity
Attraction Pole P06/02, and by the  European Science Foundation Program MISGAM.

AMF is partially supported by Junta de Andaluc\'{\i}a, grants FQM-229, FQM-481, and P06-FQM-01738.

Additionally, ABJK and AMF are partially supported by the Ministry of Education and
Science of Spain, project code MTM2005-08648-C02-01.



\obeylines
\texttt{
A. B. J. Kuijlaars (arno.kuijlaars@wis.kuleuven.be)
Department of Mathematics
Katholieke Universiteit Leuven
Celestijnenlaan 200B
3001 Leuven, BELGIUM
\medskip
A. Mart\'{\i}nez-Finkelshtein (andrei@ual.es)
Department of Statistics and Applied Mathematics
University of Almer\'{\i}a, SPAIN, and
Instituto Carlos I de F\'{\i}sica Te\'{o}rica y Computacional
Granada University, SPAIN
\medskip
F. Wielonsky (Franck.Wielonsky@math.univ-lille1.fr)
Laboratoire de Math\'ematiques P. Painlev\'e
UMR CNRS 8524 - Bat.M2
Universit\'e des Sciences et Technologies Lille
F-59655 Villeneuve d'Ascq Cedex, FRANCE
}

\end{document}